%% file: Asymptotic3.tex
\documentclass[11pt,a4paper]{amsart}
\usepackage{a4}
\usepackage[utf8]{inputenc}
\usepackage{enumerate}
\usepackage{stmaryrd}
\usepackage{amsfonts,amssymb,amsmath}
\usepackage[color=green!30]{todonotes}

\usepackage{xcolor}
\usepackage[colorlinks,
    linkcolor={red!50!black},
    citecolor={blue!50!black},
    urlcolor={blue!80!black}]{hyperref}

\usepackage{tikz}
\usetikzlibrary{matrix}
\usepackage[all]{xy}

\def\FF{{\mathbb F}}
\def\CC{{\mathbb C}}

\def\ZZ{{\mathbb Z}}

\def\RR{{\mathbb R}}
\def\HH{{\mathbb H}}

\def\alphaa{{\bar\alpha}}

\def\sll{\mathfrak{sl}}
\def\sl2{\sll_2\CC}

\def\Ad{\operatorname{Ad}}

\newcommand{\Sym}{\operatorname{Sym}}

\newcommand{\mS}{\mathbb{S}}
\newcommand{\mC}{\mathbb{C}}
\newcommand{\mR}{\mathbb{R}}

\newcommand{\mZ}{\mathbb{Z}}
\newcommand{\mH}{\mathbb{H}}

\newcommand{\bm}{\begin{pmatrix}}
\newcommand{\ema}{\end{pmatrix}}
\newcommand{\bsm}{\left(\begin{smallmatrix}}
\newcommand{\esm}{\end{smallmatrix}\right)}
\newcommand{\tor}{\operatorname{tor}}

\newcommand{\Hom}{\operatorname{Hom}}

\newcommand{\Tr}{\operatorname{Tr}}

\newcommand{\SL}{\operatorname{SL}}
\newcommand{\GL}{\operatorname{GL}}
\newcommand{\SU}{\operatorname{SU}}

\newcommand{\Vol}{\operatorname{Vol}}
\newcommand{\benu}{\begin{enumerate}}
\newcommand{\eenu}{\end{enumerate}}

\def\op{\operatorname}

\newtheorem{theorem}{Theorem}[section]

\newtheorem{corollary}[theorem]{Corollary}
\newtheorem{lemma}[theorem]{Lemma}
\newtheorem{proposition}[theorem]{Proposition}

\theoremstyle{definition}
\newtheorem{definition}[theorem]{Definition}

\newtheorem{example}[theorem]{Example}
\newtheorem{Assumption}[theorem]{Assumption}

\newtheorem{remark}[theorem]{Remark}

\author{Leo Benard}
\address{Mathematisches Institut, Georg-August Universit\"at, Bunsenstrasse 3-5, 37073 G\"ottingen, Germany}
\email{leo.benard@mathematik.uni-goettingen.de}

\author{Jerome Dubois}
\address{Universit\'e Clermont Auvergne, CNRS, Laboratoire de Mathématiques Blaise Pascal, F-63000 Clermont-Ferrand, France}
\email{jerome.dubois@uca.fr.}

\author{Michael Heusener}
\address{Universit\'e Clermont Auvergne, CNRS, Laboratoire de Mathématiques Blaise Pascal, F-63000 Clermont-Ferrand, France}
\email{michael.heusener@uca.fr}

\author{Joan Porti} 
\address{ Departament de Matem\`atiques, Universitat Aut\`onoma de Barcelona, 
08193 Cerdanyola del Vall\`es, Spain, and 
Barcelona Graduate School of Mathematics (BGSMath) }
\email{porti@mat.uab.cat}


\date{\today}

\begin{document}
\title[Asymptotics of twisted polynomials ]{Asymptotics  of twisted Alexander polynomials and hyperbolic volume}
\begin{abstract}
{For a hyperbolic knot and a natural number $n$, we consider the Alexander polynomial twisted by
the $n$-th symmetric power of a lift of the holonomy. 
We establish the asymptotic behavior of these twisted Alexander polynomials evaluated at unit complex numbers, yielding the volume of the knot exterior.
More generally, we prove the asymptotic behavior for cusped hyperbolic manifolds of finite volume. The proof 
relies on results of M\"uller, and Menal-Ferrer and the last author. Using the uniformity of the convergence, we also deduce a similar asymptotic result for the Mahler measures of those polynomials.}
\end{abstract}

\maketitle

\section{Introduction}
\label{sec:Introduction}

Twisted Alexander polynomials of knots have been defined by Lin~\cite{Lin01} and Wada~\cite{Wada}.
Kitano~\cite{Kit96} showed that they are Reidemeister torsions, generalizing Milnor's theorem on the (untwisted) 
Alexander polynomial~\cite{MilnorAlexander}.
Here we take the Reidemeister torsion approach to define the twisted Alexander polynomial for oriented, cusped, hyperbolic three-manifolds of finite volume.

An orientable hyperbolic three-manifold has a natural representation of its fundamental group into $\mathrm{PSL}_2(\CC)$, the hyperbolic holonomy  
that is unique up to conjugation.
The holonomy representation lifts to $\mathrm{SL}_2(\CC)$, and a lift is unique up to multiplication
with a representation into the center of $\mathrm{SL}_2(\CC)$ (see \cite{Cul86}).

The corresponding twisted Alexander polynomial has been considered, among others, by Dunfield, Friedl and Jackson in~\cite{DFJ}. 
Here we compose the lift of the holonomy representation with the
irreducible representation of  $\mathrm{SL}_2(\CC)$ in
 $\mathrm{SL}_n(\CC)$, the $(n-1)$-th symmetric power, and study its asymptotic behavior. 

Before considering non-compact, orientable, hyperbolic three-manifolds of finite volume in general,
we discuss first the case of a hyperbolic knot complement $S^3\setminus K$.
Let $\rho_n\colon\pi_1(S^3\setminus K)\to \mathrm{SL}_n(\CC)$
be the composition of a lift of the holonomy with the   $(n-1)$-th symmetric power
$\mathrm{SL}_2(\CC)\to\mathrm{SL}_n(\CC)$. 
Let  $\Delta_K^{\rho_n}$ denote the  Alexander polynomial of $K$ twisted by $\rho_n$,
which equals Wada's definition for $n$ even, but it is Wada's polynomial divided by $(t-1)$ when $n$ is odd, so that its evaluation at $t=1$ does not vanish. 
The set of
unit complex numbers is denoted by
$
\mathbb S^1=\{ \zeta\in\CC\mid \vert \zeta\vert =1\}.
$
The following is a particular case of the main result of this paper.

\begin{theorem}
\label{theorem:knots}
For any $\zeta\in\mathbb S^1$,  
$$
 \lim_{n\to\infty}\frac{\log | \Delta_K^{\rho_n}(\zeta)|}{n^ 2}=\frac{1}{4\pi}\operatorname{vol}(S^3\setminus K)
$$
uniformly on $\zeta$.
\end{theorem}

For a knot exterior  there are two lifts $\rho$ of the holonomy, Theorem~\ref{theorem:knots} holds true for both choices of lift.
As it has been shown by Goda for knot exteriors in \cite{Goda}, 
for $\zeta =1$ this theorem is a reformulation of a result on Reidemeister torsions of cusped manifolds 
proved by Menal-Ferrer and the last author in \cite{MP14}, relying on results of M\"uller~\cite{Muller}.

Theorem~\ref{theorem:knots} is a particular case of Theorem~\ref{theorem:main} below.  To extend the definition of twisted Alexander polynomial
to general cusped manifolds (Definition~\ref{Def:twisted}), we need to make some assumptions, that are always satisfied for hyperbolic knot exteriors.
Let $M$ be an orientable, non-compact, connected, finite volume hyperbolic three-manifold. It admits a compactification 
$\overline M$ by adding $l\geq 1$ peripheral tori, one for each end:
$$
\partial \overline M=T^2_1\sqcup\dots \sqcup T^2_l.
$$
Let $$
\alpha\colon\pi_1(M)\twoheadrightarrow \ZZ^r
$$
be an epimorphism.

\begin{Assumption}
\label{Assumption:alpha}
For each peripheral torus $T_i^2$,
$
\alpha(\pi_1(T_i^2))\cong\ZZ
$. 
\end{Assumption}

Assumption~\ref{Assumption:alpha} holds true for the abelianization map of a knot in a homology sphere, or more generally
for the abelianization map of a link in a homology sphere having the property that
the linking number of pairwise different components vanish. Furthermore, for any cusped, oriented, hyperbolic 3-manifold $M$, 
there exists an epimorphism $\alpha\colon \pi_1(M)\to \ZZ$ satisfying Assumption \ref{Assumption:alpha} (compose the abelianization map
with a generic surjection of $H_1(M,\ZZ)$ onto $\ZZ$).
 
Let $l_i\in \pi_1 (T^2_i)$ be a generator of 
$\ker(\alpha\vert_{\pi_1 (T^2_i)})$, we say that $l_i$ is a \emph{longitude for} $\alpha$, and we   use the terminology \textit{$\alpha$-longitude}.
We choose 
$
\rho\colon\pi_1(M)\to\mathrm{SL}_2(\CC)
$
a lift of the hyperbolic holonomy
satisfying the following:

\begin{Assumption}
\label{Assumption:Lift}
The lift of the holonomy  $
\rho\colon\pi_1(M)\to\mathrm{SL}_2(\CC)
$ satisfies 
$$
\operatorname{tr}(\rho(l_i))=-2.
$$
for each  $\alpha$-longitude $l_i$, $i=1,\ldots,l$.
 \end{Assumption}

 For a knot exterior in a homology sphere and the abelianization map,
Assumption~\ref{Assumption:Lift}  is  satisfied for \emph{every} lift of the holonomy. 
More generally, it is also satisfied for \emph{every} lift of the holonomy for a link exterior in a
homology sphere with the property that 
the linking number of pairwise different components
vanish. 
In general, for any $M$ and $\alpha$ satisfying Assumption~\ref{Assumption:alpha}, 
there exists at least one lift of the holonomy satisfying Assumption~\ref{Assumption:Lift}  \cite[Proposition 3.2]{MP14}. 
 In terms of spin structures, this is the condition for a spin structure to extend along the Dehn fillings we will consider in the sequel, see Section \ref{sec:Dehnfillings}.

For  $n \ge 2$, recall that the unique $n$-dimensional irreducible 
holomorphic
representation of $\mathrm{SL}_2(\CC)$ is the $(n-1)$-th symmetric power. We denote it by 
$\Sym^{n-1} \colon \SL_2(\mC) \to \SL_n(\mC)$.
For a lift of the holonomy representation $\rho \colon \pi_1(M) \to \SL_2(\mC)$, 
we denote the composition with the $(n-1)$-th symmetric power by 
 $$\rho_n \colon \pi_1(M) \xrightarrow{\rho}\SL_2(\mC) \xrightarrow{\Sym^{n-1}} \SL_n(\mC).$$

\begin{remark}
This convention follows the notation of \cite{MP14}, but it   differs from  \cite{Muller}, it is shifted by 1.
\end{remark}

To construct the twisted Alexander polynomial, we consider the polynomial representation associated to $\alpha$:
$$
\begin{array}{rcl}
 \alphaa\colon \pi_1(M) & \to & \CC[t^{\pm 1}_1,\ldots, t^{\pm 1}_r]\\
 \gamma & \mapsto & t_1^{\alpha_1(\gamma)}\cdots t_r^{\alpha_r(\gamma)}
\end{array}
$$
where $\alpha=(\alpha_1,\ldots,\alpha_r)$ are the components of $\alpha$.
We define the twisted Alexander polynomial $\Delta^{\alpha,n}_M$ in Definition~\ref{Def:twisted} as the 
inverse of the Reidemeister torsion of the pair $(M, \alphaa \otimes \rho_n)$, after removing some factors $(t_{1}^{\beta_1}\cdots t_{r}^{\beta_r}-1)$ when $n$ is odd 
(one factor for each peripheral torus).
It is a Laurent polynomial with 
variables $t_1^{\pm 1},\ldots,  t_r^{\pm 1}$ defined up to sign and up to multiplicative factors $t_i^{\pm 1}$. 

\begin{remark}
We stress out the fact that, for $\zeta_1,\ldots,\zeta_r\in \mS^1$, only the modulus $\vert  \Delta^{\alpha,n}_M(\zeta_1,\ldots,\zeta_r)\vert $  is well defined.  
\end{remark}

The main result of this paper is:

\begin{theorem}
\label{theorem:main}
Under Assumptions~\ref{Assumption:alpha}  and~\ref{Assumption:Lift}, 
for any $\zeta_1,\ldots,\zeta_r\in  \mathbb S^1$, 
$$
  \lim_{n\to\infty}\frac{\log\vert \Delta^{\alpha,n}_M(\zeta_1,\dots,\zeta_r)\vert}{n^2}=\frac{\operatorname{vol}(M)}{4\pi}
$$
 uniformly on the $\zeta_1,\ldots,\zeta_r$.
\end{theorem}

The \emph{logarithmic Mahler measure} of a Laurent polynomial $ P(t_1,\ldots,t_r)\in\CC[t_1^{\pm 1},\ldots, t_r^{\pm 1}]$ is defined as
\begin{equation*}
 \mathrm{m}(P)=\frac{1}{(2\pi)^r}\int_0^{2\pi}\cdots \int_0^{2\pi} \log\vert P( e^{i\theta_1},\ldots, e^{i\theta_r} ) \vert \, d\theta_1\cdots d\theta_r.
\end{equation*}
With  Theorem~\ref{theorem:main}, as the convergence is uniform on $\zeta_1,\ldots,\zeta_r\in  \mathbb S^1$, we  also prove:

\begin{theorem}
\label{theorem:Mahler}
Under Assumptions~\ref{Assumption:alpha}  and~\ref{Assumption:Lift}, 
$$
  \lim_{n\to\infty}\frac{ \mathrm{m} (\Delta^{\alpha,n}_M)}{n^2}=\frac {\operatorname{vol}(M)}{4\pi} .
$$
\end{theorem}

Assume now that $M$ is fibered over the circle and let $\alpha\colon\pi_1(M)\to\ZZ$ be induced by the fibration $M\to \mS^1$.
We chose a representative $\Delta^{\alpha,n}_M(t)$ so that $\Delta^{\alpha,n}_M(t)\in \CC[t]$ (it is a polynomial) 
and $\Delta^{\alpha,n}_M(0)\neq 0 $. From the fiberness of $M$ follows that $\Delta^{\alpha,n}_M(0)=\pm 1$.  
Using Theorem \ref{theorem:Mahler}, Jensen's formula and some symmetry properties of $\Delta^{\alpha,n}_M$, we deduce the following corollary:

\begin{corollary}
We have 
 $$
 \lim_{n\to\infty}\frac{1}{n^2} \sum_{\lambda\in \mathrm{Spec}(  \Delta^{\alpha,n}_M )} \big\vert \log \vert \lambda   \vert\big\vert    = \frac{1}{2\pi}\operatorname{vol} (M),
 $$
 where $\mathrm{Spec}(  \Delta^{\alpha,n}_M )=\{ \lambda\in\CC\mid   \Delta^{\alpha,n}_M(\lambda)=0 \} $.
 Moreover,  the maximum of the modulus of the roots grows exponentially with $n$.

\end{corollary}
In the previous corollary, we use $\frac{1}{2\pi}$ instead of  $\frac{1}{4\pi}$ because $\Delta^{\alpha,n}_M(\lambda)=0 $ 
iff $\Delta^{\alpha,n}_M(1/\lambda)=0 $, and therefore precisely half of the roots appear in Jensen's formula.
Notice also that $\deg\Delta^{\alpha,n}_M(t) $ is linear on $n$, so the second statement follows directly from the first.

When $M$ is not fibered, it may happen that $ \Delta^{\alpha,n}_M$ is not monic and we must take into account $\Delta^{\alpha,n}_M(0)$ in Jensen's formula.

\begin{remark}
 If $M$ is a \emph{closed}, oriented,  hyperbolic three-manifold, then all our results hold true for any epimorphism
 $\alpha\colon\pi_1(M)\to\ZZ^r$, without requiring any assumption on $\alpha$ and the lift of the holonomy.
\end{remark}

 Given $\zeta_1,\ldots,\zeta_r\in  \mS^1 $, we compose $\alpha\colon\pi_1(M)\twoheadrightarrow \ZZ^r$ with  the homomorphism
 \begin{align*}
   \ZZ^r & \to  \mS^1\\
  (n_1,\ldots, n_r) & \mapsto \zeta_1^{n_1}\cdots  \zeta_r^{n_r} 
  \end{align*}
  and we denote the composition by $\chi\colon \pi_1(M)\to  \mS^1
  $.
  Namely, we evaluate $\alphaa$ at $t_j=\zeta_j$: 
  if
 $\alpha=(\alpha_1,\ldots,\alpha_r)$ are the components of $\alpha$, then
  \begin{align}
\label{eqn:chi}
   \chi \colon &\pi_1(M)  \to  \mS^1\\
 &\gamma  \mapsto \zeta_1^{\alpha_1(\gamma)}\cdots  \zeta_r^{\alpha_r(\gamma)}. \nonumber
  \end{align}

In fact, Theorem~\ref{theorem:main}
is a theorem on Reidemeister torsions, as 
$\vert \Delta^{\alpha,n}_M(\zeta_1,\ldots,\zeta_r)\vert$ is the inverse of the modulus of the Reidemeister torsion
of $M$ twisted by the representation $\chi\otimes\rho_n$ (in some cases perhaps up to some factor independent of $n$ or after the choice of basis in homology),
see Section~\ref{sec:Reidemeister}.

The definition of twisted Alexander polynomial as a Reidemeister torsion requires a  vanishing theorem in cohomology, Theorem~\ref{Theorem:vanishing}.  
Its proof mimics
the classical vanishing theorem on $L^2$-cohomology of
 Matsushima--Murakami, as we explain in Appendix~\ref{Appendix:L2} .
As a direct consequence of this vanishing theorem, we obtain that the twisted Alexander polynomials have no roots on the unit circle: 
\begin{theorem}
\label{Thm:AlexNonZero}
Under Assumptions~\ref{Assumption:alpha}  and~\ref{Assumption:Lift}, 
for any $\zeta_1,\ldots,\zeta_r\in \mS^1$, 
$$\Delta_{M}^{\alpha,n}(\zeta_1,\ldots,\zeta_r)\neq 0.$$
\end{theorem}

We apply this theorem to study the dynamics of a pseudo-Anosov diffeomorphism on the variety of representations. Let $\Sigma$ be a compact orientable surface, possibly with boundary
and with 
negative Euler characteristic. For a pseudo-Anosov diffeomorphism $\phi\colon \Sigma\to\Sigma$,  consider its action on the relative 
variety of (conjugacy classes of) representations
$\phi^*\colon \mathcal{R}(\Sigma, \partial\Sigma, \SL_n(\CC))\to \mathcal{R}(\Sigma, \partial\Sigma, \SL_n(\CC))$.
The mapping torus $M(\phi)$ is a hyperbolic manifold of finite volume 
and its holonomy restricts to a representation 
of $\pi_1(\Sigma)$ in $\SL_2(\CC)$ 
 whose conjugacy class is fixed by $\phi^*$. In particular the conjugacy class of the composition 
 $[\rho_n]=[\operatorname{Sym}^{n-1}\circ\mathrm{hol}\vert_{\pi_1(\Sigma)}]$ in  $\mathcal{R}(\Sigma, \partial\Sigma, \SL_n(\CC))$
 is fixed by $\phi^*$. In Appendix~\ref{Appendix:pseudoAnosov} we prove:
 
 \begin{theorem}
\label{Theorem:AnosovDynamicsIntro}
  The tangent map of $\phi^*$  at $[\rho_n]$ on $\mathcal{R}(\Sigma,\partial \Sigma,\SL_n(\CC))$
has no eigenvalues of norm one. Namely, of $\phi^*$  has hyperbolic dynamics at $[\rho_n]$.
 \end{theorem}

For $n=2$ and $\partial \Sigma=\emptyset$, this was proved by M.~Kapovich in \cite{Kapovich98}.
The relation with the rest of the paper comes from the formula (Proposition~\ref{prop:deltchar})
$$\det\left( (d\phi^*)_{[\rho_n]}-t\operatorname{Id}\right)
 =\prod\limits_{k=1}^{n-1}  \Delta^{\alpha,2 k+1}_{M(\phi)}(t),
$$ where $\alpha\colon\pi_1(M(\phi))\twoheadrightarrow \ZZ$ is induced by the natural fibration of the mapping torus over $\mS^1$ with fiber $\Sigma$. 
 Then the result follows from Theorem~\ref{Thm:AlexNonZero}.

\subsection*{Summary of the proof}

Most of the paper is devoted to prove Theorem~\ref{theorem:main} for $\zeta_1,\ldots,\zeta_r\in e^{i\pi\mathbb Q}$. 
For that purpose, we consider 
sequences of closed manifolds $M_{p/q}$ obtained by Dehn filling, that converge geometrically to $M$ by Thurston's hyperbolic Dehn filling theorem.
The assumption $\zeta_1,\ldots,\zeta_r\in~e^{i\pi\mathbb Q}$ allows us to chose the Dehn fillings $M_{p/q}$ so that the twist $\chi\colon\pi_1(M)\to \mathbb S^1$ in~\eqref{eqn:chi} factors through $\pi_1(M_{p/q})$. 
Then the strategy is to apply M\"uller's theorem \cite{Muller} to the asymptotic behavior of the torsion of closed Dehn fillings $M_{p/q}$. 
Even if in M\"uller's
paper there is no twist, as $\chi$ is unitary we can modify the proof of M\"uller's theorem by considering Ruelle functions twisted by $\chi$.  We prove the main theorem for 
$\zeta_1,\ldots,\zeta_r\in e^{i\pi\mathbb Q}$ by analyzing the behavior of those twisted Ruelle zeta functions and the arguments of M\"uller's proof under limits of Dehn fillings,
as in  \cite{MP14}. To conclude the proof for  arbitrary $\zeta_1,\ldots,\zeta_r\in \mathbb S^1$,
we take an intermediary result (Corollary~\ref{coro:rational}) where Dehn fillings do not appear anymore. As 
Corollary~\ref{coro:rational} holds for   $\zeta_1,\ldots,\zeta_r\in e^{i\pi\mathbb Q}$,
 we use continuity and a density argument to extend it to any unitary $\zeta_1,\ldots,\zeta_r\in\mathbb S^1$.

\subsection*{Organization of the paper}

In Section \ref{sec:Reidemeister} we define (a normalized version of) the twisted Alexander polynomial $\Delta_M^{\alpha, n}$, in particular we
establish the basic results in cohomology required for that, based on Appendix~\ref{Appendix:L2} and we prove Theorem~\ref{Thm:AlexNonZero}. 
Section~\ref{sec:Dehnfillings} is devoted to construct Dehn fillings that approximate $M$, so that the character in \eqref{eqn:chi} extends to them, as we assume that
$\zeta_1,\ldots,\zeta_r\in e^{i\pi\mathbb Q}$. Analytic torsion is discussed in Section~\ref{sec:Analytic}, and the main results on twisted Ruelle zeta functions are established
in Section~\ref{subsec:Ruelle}. Section~\ref{section:approx} discusses the behavior of Reidemeister torsion and Ruelle zeta functions under sequences of  approximating
Dehn fillings, and 
the proof of the main theorem is completed in Section~\ref{sec:asymtotic}.

The paper contains three appendices. In Appendix~\ref{Appendix:Orbifolds} we recall the main properties of combinatorial torsion. The results in
$L^2$-cohomology needed in Section~\ref{sec:Reidemeister}
are established in Appendix~\ref{Appendix:L2}. Finally, in Appendix~\ref{Appendix:pseudoAnosov} we 
 establish 
 Theorem~\ref{Theorem:AnosovDynamicsIntro}.

\subsection*{Acknowledgements}
L.B. thanks warmly Nicolas Bergeron for many enlightening conversations on related topics. He also thanks Shu Shen for indicating him that M\"uller's proof of Fried's theorem should generalize to the case of unitary twist.
A large part of this work has been conducted while L.B. was hosted by the University of Geneva, supported by the NCCR SwissMAP (Swiss National foundation). L.B. is partially funded by the RTG 2491 "Fourier Analysis and Spectral Theory".
M.H.~and J.P.~have been funded by the MEC grant MTM2015-66165-P
 and J.~P.~by the MEC through ``María de Maeztu'' Programme for Units of Excellence in R{\&}D (MDM-2014-0445).

\section{Reidemeister torsion and twisted Alexander polynomials}
\label{sec:Reidemeister}
In this section we define the twisted Alexander polynomial for a 
cusped hyperbolic manifold $M$, equipped with  an epimorphism 
$\alpha \colon \pi_1(M) \twoheadrightarrow \mZ^r$   
satisfying Assumptions \ref{Assumption:alpha}, 
and a lift $\rho$ 
of the holonomy satisfying Assumption~\ref{Assumption:Lift}.  
Before defining the  polynomial from the Reidemeister torsion of the pair $(M,\alphaa\otimes \rho_n)$,
we need to consider homology and cohomology of $(M,\chi\otimes\rho_n)$, 
as the representation  $\chi\otimes\rho_n$  is a specialization of $\alphaa \otimes \rho_n$.

In Subsection \ref{sub:HomologyM} we study (co)-homology of $(M,\chi \otimes \rho_n)$ and $(M,\alphaa \otimes \rho_n)$.
In Subsection \ref{sub:Reidemeister} we consider the Reidemeister torsion of $(M,\chi \otimes \rho_n)$.
Then we define the twisted Alexander polynomial in Subsection \ref{subsec:twistedAlex} from the Reidemeister torsion of 
$(M,\alphaa\otimes\rho_n)$. We express evaluations of the twisted Alexander polynomial at unit complex numbers as Reidemeister torsions of the representations $\chi \otimes \rho_n$ and we prove Theorem~\ref{Thm:AlexNonZero}.

Preliminary constructions and results on homology, 
cohomology and Reidemeister torsion are gathered in Appendix~\ref{Appendix:Orbifolds}, where we also recall some
properties of $\operatorname{Sym}^{n-1}$.
This section also relies on results on $L^2$-cohomology from Appendix~\ref{Appendix:L2}.

\subsection{Cohomology of $(M,\chi\otimes\rho_n)$ and $(M,\alphaa\otimes\rho_n)$}
\label{sub:HomologyM}
When $\chi$ is trivial, the results of this subsection 
on cohomology twisted by $\chi\otimes\rho_n=\rho_n$
can be found in \cite[Section~4]{MP14}.  

In Corollary~\ref{Coro:injection} (in Appendix~\ref{Appendix:L2}) we prove that the inclusion 
$\partial \overline M \hookrightarrow\overline  M$ induces a monomorphism
\begin{equation}
\label{eqn:injection}
0\to H^1(\overline M, \chi\otimes\rho_n) \to H^1(\partial\overline M, \chi\otimes\rho_n).
\end{equation}
Thus to understand the cohomology of $M$ we need to understand the cohomology of the peripheral tori $T^2_i$, $i=1,\ldots, l$, 
where $$
\partial\overline M=T^2_1\sqcup\cdots\sqcup T^2_l.
$$
is the decomposition in connected components.
In particular $l$ is the number of cusps  of $M$.

\subsubsection{Peripheral cohomology}

\begin{lemma}
\label{Lemma:dimT2} If  Assumptions~\ref{Assumption:alpha} and~\ref{Assumption:Lift} hold, then
 for any peripheral torus~$T_i^2$
  \begin{enumerate}[(a)]
 \item
  $
 \dim_\CC H^0(T^2_i, \chi\otimes\rho_n)=
 \begin{cases}
  0 & \textrm{if } n \textrm{ even or }   \chi(\pi_1( T^2_i))\neq\{ 1\},   \\
  1 & \textrm{if } n \textrm{ odd and }   \chi(\pi_1( T^2_i))=\{ 1\}.
 \end{cases}
 $
 \item $
  \dim_\CC H^0(T^2_i, \chi\otimes\rho_n) =\dim_\CC H^2(T^2_i, \chi\otimes\rho_n) =\frac{1}{2}\dim_\CC H^1(T^2_i, \chi\otimes\rho_n) .
 $ 
  \end{enumerate}
\end{lemma}

\begin{proof}
(a) To compute its dimension, we view  $H^0(T^2_i, \chi\otimes\rho_n)$  as the space of invariants $(\CC^n)^{\chi\otimes\rho_n(\pi_1(T^2_i) )}$. 
For any non-trivial element $\gamma$ in $\pi_1(T^2_i)$ its image $\rho(\gamma)$ by the holonomy is parabolic, with trace $\epsilon_\gamma 2$, for some $\epsilon_\gamma=\pm 1$. 
Hence  $\chi(\gamma)\,\rho_n( \gamma)$ has only one  
eigenspace, with dimension one and eigenvalue $\chi(\gamma) \epsilon_\gamma^{n-1}$ (see Remark~\ref{remark:isotropic} in Appendix~\ref{Appendix:Orbifolds}). 

By Assumptions~\ref{Assumption:alpha} and~\ref{Assumption:Lift}, when $n$ is even or when  
$\chi( \pi_1( T^2_i)) \neq\{ 1\}$, there is always  an element $\gamma\in\pi_1(T^2_i)$
that satisfies $\chi(\gamma) (\epsilon_\gamma)^{n-1}\neq 1$, thus $\dim_\CC (H^0(T^2_i, \chi\otimes\rho_n))=0$.
In case $n$ is odd and  $  \chi( \pi_1( T^2_i))=\{ 1\}$, then $\chi(\gamma) (\epsilon_\gamma)^{n-1}= 1$ for every $\gamma \in\pi_1(T^2_i)$.

(b) Poincar\'e duality induces a nondegenerate pairing, see Remark~\ref{Remark:Poincare}:
$$
H^0(T^2_i, \chi\otimes\rho_n)\times H^2(T^2_i, \overline \chi\otimes\rho_n)
\to\CC.
$$
Notice that we use $\chi$ and its inverse $\overline\chi$ (its complex conjugate). 
Hence by (a):
$$
\dim_\CC (H^0(T^2_i, \chi\otimes\rho_n))= \dim_\CC (H^0(T^2_i, \bar \chi\otimes\rho_n))= \dim_\CC (H^2(T^2_i, \chi\otimes\rho_n)).
$$
Then the assertion follows from these computations and vanishing of the Euler characteristic of $T^2$.
\end{proof}

To compute further cohomology groups, we  discuss $L^ 2$-forms, in particular de Rham cohomology.
Let $ E_{\chi\otimes\rho_n}$ denote the flat bundle on $M$ (or on any submanifold) twisted by the representation  ${\chi\otimes\rho_n}$, see Appendix~\ref{Appendix:L2}.
For each peripheral torus  $T^ 2_i$, let  $T^ 2_i\times [0,\infty)\subset M$ denote the cusp,   which is an end of $M$, and consider
the space of forms on the cusp valued on the bundle $ E_{\chi\otimes\rho_n}$, $\Omega^ *(T^ 2_i\times [0,\infty), E_{\chi\otimes\rho_n})$.
It is equipped with a metric as in 
 Appendix~\ref{Appendix:L2}, in particular we may talk about $L^ 2$-forms, as forms with a finite norm.
A cohomology class is called $L^ 2$ if represented by an $L^ 2$-form, and the subspace of $L^2$-cohomology classes in $H^i(T^ 2_j\times [0,+\infty), E_{\chi \otimes \rho_n})$
is denoted by $H^i(T^ 2_j\times [0,+\infty), E_{\chi \otimes \rho_n})_{L^2}$.

\begin{lemma}
\label{lemma:L2}
Assume that $n$ is odd and that the restriction of the character $\chi(\pi_1(T^ 2_j))$ is trivial. Then:
  \begin{enumerate}[(a)]
 \item
Every class in $H^i(T^ 2_j\times [0,+\infty), E_{\rho_n})_{L^2}$ is represented
by a form $v\otimes \omega$, where $ \omega$ is an $i$-form on $T^ 2$ and $v\in (\CC^ n)^{\rho(\pi_1(T^ 2))  }$. 
\item
$ 
\dim_\CC H^0(T^ 2\times [0,+\infty), E_{\rho_n})_{L^2}= \dim_\CC H^1(T^ 2\times [0,+\infty), E_{\rho_n})_{L^2}=1
$
and $H^2(T^ 2\times [0,+\infty), E_{\rho_n})_{L^2}=0$. 
\end{enumerate}
\end{lemma}

%

\begin{proof}
In \cite[Lemma~3.3]{MP12} the same statement is proved for the composition with the adjoint representation on the
Lie algebra $\mathfrak{sl}_n(\CC)$,
$\operatorname{Ad}\circ \operatorname{Sym}^{n-1}$. 
Recall that from Lemma~\ref{Lemma:dimT2} the space of invariants $(\CC^n)^ {\rho_n(\pi_1(T^ 2_i))}$ is one-dimensional, for $n$ odd. 
Then the lemma follows from Clebsch-Gordan formula, see Equation~\eqref{eqn:Clebsch-Gordan} in Appendix~\ref{Appendix:Orbifolds}.
\end{proof}

\subsubsection{Cohomology of M}
In this paragraph we prove the properties of the cohomology of $M$ required for our definition of twisted Alexander polynomial.
\begin{theorem}
\label{Theorem:vanishing}
Let $M$, $\rho$, and $\alpha$ satisfy Assumptions~\ref{Assumption:alpha} and \ref{Assumption:Lift}. 
\begin{enumerate}[(a)]
\item
If $n$ is even or if $\chi$ is non trivial on every peripheral subgroup, then
$H^*(M,\chi\otimes\rho_n)=0$.
\item
If $n$ is odd, then $\dim_\CC H^1(M;\chi\otimes\rho_n)=\dim_\CC H^2(M;\chi\otimes\rho_n)$ is the number of peripheral subgroups to which the restriction
of $\chi$ is  trivial. 
\end{enumerate}
\end{theorem}

\begin{proof}
For both (a) and (b), first notice that $M$ has the homotopy type of a 2--complex, hence $H^i(M, \chi\otimes \rho_n)=0$ for any $i \ge 3$.
In addition,  the space of invariants $H^0(M, \chi \otimes \rho_n)\cong ( \CC^n)^{\chi\otimes\rho_n (\pi_1(M)) }$ also vanishes since $\rho_n$ is irreducible.

To prove (a),  
the vanishing of $H^1(M, \chi \otimes \rho_n)$ is a consequence of Lemma~\ref{Lemma:dimT2} and the monomorphism in~\eqref{eqn:injection}.
We conclude that $H^2(M, \chi \otimes \rho_n)=0$ because the Euler characteristic $\chi(M)$ is zero.
 
For (b) assume that $n$ is odd.  
We use that  $H^1( M,\chi\otimes\rho_n)$ has no $L^ 2$-forms, by  Theorem~\ref{Thm:vanishingL2},
hence by Lemmas~\ref{Lemma:dimT2} and~\ref{lemma:L2},
the map
$$
H^1(\overline M,\chi\otimes\rho_n)
\to
H^1(T^2_i, \rho_n)
$$
has rank at most one if $\chi\vert_{\pi_1(T^2_i)}$ is trivial, and 0 otherwise.
 Thus, if $s$ is the number of peripheral tori $T^2_ i$ where
 $\chi$ restricts  trivially, by \eqref{eqn:injection},
$$
\dim_\CC H^1(M; \chi\otimes\rho_n) \leq  s.
$$
On the other hand,  using duality twice (Poincar\'e and homology/cohomology) 
$H^3(\overline{M},\partial \overline{M} ; \chi\otimes\rho_n)=0$ and,
by the long exact sequence of the pair,
$
H^2(\overline M,\chi\otimes\rho_n)
\twoheadrightarrow
H^2(\partial \overline M,\chi\otimes\rho_n)
$ is a surjection.
Hence by  Lemma~\ref{Lemma:dimT2}:
$$
\dim_\CC H^2(M; \chi\otimes\rho_n) \geq  s.
$$
Finally, as $\chi(M)=0$, 
$
\dim_\CC H^1(M; \chi\otimes\rho_n)= \dim_\CC H^2(M; \chi\otimes\rho_n) = s
$.
\end{proof}

We need  to precise the bases for the cohomology groups. It is easier to describe them for the homology groups.
For a torus $T^2_i$ such that $\chi( \pi_1 (T^2_i))$ is trivial, if $h_i\in \CC^n$ is
invariant by $\pi_1(T^2_i)$, then the class of $h_i\otimes T^2_i$ is a well defined element in 
$H_2(T^2_i, \chi\otimes\rho_n)$, and so is $h_i\otimes l_i$  in 
$H_1(T^2_i, \chi\otimes\rho_n)$.

\begin{lemma}
\label{Lemma:basis}
 Assume that $\chi$ is trivial precisely on $\pi_1(T^2_1),\ldots, \pi_1(T^2_s)$. Let $h_i\in \CC^n$ be 
 non-zero and invariant by $\pi_1(T^2_i)$, for $i=1\ldots,s$.
 Let $i_*$ denote the map induced by inclusion in homology. Then:
 \begin{enumerate}[(a)]
  \item $\{i_*( h_1\otimes T^2_1),\ldots,i_*( h_s\otimes T^2_s)\}$ is a basis for $H_2(M, \chi\otimes\rho_n)$.
  \item $\{i_*( h_1\otimes l_1),\ldots,i_*( h_s\otimes l_s)\}$ is a basis for $H_1(M, \chi\otimes\rho_n)$.
 \end{enumerate}
\end{lemma}

\begin{proof}
For $i=1, \ldots, s$, since $\chi$ is trivial on $T_i^2$,   (a) follows from the isomorphisms
$$H_2(T^2_i,\rho_n)\cong H^0(T^2_i,\rho_n)\cong (\CC^n)^{\rho_n(\pi_1(T^2))},
$$
and from
the isomorphism 
$$
0\to H_2(T^2_1,\rho_n)\oplus\cdots\oplus H_2(T^2_s,\rho_n)\overset{i_*}\to  H_2(\overline M, \chi\otimes\rho_n)\to 0
$$
coming from the long exact sequence in homology.

For (b) we claim first that  $ h_j\otimes l_j$ is non-zero in $H_1( T^2_j,\rho_n)$, for $j=1,\ldots s$.
We prove the claim  by computing cellular homology explicitly. For this purpose, chose a cell decomposition of the torus with
one 0-cell, one 2-cell and two 1-cells, that are loops, and assume that one of these loops represents $l_j$. 
Furthermore, using the description of  $\rho_n(\pi_1(T^2_j))$,  a straightforward computation  shows that 
$ h_j\otimes l_j$ is not a boundary,
see \eqref{eqn:de2} below.
Alternatively, as in the proof of Lemma~\ref{lemma:L2}, an equivalent statement is proved in 
\cite[Lemma~3.4]{MP12} for $\operatorname{Ad}\circ \operatorname{Sym}^{n-1}$, and our claim follows from Clebsch-Gordan 
formula~\eqref{eqn:Clebsch-Gordan}.

From the proof of Theorem~\ref{Theorem:vanishing} we have an injection
$$
0\to  H^ 1(M, \overline\chi\otimes\rho_n))\overset{i^ *}\to H^1(T^2_1,\rho_n)\oplus\cdots\oplus H^1(T^2_ s,\rho_n)
$$
and a surjection
$$
 H_1(T^2_1,\rho_n)\oplus\cdots\oplus H_1(T^2_s,\rho_n)\overset{i_*}\to  H_1(M, \chi\otimes\rho_n)\to 0.
$$
We also have naturality with the pairing between homology and cohomology (see Appendix~\ref{Appendix:Orbifolds}):
$$
\langle i_* (-) , - \rangle = \langle -, i^* (-) \rangle
$$
where the pairing on $\partial \overline M$ is understood to be the sum of pairings on each component $T^2_i$.
Thus, by Poincar\'e duality,
\begin{equation}
 \label{eqn:i*perp}
\ker (i_*)=\operatorname{im} (i^*) ^\perp.
 \end{equation}
By Remark~\ref{remark:isotropic}, $h_j\in(\mathbb{C}^n)^{\rho_n(\pi_1(T^2_i)) }$ is isotropic for the $\rho_n$-invariant bilinear form. Hence
by Lemma~\ref{lemma:L2}(a),   the pairing between $   h_j\otimes l_j $
 and any $L^ 2$-class in $H^ 1(T^ 2_j\times [0,\infty),\rho_n)$ vanishes. Thus,  by dimension considerations:
\begin{equation}
 \label{eqn:L2*perp}
\langle h_1\otimes l_1,\ldots, h_s\otimes l_s\rangle= 
\left( H^ 1(\partial\overline M\times [0,\infty),\overline\chi\otimes\rho_n)_ {L^2} \right)^\perp.
\end{equation}
Furthermore, by Theorem~\ref{Thm:vanishingL2}:
$$
\operatorname{im} (i^*)\cap H^ 1( \partial\overline M  \times [0,\infty),\overline\chi\otimes\rho_n)_ {L^2}= 0.
$$
As
$\dim \operatorname{im} (i^*)= \dim H^ 1( \partial\overline M  \times [0,\infty),\overline\chi\otimes\rho_n)_ {L^2}=\frac{1}{2} \dim  H^ 1(\partial\overline M,\overline\chi\otimes\rho_n)=s$,
\begin{equation}
 \label{eqn:directsum}
 \operatorname{im} (i^*)\oplus H^ 1( \partial\overline M  \times [0,\infty),\overline\chi\otimes\rho_n)_ {L^2}=H^ 1(\partial\overline M,\overline\chi\otimes\rho_n).
\end{equation}
Finally, \eqref{eqn:directsum}, \eqref{eqn:i*perp} and \eqref{eqn:L2*perp} yield
$$
\ker(i_*) \oplus  \langle h_1\otimes l_1,\ldots, h_s\otimes l_s\rangle  =  H_ 1(\partial\overline M,\chi\otimes\rho_n),
$$
in particular
$\langle h_1\otimes l_1,\ldots, h_s\otimes l_s\rangle\cap \ker(i_*)=0$.
Thus $\{ i_*(h_1\otimes l_1),\ldots, i_*( h_s\otimes l_s) \}$
are linearly independent, hence a basis.
\end{proof}

When $\chi$ is trivial,  Lemma~\ref{Lemma:basis} is \cite[Proposition~4.6]{MP14}.

\subsection{Reidemeister torsion}
\label{sub:Reidemeister}

We use the convention of \cite{Mil66} and \cite{Turaev86} for Reidemeister 
torsion, so that it is compatible with the standard convention for analytic 
torsion but it is the reciprocal
to the twisted Alexander polynomial. See Appendix~\ref{Appendix:Orbifolds}.

As a consequence of Theorem~\ref{Theorem:vanishing}, we have that 
$$\vert\tor(M,\chi\otimes\rho_n)\vert\in \RR_{> 0}$$ is well defined when $n$ 
is even or when $n$ is odd and the restriction
of $\chi$ to every peripheral torus is nontrivial. 
The absolute value in $\vert\tor(M,\chi\otimes\rho_n)\vert$ is needed, because 
$\chi$ introduces an indeterminacy of the argument, more precisely $\tor(M, 
\chi \otimes \rho_n)$ is only defined up to multiplication by a unit complex 
number.

In the non-acyclic case we shall consider  
$$
\vert\tor(M,\chi\otimes\rho_n; 
b_1, b_2)\vert,
$$
where $b_1$ and $b_2$ are the basis of the homology provided by 
Lemma~\ref{Lemma:basis}. Notice that $\vert\tor(M,\chi\otimes\rho_n; b_1, 
b_2)\vert$ is independent on the vectors $h_i$ in Lemma~\ref{Lemma:basis}. This 
follows since $(\mC^n)^{\pi_1(T_i)}$ is one-dimensional and 
$\vert\tor(M,\chi\otimes\rho_n; b_1, b_2)\vert$ does not change if we replace 
the vector $h_i$ in Lemma~\ref{Lemma:basis} by a multiple, since in the formula 
for the torsion the multiple   cancel out.

\subsection{Twisted Alexander polynomials}
\label{subsec:twistedAlex}

In this subsection we introduce the twisted Alexander polynomial for 
a finite volume 3-manifold $M$ (connected and orientable) and an epimorphism $\alpha\colon\pi_1(M)\to \ZZ^r$ as in the introduction. 
We define it as the inverse of a Reidemeister torsion of $\alphaa\otimes\rho_n$, where $\alphaa(\gamma)=t_1^{\alpha_1(\gamma)}
\cdots t_r^{\alpha_r(\gamma)}$, $\forall\gamma\in \pi_1(M)$.

We use the notation $\CC[t^{\pm 1}]=\CC[t_1^{\pm 1},\ldots,t_r^{\pm 1}]$ for the ring of Laurent polynomials and 
$\CC(t)=\CC(t_1,\ldots,t_r)$ for its field of fractions. 

\begin{lemma}
\label{Lemma:alphaacyclic}
If Assumptions~\ref{Assumption:alpha} and \ref{Assumption:Lift} hold, then 
$H^*(M, \alphaa \otimes \rho_n)=H_*(M, \alphaa \otimes \rho_n)=0$.
\end{lemma}

\begin{proof}
Choose $\zeta_1,\ldots\zeta_r\in\mS^1$ generic so that the corresponding homomorphism $\chi\colon\pi_1(M)\to \mathbb S^1$ in \eqref{eqn:chi}  has non trivial restriction
on each peripheral subgroup $\pi_1(T^2_i)$. Then by 
 Theorem~\ref{Theorem:vanishing} we have $H^*(M,\chi\otimes\rho_n)=0$.
 Using combinatorial cohomology, 
 notice that  the matrices used to compute $H^*(M,  \chi\otimes\rho_n)$ are the evaluation at $(t_1, \ldots, t_r) =(\zeta_1, \ldots, \zeta_r)$
of the matrices used to compute  
$H^*(M,\alphaa\otimes\rho_n)$. In addition,   the $  \CC(t)$-rank of a matrix with coefficients in $\CC[t^{\pm 1}]$ is larger than or equal 
to its $\CC$-rank after evaluation at  $(t_1, \ldots, t_r) =(\zeta_1, \ldots,\zeta_r)$.
Thus, by 
acyclicity of $\chi\otimes\rho_n$, the  $\CC$-rank of the matrices used to compute cohomology is maximal,
hence the $\CC(t)$-rank of these matrices before evaluation at $(t_1, \ldots, t_r) =( \zeta_1, \ldots,  \zeta_r)$ is also maximal, and therefore 
$\alphaa\otimes\rho_n$ is acyclic.
\end{proof}

For each peripheral torus $T^2_i$ chose $m_i$ so that $\pi_1(T^ 2_i)=\langle l_i, m_i\rangle\cong \ZZ^ 2 $, where $l_i$ is an $\alpha$-longitude.
Writing $\alpha(m_i)=(\alpha_1(m_i),\ldots, \alpha_r(m_i))\in\ZZ^ r$, we denote
$$
t^{\alpha(m_i)}= \alphaa(m_i)=t_1^{\alpha_1(m_i)}\cdots t_r^{\alpha_r(m_i)}.
$$

\begin{definition}
\label{Def:twisted}
The \emph{twisted Alexander polynomial} of $(M, \alphaa\otimes\rho_n)$ is
$$
\Delta_{M}^{\alpha,n}(t_1,\ldots,t_r):=\begin{cases}
\dfrac{1}{\tor(M,\alphaa\otimes \rho_n)} & \textrm{ for }n \textrm{ even,}                        
\\
\dfrac{1}{\tor(M,\alphaa\otimes \rho_n) (t^{\alpha(m_1)}-1)\cdots (t^{\alpha(m_l)}-1)} & \textrm{ for }n \textrm{ odd.}
                       \end{cases}
$$
\end{definition}

It is an element of $\CC(t)=\CC(t_1,\ldots,t_r)$, a quotient of polynomials in the variables $t_1,\ldots,t_r$, defined up to sign and up to multiplication by
monomials $t_1^{n_1}\cdots t_r^{n_r}$.  In Corollary~\ref{cor:LaurentPolynomial} we prove that it is a Laurent poynomial, an element of 
 $\CC[t^ {\pm 1}]=\CC[t_1^ {\pm 1},\ldots,t_r^ {\pm 1}]$.
 
\begin{remark}
For even dimensional representations, this is the same as Wada's  polynomial \cite{Wada}, using Kitano's Theorem \cite{Kit96}.
For odd dimensional representations,  it is a normalization of the latter. 
\end{remark}

We view $\CC[t^ {\pm 1}]^ n\cong \CC[t^ {\pm 1}]\otimes\CC^ n  $ as a $\pi_1(M)$-module via $\alphaa\otimes\rho_n$, and denote it by $\CC[t^ {\pm 1}]^ n_{\alphaa\otimes\rho_n}$. 
For the definition of the order of a  $\CC[t^ {\pm 1}]$-module, see~\cite{Turaev86}.

\begin{lemma}
\label{Lemma:polynomial}

\begin{enumerate}[(a)]
 \item Up to units in $\CC[t^{\pm 1}]$:
 $$ \dfrac{1}{\tor(M,\alphaa\otimes \rho_n)} =\operatorname{order}_{\CC[t^ {\pm 1}]} H_1(M,\CC[t^ {\pm 1}]^ n_{\alphaa\otimes\rho_n}).$$
\item For $n$ odd, $\operatorname{order}_{\CC[t^ {\pm 1}]} H_1(M,\CC[t^ {\pm 1}]^ n_{\alphaa\otimes\rho_n})\in (t^{\alpha(m_1)}-1)\cdots (t^{\alpha(m_l)}-1)
\CC[t^{\pm 1}]$.
\end{enumerate}
\end{lemma}


\begin{corollary}
\label{cor:LaurentPolynomial}
 The {twisted Alexander polynomial}  is a Laurent polynomial: $$\Delta_{M}^{\alpha,n}\in \CC[t^ {\pm 1}].$$

\end{corollary}

Before proving  Lemma~\ref{Lemma:polynomial} we need the following lemma:

\begin{lemma}
 \label{Lemma:HCT}
 Assume that $n$ is odd.
\begin{enumerate}[(a)]
 \item For each peripheral torus $T^2_j$,
 $$
 H_1(T^2_j, \CC[t^ {\pm 1}]^ n_{\alphaa\otimes\rho_n})\cong \CC[t^{\pm 1}]/ ( t^{\alpha(m_j)}-1).
 $$
In addition, it is generated by the image of 
$h_j\otimes l_j$ via the natural map
$$
(\CC^n)^{\rho_n(l_j)}\otimes H_1(S^1_j;\ZZ)\to  H_1(T^2_j, \CC[t^ {\pm 1}]^ n_{\alphaa\otimes\rho_n}),
$$
where $(\CC^n)^{\rho_n(l_j)}$ is the (1-dimensional) subspace invariant by $\rho_n(l_j)$,
$0\neq h_j\in (\CC^n)^{\rho_n(l_j)}$,
and $S^1_j$ is a circle representing $l_j$.
 \item The inclusion induces a monomorphism
 $$
 H_1( \partial \overline{M},  \CC[t^ {\pm 1}]^ n_{\alphaa\otimes\rho_n})\hookrightarrow H_1(  \overline{M},  \CC[t^ {\pm 1}]^ n_{\alphaa\otimes\rho_n}).
 $$
\end{enumerate}
\end{lemma}

\begin{proof}
To prove  (a) consider $S^1_j\times\RR\to T^2_j$ the infinite cyclic covering with deck transformation group $\ZZ$ generated by $\tau\colon  S^1_j\times\RR\to S^1_j\times\RR$. There 
is a long exact sequence in homology \cite{MilnorCoverings, Porti97}
\begin{multline*}
\cdots\to H_i(S^1_j\times\RR, \CC[t^ {\pm 1}]^ n_{\alphaa\otimes\rho_n})\overset{\tau_*-1}\longrightarrow H_i(S^1_j\times\RR, \CC[t^ {\pm 1}]^ n_{\alphaa\otimes\rho_n}) \\ \to
H_i(T^2_j\times\RR, \CC[t^ {\pm 1}]^ n_{\alphaa\otimes\rho_n})\to\cdots 
\end{multline*}
As $\alphaa(\pi_1(S^1_j))={1}$, 
$
H_i(S^1_j\times\RR, \CC[t^ {\pm 1}]^ n_{\alphaa\otimes\rho_n})\cong \CC[t^ {\pm 1}]\otimes H_i(S^1_j, \CC^ n_{\rho_n})
$
and the action of $\tau_*$ on $\CC[t^ {\pm 1}]$ corresponds to multiplication by $ t^{\alpha(m_j)}$. Furthermore
 $H_i(S^1_j, \CC^ n_{\rho_n})=0$ for $i\neq 0, 1$ and 
 $H_1(S^1_j, \CC^ n_{\rho_n})\cong H^0(S^1_j, \CC^ n_{\rho_n})\cong (\CC^n)^{\rho_n(l_j)} $. Then  (a) follows from these considerations.
 
In the proof of (b) we omit the subindex ${\alphaa\otimes\rho_n}$ of the module $\CC[t^ {\pm 1}]^ n_{\alphaa\otimes\rho_n}$.
As a consequence of  (a) there is a finite set of characters $\Sigma_j\subset \hom(\ZZ^r,\mS^1)$ such that 
$$
H_1(T^2_j,  \CC[t^{\pm 1}]^n) \cong  \bigoplus_{\chi\in \Sigma_j} H_1(T^2_j,  \CC[t^{\pm 1}]^n)\otimes_\chi \CC.
$$
Thus, for $\Sigma=\Sigma_1\cup\cdots\cup \Sigma_l$, we have a monomorphism
$$
H_1(\partial\overline M,  \CC[t^{\pm 1}]^n) \hookrightarrow  \bigoplus_{\chi\in \Sigma} H_1(\partial\overline M,  \CC[t^{\pm 1}]^n)\otimes_\chi \CC.
$$
Next view this inclusion in the following commutative diagram:
$$
\xymatrix{  H_1(\partial\overline M,  \CC[t^{\pm 1}]^n)   \ar[d]  \ar@{^{(}->}[r]  &   \bigoplus\limits_{\chi\in \Sigma} H_1(\partial\overline M,  \CC[t^{\pm 1}]^n)\otimes_\chi\CC  \ar[d] \ar[r]  
             &   \bigoplus\limits_{\chi\in \Sigma}
            H_1(\partial\overline M,  \rho_n\otimes\chi) \ar[d]  \\
            H_1(\overline M,  \CC[t^{\pm 1}]^n) \ar[r]      &   \bigoplus\limits_{\chi\in \Sigma} H_1( \overline M,  \CC[t^{\pm 1}]^n) \otimes_\chi\CC \ar[r]        &   \bigoplus\limits_{\chi\in \Sigma}
            H_1(\overline M,  \rho_n\otimes\chi) 
            }
$$
By  Theorem~\ref{Theorem:vanishing}, Lemma~\ref{Lemma:basis} and (a), the morphism
$$
  H_1(\partial\overline M,   \CC[t^{\pm 1}]^n)\otimes_\chi\CC  \to
            H_1(\overline M,  \rho_n\otimes\chi) 
$$
is an isomorphism for every $\chi\in\hom(\ZZ^r,\mathbb S^1)$. Thus it follows from commutativity of the diagram that 
$H_1(\partial\overline M,   \CC[t^{\pm 1}]^n)\to  H_1(\overline M,   \CC[t^{\pm 1}]^n) $ is injective.
\end{proof}

\begin{proof}[Proof of Lemma~\ref{Lemma:polynomial}]
By Turaev \cite{Turaev86} we have an equality (up to units in $\CC[t^{\pm 1}]$):
$$
 \frac1{\tor(M,\alphaa\otimes \rho_n)} =
 \frac{\operatorname{order}(H_1(M,\CC[t^ {\pm 1}]^ n_{\alphaa\otimes\rho_n}))}{\operatorname{order} (H_0(M,\CC[t^ {\pm 1}]^ n_{\alphaa\otimes\rho_n}))
 \operatorname{order}(H_2(M, \CC[t^ {\pm 1}]^ n_{\alphaa\otimes\rho_n}))}
$$
As the manifold $M$ has the simple homotopy type of a 2-dimensional complex,  we have that $\operatorname{order}  (H_2(M, \CC[t^ {\pm 1}]^ n_{\alphaa\otimes\rho_n}))=1$. 
Hence it suffices  to prove that
$\operatorname{order}(H_0(M, \CC[t^ {\pm 1}]^ n_{\alphaa\otimes\rho_n}))=1$. So, looking for a contradiction, assume  
$\operatorname{order}(H_0(M, \CC[t^ {\pm 1}]^ n_{\alphaa\otimes\rho_n}))\neq 1$ and pick
$\lambda=(\lambda_1, \ldots, \lambda_r) \in~(\mC^*)^r$ a root of the polynomial 
$\operatorname{order}(H_0(M,\CC[t^ {\pm 1}]^ n_ {\alphaa\otimes\rho_n}))$
(this polynomial defines a hypersurface in $\CC^r$, and since the order is defined up to factors $t_ i^{\pm 1}$,
it intersects  $(\CC^*)^r$ non trivially). 
Evaluation at $(t_1,\ldots, t_r)=(\lambda_1,\ldots , \lambda_r)$
defines a morphism $\Lambda\colon\pi_1(M)\to\CC^*$,
by $\Lambda(\gamma)= \lambda_1 ^{\alpha_1(\gamma)}\cdots \lambda_r ^{\alpha_r(\gamma)}$.
Since $\lambda$ is a root   of the polynomial 
$\operatorname{order}(H_0(M,\CC[t^ {\pm 1}]^ n_ {\alphaa\otimes\rho_n}))$, 
$H_0(M,\Lambda\otimes \rho_n)\neq 0$. 
This means that $\CC^n$
has non-zero coinvariants for the action of $\Lambda\otimes \rho_n$. By duality, $\CC^n$
has non-zero invariants by the action of $\overline\Lambda\otimes \rho_n$, in particular $\CC^n$ has a proper subspace preserved by $\rho_n$. By Zariski density of the holonomy representation, 
this contradicts irreducibility of  $\mathrm{Sym}^{n-1}$. This proves~(a).

For (b) construct the twisted chain complex from a CW--complex $K$, with $\vert K\vert=M$:
$$
C_*(K, \CC[t^ {\pm 1}]^ n_{\alphaa\otimes\rho_n})= \CC[t^{\pm1}]^n\otimes_{\alphaa\otimes \rho_n} C_*(\widetilde K, \ZZ).   
$$
We may assume furthermore that there are 2-cells of $K$, denoted by $e^2_i$, representing $T^2_i$ and 1-cells denoted by  $e^1_i$ representing $l_i$.
Chose respective lifts $ \tilde e^{2}_i$ of  $e^2_i$,
and $ \tilde e^{1}_i$  of  $e^1_i$, in the universal cover $\widetilde{K}$ 
that correspond to the same
connected component of the lift of the peripheral torus $T^2_i$ in the universal covering. Moreover, chose
$\tilde e^1_i$ to be adjacent to  $\tilde e^2_i$, so that 
\begin{equation}
\label{eqn:de2}
\partial \tilde e^2_i= (m_i-1) \tilde e^1_i+ (1-l_i) \tilde f^1_i 
\end{equation}
for some other $1$-cell $f^ 1_i$. Notice that $\langle m_i, l_i\rangle\cong\pi_1(T^ 2_i)$,
with $\alphaa(l_i)=1$ and $\alphaa(m_i)= t^ {\alpha(m_i)}$.
Chose also $h_i \in \CC^n$ a non-zero element invariant by $\rho_n(\pi_1(T^2_i))$. 
Let $L_*\subset C_*(K, \CC[t^ {\pm 1}]^ n_{\alphaa\otimes\rho_n})$ be the  $\CC[t^{\pm 1}]$-subcomplex generated by the 
elements
$h_i\otimes \tilde e^{j}_i$, $j=1, 2$, $i=1,\ldots , l$. 
By the choice of lifts:
\begin{equation}
\label{eqn:boundary}
\partial (h_i\otimes \tilde e^2_{i}) =  (t^{\alpha(m_i)}-1) h_i\otimes \tilde e^1_{i} ,
\end{equation}
up to sign and up to powers of $t^{ \alpha(m_i)}$,
and $\partial (h_i\otimes \tilde e^1_{i}) =0$ since $e_i^1$ represents $l_i$.
Hence $L_*$ is a subcomplex of $C_*(K, \alphaa\otimes\rho_n)$. Moreover it follows from (\ref{eqn:boundary}) 
that this complex $L_*$ is also acyclic as a complex of $\mC(t)$-vector spaces. We have a short exact sequence of acyclic complexes of $\mC(t)$-vector spaces:
\begin{equation}
 \label{eqn:Lshort}
0\to L_* \to C_*(K, \CC[t^ {\pm 1}]^ n_{\alphaa\otimes\rho_n})\to C_*(K, \CC[t^ {\pm 1}]^ n_{\alphaa\otimes\rho_n})/L_*\to 0.
\end{equation}
Furthermore we can construct geometric bases \`a la Milnor 
for $ C_*(K, \CC[t^ {\pm 1}]^ n_{\alphaa\otimes\rho_n})$
that include 
the elements $h_i\otimes \tilde e^{j}_i$, Definition~\ref{definition:geometricbasis}. 
Thus there are compatible geometric bases in the sequence and the multiplicativity formula for the torsion \cite[Theorem~3.2]{Mil66} provides the equality:
$$
\tor (M, \alphaa\otimes \rho_n)= \tor(L_*) \tor (  C_*(K, \CC[t^ {\pm 1}]^ n_{\alphaa\otimes\rho_n})/L_*  ).
$$
From \eqref{eqn:boundary} the contribution of $\tor(L_*) $ is $(t^{m_1}-1)\cdots (t^{m_l}-1)$. Finally, 
we show that the zeroth and second homology groups of $C_*(K, \CC[t^ {\pm 1}]^ n_{\alphaa\otimes\rho_n})/L_*$ 
vanish (hence its torsion is the inverse of a Laurent polynomial).
For that purpose, notice that from Lemma~\ref{Lemma:HCT}~(a) and \eqref{eqn:boundary} we have a natural isomorphism 
$$
H_1(L_*)\cong H_1(\partial M, \CC[t^ {\pm 1}]^ n_{\alphaa\otimes\rho_n}).
$$
Hence by Lemma~\ref{Lemma:HCT} (b) we have an injection induced by inclusion
$$
H_1(L_*)\hookrightarrow H_1( M, \CC[t^ {\pm 1}]^ n_{\alphaa\otimes\rho_n}).
$$
Using this monomorphism,  the long exact sequence in homology corresponding to \eqref{eqn:Lshort}, and the vanishing of 
$H_i( M, \CC[t^ {\pm 1}]^ n_{\alphaa\otimes\rho_n})$  for $i=0,2$, it follows that 
zeroth and second homology groups of $C_*(K, \CC[t^ {\pm 1}]^ n_{\alphaa\otimes\rho_n})/L_*$ also vanish.
\end{proof}

\begin{proposition}
\label{Prop:evaluation}
For $n$ even,
 $$
 | \Delta_{M}^{\alpha,n}(\zeta_1,\ldots,\zeta_r) | = \frac{1}{|\tor(M, \chi\otimes\rho_n)|}. 
 $$
For $n$ odd,
  $$
 |\Delta_{M}^{\alpha,n}(\zeta_1,\ldots,\zeta_r)  | = \frac{1}{|\tor(M, \chi\otimes\rho_n; b_1, b_2  )|} 
 \prod_{\zeta^{\alpha(m_i)}\neq 1 }\frac{1}{|\zeta^{\alpha(m_i)}-1|} .
 $$
\end{proposition}

In the proposition, $b_2$ and $b_1$ are the basis in homology of Lemma~\ref{Lemma:basis}, according to the components where $\chi (\pi_1(T^2_i))$ is trivial, 
$\alpha(m_i)\in\ZZ^r$ is a generator of the image of $\alpha(\pi_1(T^2_i))$. We use the notation $\zeta^{\alpha(m_i)}=\zeta_1^{\alpha_1(m_i)}\cdots \zeta_r^{\alpha_r(m_i)}$.
The product in the odd case runs on the components where  $\chi(\pi_1(T^2_i))$  is non trivial.

\begin{proof}
In the acyclic case (when $n$ is even or when $\chi$ is non-trivial on
each peripheral subgroup)
the proposition follows  from naturality, cf.~\cite[\S~6]{Mil66}.

The proof of the non-acyclic case is very similar to the proof of Lemma~\ref{Lemma:polynomial} b), 
but the subcomplex $L_*$ is only constructed from the peripheral tori for which the restriction of $\chi$ is trivial.
Namely, assume that $n$ is odd and that $\chi$ is trivial precisely on $\pi_1(T^2_1),\ldots,\pi_1(T^2_s)$. Then, choosing a CW-complex $K$ as in the proof of
Lemma~\ref{Lemma:polynomial}, we take $L_*$ to be the subcomplex of $C_*(K, \chi\otimes\rho_n)$ generated by elements $h_i\otimes \tilde e^j_i$, $j=1, 2$,
$i=1,\ldots, s$. In this case the boundary operator is  zero in $L_*$, and a geometric basis for $L_*$ is precisely a lift of the basis $b_1$ and $b_2$
for $H_*(M, \chi\otimes\rho_n)$, by Lemma~\ref{Lemma:basis}. From the defining short exact sequence and the previous consideration,
it follows that 
$C_*(K, \chi\otimes\rho_n)/L_*$ is acyclic and its torsion equals 
$|\tor(M, \chi\otimes\rho_n; b_1, b_2  )|$. Then the lemma follows from naturality applied to $C_*(K, \CC[t^{\pm 1}]^n)/L_*'$, where 
$L_*'$ is the subcomplex $\CC[t^{\pm 1}]$-generated by the same elements as $L_*$, $h_i\otimes \tilde e^j_i$, $j=1, 2$,
$i=1,\ldots, s$.
\end{proof}

For $\rho_3=\operatorname{Sym}^2\circ\rho= \operatorname{Ad}\circ\rho$ and $\chi$ trivial,
Proposition~\ref{Prop:evaluation} has been proved  by Yamaguchi in \cite{YamaguchiAIF} for knot exteriors
and by Dubois and Yamaguchi in \cite{JeromeYoshi} in a more general setting.


\begin{proof}[Proof of Theorem~\ref{Thm:AlexNonZero}]
 Proposition~\ref{Prop:evaluation} expresses 
the evaluation of the twisted Alexander polynomial at unitary complex numbers as a Reidemeister torsion,
which is an element of $\mC^*$, up to multiplication by a unit complex number. In particular it is not zero. 
\end{proof}

%

\section{Dehn fillings and rational twists}
\label{sec:Dehnfillings}

In this section we consider Dehn fillings on $M$  and sequences of those fillings
that converge geometrically to $M$. In the first subsection we discuss compatibility conditions with the twist and with the
lift of the holonomy (equivalently the spin structure). In particular we restrict to rational Dehn fillings. In the second subsection we discuss surgery formulas for the torsion.

\begin{definition}
 A unitary twist $\chi\colon \pi_1(M)\to \mathbb S^1\subset \CC$ is called \emph{rational} if it takes values in $e^{2\pi i\,\mathbb Q}$.
\end{definition}

In this section we shall assume that $\chi$ is rational, so that it induces a twist of certain Dehn fillings, as we explain in the next subsection.

\subsection{Compatible Dehn fillings}
For each peripheral torus $T^2_i$ we have an $\alpha$-longitude, namely an element $l_i\in \pi_1(T^2_i)$ 
that generates the kernel of $\alpha\vert_{ \pi_1(T^2_i)}$, by 
Assumption~\ref{Assumption:alpha}. 
We fix a basis for the fundamental group of the peripheral group that contains this element: 
   $\langle m_i, l_i\rangle = \pi_1(T^2_i)\cong\ZZ^2$.
As $\operatorname{trace}(\rho(l_i))=-2$ (Assumption~\ref{Assumption:Lift}), we may chose  $\operatorname{trace}(\rho(m_i))=+2$,
after replacing $m_i$ by $m_il_i$ if needed. 

Once we have fixed   the $m_i$ and $l_i$, given pairs of coprime integers $p_i,q_i$, the Dehn filling with filling meridians
$p_i m_i+q_il_i$ is denoted by $M_{p_1/q_1,\ldots, p_l/q_l }$. To simplify notation we write 
$$
M_{p/q}:= M_{p_1/q_1,\ldots, p_l/q_l }.
$$
The  inclusion map is denoted by $
i\colon M\to M_{p/q}
$,
it induces an epimorphism $i_*\colon \pi_1(M)\to \pi_1(M_{p/q})$.
Another convention is that $(p,q)\to \infty$ means that $p_i^2+q_i^2\to+\infty$ for $i=1,\ldots, l$.

By Thurston's hyperbolic Dehn filling theorem,   when $p_i^2+q_i^2$ is sufficiently large for each $i=1.\ldots,l$, then $M_{p/q}$ is hyperbolic. 
The (conjugacy class of the) holonomy of $M_{p/q}$ composed with $i_*$ converges to the (conjugacy class of the) holonomy of $M$ in the set of conjugacy classes of such representations
$\Hom(\pi_1(M), \operatorname{PSL}(2,\CC))/\operatorname{PSL}(2,\CC)$.

As  we work with representations in  $\operatorname{SL}(2,\CC)$, we need to impose compatibility conditions on the Dehn filling to get the
same conclusion for the lifts. We 
shall also impose conditions so that the rational twist $\chi$ factors through $i_*$ to  a twist of $M_{p/q}$.

\begin{definition}
\label{def:compatibility}
The Dehn filling  $M_{p/q}=M_{p_1/q_1,\ldots, p_1/q_l }$ is called \emph{compatible} with $\chi$ and $\rho$ if, for each $i=1,\ldots,l$:
\begin{enumerate}
 \item  $\chi(m_i^{p_i})=1$, and
 \item   $q_i\equiv 1 \mod 2$.
\end{enumerate}
\end{definition}

Since $\chi(l_i)=1$ by Assumption~\ref{Assumption:alpha},  Condition~(1) in the definition
amounts to say that the twist of $M$ factors through $\pi_1(M_{p/q})$.
As we assume $\chi$ rational, this is achieved by taking $p_i\in \operatorname{order}(\chi(m_i)) \ZZ$. 
Condition~(2) in the definition is explained by the following lemma.

\begin{lemma}
\label{lemma: seqconverge}
For an infinite family of compatible Dehn fillings  
$M_{p/q}$ such that ${(p,q)\to\infty}$, there exists a lift of the holonomy $\rho^{p/q}$ of $M_{p/q}$ in 
 $\operatorname{SL}(2,\CC)$ such that 
 $$
 \lim _{(p,q)\to\infty} [\rho^{p/q}\circ i_* ]= [\rho]
 $$
in    $\Hom(\pi_1(M), \operatorname{SL}(2,\CC)) / \operatorname{SL}(2,\CC)$, where $\rho$ is the lift of the holonomy of the complete structure on $M$.
\end{lemma}

\begin{proof} As we chose $m_i$ so that $\operatorname{trace}(\rho(m_i))=+2$, 
 Condition (2) in Definition~\ref{def:compatibility}  means
 $\operatorname{trace}(\rho(p_i m_i+q_i l_i))=-2$.
 Using the natural bijection between spin structures and lifts of the holonomy, 
 this is precisely the condition required
for a spin structure on $M$ to extend to $M^{p/q}$, see \cite{MP14} for instance. In terms of lifts of representations, this 
is the compatibility condition for the lifts of the deformation for the holonomy in Thurston's hyperbolic Dehn filling theorem.
\end{proof}

We shall use the following notation:
\begin{align*}
 \rho^{p/q}_n&:=\operatorname{Sym}^{n-1}\circ\rho^{p/q}\colon \pi_1(M_{p/q})\to \operatorname{SL}_n(\CC), \\
 \varrho^{p/q}_n&:=\rho^{p/q}_n\circ i_*\colon \pi_1(M)\to \operatorname{SL}_n(\CC).
\end{align*}
Thus for large $p_i^2+q_i^2$, $i=1,\ldots, l$, $[\varrho^{p/q}_n]$ lies in a neighborhood of $[\rho_n]$ in $\hom(\pi_1(M),\mathrm{SL}_n(\CC))/\operatorname{SL}_n(\CC)$.

\subsection{Dehn filling formula}

We shall only consider compatible Dehn fillings. 
In particular
 $\chi$ factors through $\pi_1(M_{p/q})$. Since $M_{p/q}$ is closed, Corollary \ref{Coro:injection}
yields that $H^*(M_ {p/q}, \chi\otimes \rho_n^{p/q})$ vanishes.
However $H^*(M, \chi\otimes \varrho_n^{p/q})$ does not need to vanish, we have:

\begin{lemma}
 \label{lemma:cohomologyrhopq} For sufficiently large $p_i^2+q_i^2$, $i=1,\ldots,l$:
  \begin{enumerate}[(a)]
   \item The inclusion $\partial \overline M\to  \overline M$ induces a monomorphism $$
   0 \to H^*(M, \chi\otimes \varrho_n^{p/q})\to 
   H^*(\partial \overline M, \chi\otimes \varrho_n^{p/q}).$$
   \item If $n$ is even or if $\chi$ is non trivial on every peripheral subgroup, then
$$H^*(M,\chi\otimes\varrho_n^{p/q})=0.$$
\item
If $n$ is odd, then $\dim_\CC H^1(M;\chi\otimes\varrho_n^{p/q})=\dim_\CC H^2(M;\chi\otimes\varrho_n^{p/q})$
is the number of peripheral subgroups to which the restriction
of $\chi$ is  trivial. 
  \end{enumerate}
\end{lemma}

\begin{proof}
The proof is analogous to Theorem~\ref{Theorem:vanishing}, but simpler, as we no not need a theorem in $L^2$-cohomology but Mayer-Vietoris exact sequence. More precisely, 
we apply Mayer-Vietoris to the pair $(\overline M, V)$, where $V$ is the union of the attached solid tori 
 in the Dehn filling, so that $V\cap \overline M= \partial \overline M$ and $V\cup \overline M=M_{p/q}$. As $H^*(M_ {p/q}, \chi\otimes \rho_n^{p/q})=0$,
 Mayer-Vietoris long exact sequence yields an isomorphism induced by inclusion maps:
\begin{equation}
 \label{eqn:isoMV}
  H^*(\overline M;\chi\otimes\varrho_n^{p/q})\oplus H^*(V;\chi\otimes\varrho_n^{p/q})\cong H^*(\partial \overline M; \chi\otimes\varrho_n^{p/q}),
\end{equation}
 which proves (a). For the other assertions, it amounts to compute the cohomology of the peripheral 2-tori $T^2_i$ and the solid tori
 $V_i\cong D^2\times S^1$ (the connected components of $V$).
For this purpose, we use that 
$$
H^0(T^2_j, \chi\otimes\varrho^{p/q}_n)
\cong
H^0(V_j, \chi\otimes\varrho^{p/q}_n)
\cong
(\CC^n) ^{ \chi\otimes\varrho^{p/q}_n (T^2_ j) }
=
(\CC^n) ^{ \chi\otimes\varrho^{p/q}_n (V_ j) },
$$
where 
$
(\CC^n) ^{ \chi\otimes\varrho^{p/q}_n (T^2_ j) }
$
denotes the subspace invariant by 
$\chi\otimes\varrho^{p/q}_n (T^2_ j)=\chi\otimes\varrho^{p/q}_n (V_ j)$. 
If $\gamma_j$ denotes the soul of the $j$-th attached solid torus $V_j$, then, after conjugation, 
\begin{equation}
\label{eqn:rhopq}
\rho^{p/q}(\gamma_j)=\left(\begin{smallmatrix}
 e^{\lambda(\gamma_j)/2} & 0 \\
 0  &       e^{-\lambda(\gamma_j)/2}                  
                       \end{smallmatrix}
\right)
\end{equation}
where $\lambda(\gamma_j)$ is the complex length, 
$\lambda(\gamma_j) = \ell(\gamma_j) + i\, \theta(\gamma_j)$, and $\ell(\gamma_j)>0$ is the (real) length of the geodesic $\gamma_j$ in $M_{p/q}$.
Therefore
$$
(\CC^n) ^{ \chi\otimes\varrho^{p/q}_n (T^2_ j) }=
(\CC^n) ^{ \chi(\gamma_j)\rho^{p/q}_n (\gamma_ j) }\cong
\begin{cases}
                                               \CC &  \textrm{for } n \textrm{ odd and }  \chi\vert_{\pi_1T^2_j}\equiv 1 ,\\
                                               0 & \textrm{otherwise}.
                                              \end{cases}
$$
To compute the other cohomology groups, using the vanishing of Euler characteristic and Poincar\'e duality, we have:
$$
\dim H^0(V_j, \chi\otimes\varrho^{p/q}_n)= \dim H^1(V_j, \chi\otimes\varrho^{p/q}_n),
$$
and
$$
\dim H^0(T^2_j, \chi\otimes\varrho^{p/q}_n)= \dim H^2(T^2_j, \chi\otimes\varrho^{p/q}_n)= \frac12\dim H^1(T^2_j, \chi\otimes\varrho^{p/q}_n).
$$
The lemma follows from these computations and  from~\eqref{eqn:isoMV}.
\end{proof}

We also need a basis in homology:

\begin{lemma}
\label{Lemma:basispq}
 Assume that $n$ is odd and that  $\chi$ is trivial precisely on the first $s$ peripheral tori $T^2_1,\ldots, T^2_s$. Let $h_i^{p/q}\in \CC^n$ be a
 non-zero element invariant by $\chi\otimes\varrho_n^{p/q}(\pi_1(T^2_i))$, for $i=1,\ldots,s$.
 Then:
 \begin{enumerate}[(a)]
  \item $\{i_*( h_1^{p/q}\otimes T^2_1),\ldots,i_*( h_s^{p/q}\otimes T^2_s)\}$ is a basis for $H_2(M, \chi\otimes\varrho_n^{p/q})$.
  \item $\{i_*( h_1^{p/q}\otimes (p_1m_1+q_1l_1)),\ldots,i_*( h_s^{p/q}\otimes (p_sm_s+q_sl_s))\}$ is a basis for $H_1(M, \chi\otimes\varrho_n^{p/q})$.
 \end{enumerate}
\end{lemma}

\begin{proof}
 We continue the argument in the proof of Lemma~\ref{lemma:cohomologyrhopq}. 
 The key fact is not only that the invariant subspace $(\CC^n) ^{ \chi\otimes\rho^{p/q}_n (T^2_ j) }$ has dimension 1 
 but that $\rho^{p/q}(\gamma_j)$ (and $\rho_n^{p/q}(\gamma_j)$) is semisimple \eqref{eqn:rhopq}, therefore  we have a decomposition 
 $$\CC^n=(\CC^n) ^{ \chi\otimes\rho^{p/q}_n (T^2_ j) }  \oplus \big( (\CC^n) ^{ \chi\otimes\rho^{p/q}_n (T^2_ j) } \big)^\perp
 $$
 that is invariant by the action of $\chi\otimes\rho^{p/q}_n (T^2_ j)$, where $\perp$ means orthogonal by the $\rho^{p/q}_n$-invariant bilinear form
 (as $\chi$ is one dimensional, neither $\chi$ nor $\overline\chi$ play any role in the orthogonal decomposition). It follows that
 the cohomology valued in $\big( (\CC^n) ^{ \chi\otimes\rho^{p/q}_n (T^2_ j) } \big)^\perp$ vanishes and that
 we have natural isomorphisms:
 $$
 H^*(T^2_j,\chi\otimes \varrho_ n^{p/q})\cong H^*(T^2_j,\CC)\otimes   (\CC^n) ^{ \chi\otimes\rho^{p/q}_n (T^2_ j) }
 $$
 and similarly for $V_ j$ and for homology. The lemma follows from this consideration and from Mayer-Vietoris isomorphism \eqref{eqn:isoMV}.
\end{proof}

Let us denote by $b_{p/q}^2$ and $b_{p/q}^1$ the basis obtained from Lemma~\ref{Lemma:basispq}, in dimension 2 and 1 respectively. 
By applying Milnor's formula on the torsion of a Mayer-Vietoris sequence, we have
(further details on the proof can be found in \cite[Lemmas~5.7 and 5.11]{MP14}):

\begin{proposition}

\label{prop:Dehhfillingformula}
\begin{enumerate}
                    \item 
 When $n$ is even or $\chi$ is nontrivial on each peripheral tori:
 \begin{equation*}
 \tor(M_{p/q},\chi\otimes\rho^{p/q}_n)=\tor(M,\chi\otimes\varrho^{p/q}_n)
 \prod_{j=1}^l\prod _{k=0}^{n-1} (e^{\frac{\lambda(\gamma_j)}{2}(n-1-2k)}\chi(m_j)-1).
 \end{equation*}

\item When $n$ is odd and $\chi$ is trivial precisely on the first $s$ peripheral tori:
 \begin{multline*}
 \tor(M_{p/q},\chi\otimes\rho^{p/q}_n)=\tor(M,\chi\otimes\varrho^{p/q}_n;  b^1_{p/q}, b^2_{p/q})
  \prod_{j=1}^s\prod _{\underset{2k\neq n-1} {k=0}}^{n-1} (e^{\frac{\lambda(\gamma_j)}{2}(n-1-2k)}-1)
  \\
 \times
 \prod_{j=s+1}^l\prod _{k=0}^{n-1} (e^{\frac{\lambda(\gamma_j)}{2}(n-1-2k)}\chi(m_j)-1).
 \end{multline*}
\end{enumerate}
where the complex length $\lambda(\gamma_j)$ is defined in (\ref{eqn:rhopq}).
\end{proposition}

We want to consider sequences of admissible Dehn fillings $M_{p/q}$
such that $(p,q)\to\infty$, hence by Lemma~\ref{lemma: seqconverge}  $\varrho^{p/q}\to\rho$. 
We need to take care of the basis in homology. Hence set $a_{p/q}^2=b_{p/q}^2$ and (with the notation of Lemma~\ref{Lemma:basispq})
$$
a^1_{p/q}=\{i_*( h_1^{p/q}\otimes l_1),\ldots,i_*( h_s^{p/q}\otimes l_s)\}.
$$

\begin{lemma}
\label{lemma:abasis}
 For $p_j^2+q_j^2$ sufficiently large, $j=1,\ldots, l$,  $a^1_{p/q}$ is a basis for $H^1(M,\chi\otimes \varrho_n^{p/q})$. 
\end{lemma}

\begin{proof}
 This is a semi-continuity argument: at the complete structure $\rho$ the dimension of the invariant subspace $(\CC^n)^{\chi\otimes \rho_n(\pi_1(T^2_ j))}$
 is one, and by semi-continuity in a neighborhood of $\rho$ in $\hom(\pi_1(M),\mathrm{SL}_2(\CC))$ 
 it cannot increase. The discussion in the proof of Lemma~\ref{lemma:cohomologyrhopq} provides a lower bound for the dimension,
 hence the dimension of the invariant subspace is one in a neighborhood
 of the complete holonomy $\rho$.
 Now we view the $a^i_{p/q}$ as a deformation of the elements $b^i$ constructed in Lemma~\ref{Lemma:basis}, by deforming the elements of the invariant subspaces 
 $(\CC^n)^{\chi\otimes \rho_n(\pi_1(T^2_ j))}$ when we deform the representation.  
  This yields a continuous family of cocycles that varies continuously with the representation, and 
 again by semi-continuity  those are linearly independent elements when projected to homology.
  See \cite[Section~5]{MP14} for further details.
\end{proof}

Recall that $n=2k+1$ is odd, and set $A_{2k+1}(p,q)$ the change of basis matrix from $a_{p/q}^1$ to $b_{p/q}^1$. Then, the formula of  change of basis \cite{Mil66} yields:
\begin{equation}
 \label{eqn:changebasis}
\tor(M,\chi\otimes\varrho^{p/q}_{2k+1},  b^1_{p/q},  b^2_{p/q})= \tor(M,\chi\otimes\varrho^{p/q}_{2k+1},  a^1_{p/q},  a^2_{p/q}) \det A_{2k+1}(p,q).
\end{equation}
By \cite[Lemma~5.13]{MP14}, 
\begin{equation}
 \label{eqn:A2k+1/A3}
\lim_{(p,q)\to\infty}\frac{\det A_{2 k+1}(p,q)}{\det A_{5}(p,q)}= 1,
\end{equation}
where $(p,q)\to\infty$ means $p_j^2+q_j^2\to\infty$, $j=1,\ldots,l$.

Finally, we use the convergence of representations:

\begin{lemma}
\label{lemma:continuity}
 \begin{enumerate}[(a)]
  \item When $n$ is even or $\chi$ is nontrivial on each peripheral torus:
  $$
  \lim_{(p,q)\to\infty} \tor(M,\chi\otimes\varrho_n^{p/q})=\tor(M, \chi\otimes\rho_n).
  $$
  \item When $n$ is odd and $\chi$ is trivial when restricted to some peripheral tori:
  $$
  \lim_{(p,q)\to\infty} \tor(M,\chi\otimes\varrho_n^{p/q}, a_{p/q}^1,a_{p/q}^2)=\tor(M, \chi\otimes\rho_n, b^1, b^2).
  $$
 \end{enumerate}
\end{lemma}

The proof of this lemma uses the discussion in Lemma~\ref{lemma:abasis}, as the $a_{p/q}^i$ are viewed as a deformation of the $b^i$, see
again \cite[Section~5]{MP14} for further details.
From Lemma~\ref{lemma:continuity}, \eqref{eqn:changebasis} and \eqref{eqn:A2k+1/A3} we deduce:

\begin{corollary}
\label{coro:quotients} \begin{enumerate}[(a)]
  \item When $n=2k$:
 $$
 \lim_{(p,q)\to\infty}\frac{  \tor(M,\chi\otimes\varrho_{2k}^{p/q})  }{  \tor(M,\chi\otimes\varrho_{4}^{p/q}) }
 =
 \frac{  \tor(M,\chi\otimes\rho_{2 k})  }{  \tor(M,\chi\otimes\rho_{4}) }.
 $$ 
    \item When $n= 2 k+1$:
$$
 \lim_{(p,q)\to\infty}\frac{  \tor(M,\chi\otimes\varrho_{2 k+1}^{p/q}, b_{p/q}^1,b_{p/q}^2)  }{  \tor(M,\chi\otimes\varrho_{5}^{p/q}, b_{p/q}^1,b_{p/q}^2) }
 =
 \frac{  \tor(M,\chi\otimes\rho_{2 k+1}, b^1,b^2)  }{  \tor(M,\chi\otimes\rho_{5}, b^1,b^2) }.
 $$
 \end{enumerate}
\end{corollary}

 \section{Analytic torsion}
 \label{sec:Analytic}

Let $N$ be a \emph{closed} orientable hyperbolic three-manifold, $\rho\colon\pi_1(N)\to \mathrm{SL}_2(\CC)$ a lift of its holonomy,
$\rho_n=\operatorname{Sym}^ {n-1}\circ\rho$ and  $\chi \colon \pi_1(N) \to \mS^1$ a group morphism.
The representation $\chi \otimes \rho_n$ defines a flat 
vector bundle $E_{\chi\otimes\rho_n} = \HH^3 \times_{\chi \otimes \rho_n} \mC^n$ over $N$.
Its de Rham cohomology (that is naturally isomorphic to $H^*(N, \chi\otimes\rho_n)$) vanishes by Corollary~\ref{Coro:injection}, because $N$ is closed. 

Since $\chi$ is unitary, $E_{\chi \otimes \rho_n}$ comes with a canonical hermitian metric $h$. 
For any $p$ we have the usual Hodge Laplacian 
$$
\Delta_{\chi,n}^p \colon \Omega^p(N, E_{\chi\otimes\rho_n}) \to \Omega^p(N, E_{\chi\otimes\rho_n}).
$$ 
The space $\Omega^p(N, E_{\chi\otimes\rho_n})$ is canonically isomorphic to the space of equivariant 
$p$-forms $\Omega^p(\mH^3, \mC^n )^{\pi_1(N)}$, and the Laplace operator acts equivariantly on 
$\Omega^p(\mH^3, \mC^n )$, so that  $\Delta^p$  can be thought as an equivariant operator $\widetilde{\Delta}^p$ on 
$ E_{\chi\otimes\rho_n}$-valued $p$-forms over
$\mH^3$ that descends to the quotient $N$.



%

\subsection{Heat Kernels}
\label{subsec:Heat}
In this section we consider the bundle of  differential forms on $N$ twisted by $\chi \otimes \rho_n$ and its Laplace operator.
We denote by $G = \SL_2(\mC)$, and by $K=\SU(2)$ its maximal compact subgroup.
The Lie algebra of $G$ splits as $$\mathfrak g = \mathfrak k \oplus \mathfrak p,$$ where $\mathfrak p$ 
is canonically identified with the tangent space of $\mH^3 \simeq G/K$ at the base point $[K]$, and the Killing form on $\mathfrak g$ restricts to a definite metric on $\mathfrak p$.

 We fix a Haar measure $dk$ on $K$ such that 
$$\int_K dk = 1$$
and let $dx$ be the standard hyperbolic metric on $\mH^3 = G/K$, normalized such that it coincides with the Killing form on $\mathfrak p$.
Finally we obtain a normalized Haar measure $dg$ on $G$ given by
$$\int_G f(g) dg = \int_{G/K} \int_Kf(gk) dkd(gK).$$
Recall that $\Sym^{n-1} \colon \SL_2(\mC) \to \SL_n(\mC)$ is the unique holomorphic $n$-dimensional representation of $G$. 
We consider  $$\Sym^{n-1}|_K\colon K \to \SL_n(\mC)$$ the restriction of $\Sym^{n-1}$ to $K$, and for any $p=0, \ldots, 3$, 
the action on $p$-forms on $K$:
$$\nu^p_n= \Lambda^p \Ad^* \otimes \Sym^{n-1}|_K\colon K \to \GL(\Lambda^p \mathfrak{p}^* \otimes \mC^n).$$
Finally, we consider the bundle of twisted $p$-forms on $\mH^3$
$$\widetilde E^p_n = G \times_{\nu^p_n} \Lambda^p \mathfrak p^* \otimes \mC^n$$
together with its Laplace operator 
$$\widetilde \Delta^p_n\colon L^2(\mH^3, \widetilde E^p_n) \to L^2(\mH^3, \widetilde E^p_n).$$

It is known as Kuga's lemma (see \cite[Section 4]{MuellerPfaff}) that this operator coincides with (a scalar plus) the Casimir operator. Moreover, for any $t>0$, the heat operator $e^{-t\widetilde \Delta^p_n}$ has a symmetric kernel 
$$\mathfrak H_n^p(t) \in C^\infty(G\times G, \operatorname{End}(\Lambda^p \mathfrak p^* \otimes \mC^n)).$$
In other words, for any $g,g'$ in $G$, it is a linear map
$$\mathfrak H_n^p(t,g,g') \colon \widetilde E^p_n\vert_{[gK]} \to \widetilde E^p_n\vert_{[g'K]}$$
with the following properties:
\benu
\item \label{item:Diagonal}
It is invariant under the diagonal action of $G$.
\item \label{item:Equivariant}
It is $K$-equivariant, more precisely for any $g,g'$ in $G$ and $k,k'$ in $K$, 
$$\mathfrak H_n^p(t,gk,g'k') = \nu^p_n(k^{-1}) \mathfrak H_n^p(t,g,g') \nu^p_n(k').$$

\item For any $\phi \in  L^2(\mH^3, \widetilde E^p(n))$, 
$$e^{-t\widetilde \Delta^p_n} \phi(g) = \int_G \mathfrak H_n^p(t,g,g') \phi(g') dg'$$
\eenu
Property~(\ref{item:Diagonal}) implies that there exists a \textit{convolution kernel}.
That is, a function $H_n^p(t) \colon G \to \operatorname{End}(\Lambda^p \mathfrak p^* \otimes \mC^n))$ such that 
$$\mathfrak H_n^p(t,g,g') = H_n^p(t,g^{-1}g')$$
and so that the heat operator $e^{-t\widetilde \Delta^p_n}$ acts by convolution:
$$e^{-t\widetilde \Delta^p_n} \phi(g) = \int_G H_n^p(t,g^{-1}g') \phi(g') dg'$$

\subsection{Twisted Laplacian}
\label{subsec:STF}
In what follows we denote by $\Gamma = \rho(\pi_1(N))$ the fundamental group of $N$ seen as a uniform lattice in $\SL_2(\mC)$.
We denote by $F = G/K \times_\chi \mC$ the associated line bundle on $N$, and by $E^p_n = \Gamma\backslash \widetilde E^p_n$ the bundle of $\rho_n$-twisted $p$-forms on $N$.
There is a canonical isomorphism of vector bundles
$$
\Lambda^p(E_{\chi\otimes\rho_n}) \simeq  F\otimes E^p_n.
$$
We also have a heat kernel for forms valued on this bundle:
\begin{proposition}
The heat operator $e^{-t\Delta_{\chi,n}^p}$ acting on the space $L^2(\Gamma \backslash G/K, F \otimes E^p_n)$ of $\chi \otimes \rho_n$-twisted $p$-forms on $N$ has a smooth kernel
\begin{equation}
\label{equa:Kernel}
H^p_{\chi,n}(t,\Gamma gK,\Gamma g'K') = \sum_{\gamma \in \Gamma} H_n^p(t,g^{-1} \gamma g) \otimes \chi(\gamma).
\end{equation}
Furthermore,  letting $h_n^p(t,g)$ denote $\Tr H_n^p(t,g)$, 
\begin{equation}
\label{equa:STF1}
\Tr (e^{-t\Delta_{\chi,n}^p}) = \int_{\Gamma \backslash G} \sum_{\gamma\in \Gamma} \chi(\gamma)h_n^p(t,g^{-1}\gamma g) dg.
\end{equation}

\end{proposition}
\begin{proof}
Formula (\ref{equa:Kernel}) is \cite[(5.6)]{Mueller11}. In particular, since the twist $\chi$ is unitary, 
with the usual arguments (for instance from \cite[Section 3]{MuellerPfaff}) one can show that the series on the right converges absolutely
and uniformly on compact sets. Then (\ref{equa:STF1}) follows immediately from the definition
$$\Tr(e^{-t\Delta^p_{\chi,n}}) = \int_{\Gamma \backslash G/K} \Tr H^p_{\chi,n}(t,\Gamma gK,\Gamma gK') dg,$$ from the $K$-equivariance of $H^p_{\chi,n}$ and from the fact that $\Vol(K) = 1$.
\end{proof}

%
Using (\ref{equa:STF1}) and that $\chi$ is  unitarity, one gets the usual estimates for the heat kernel. Namely, it
admits the following asymptotic expansion at $t=0$:
\begin{equation}
\label{equa:HKSmallTime}
\Tr (e^{-t\Delta^p_{\chi,n}}) \sim t^{-3/2}\sum_{k=0}^\infty a_k t^k
\end{equation}
with $a_0 = \Vol(N)$, and 
one has for $t\to \infty$:
\begin{equation}
\label{equa:HKLargeTime}
\Tr (e^{-t\Delta^p_{\chi,n}}) = O(e^{-\lambda(n)t})
\end{equation}
for some positive constant $\lambda(n)$.

Finally,  let $K_{\chi,n}(t)= \sum_{p=1}^3 (-1)^p p \Tr (e^{-t\Delta_{\chi,n}^p})$ be the alternating sum of the traces of the heat kernel, 
and $k_n(t)  = \sum_{p=1}^3 (-1)^p p \,h_n^p(t)$, then we immediately deduce from (\ref{equa:STF1}) the equation:
\begin{equation}
K_{\chi,n}(t) =  \int_{\Gamma \backslash G} \sum_{\gamma \in \Gamma}  \chi(\gamma) k_{n}(t,g^{-1}\gamma g)dg .
\end{equation}

 
\begin{definition}
\label{defi:AnalyticTorsion}
The analytic torsion is defined by
$$ \log T(N, \chi \otimes \rho_n) = \frac 1 2 \frac d{ds} \left( \frac 1{\Gamma(s)} \int_0^\infty t^{s-1} K_{\chi,n}(t) dt \right)\vert_{s=0}.$$
\end{definition}

The fact that this definition makes sense follows from estimates (\ref{equa:HKSmallTime}) and (\ref{equa:HKLargeTime}). Namely the integral on the right-hand side is convergent for $\operatorname{Re}(s) >\frac 3 2$, and it defines by analytic continuation a meromorphic function for $s$ in $\mC$, that turns out to be regular at $s=0$.

The representation $\rho_n$ has determinant one, and as $\chi$ has module one,  $\chi\otimes\rho_n$ is unimodular, 
hence we can apply Cheeger-M\"uller theorem, in the version of~\cite{Mul93}:

\begin{theorem}[Cheeger-M\"uller]
\label{theo:EquivCM}
For $N$ closed, 
the combinatorial and analytic torsions  coincide, that is 
$$\lvert \tor(N, \chi\otimes\rho_n)\rvert = T(N, E_{\chi\otimes\rho_n}).$$
\end{theorem}

\section{Ruelle  and Selberg zeta functions}
\label{subsec:Ruelle}

Along this section we assume that 
$N$ is a closed, orientable, hyperbolic 3-manifold, as in Section~\ref{sec:Analytic}.
The results of this section when $\chi$ is trivial are proved by M\"uller in \cite{Muller}.
We follow \cite{Muller} by adding the twist $\chi$ when needed.

Let
 $\mathcal{PC}(N)$ denote the set of prime closed oriented geodesics in $
N$. It is in bijection with the set of conjugacy classes of primitive elements in $\Gamma=\rho(\pi_1(N))\cong\pi_1(N)$
(recall from Section~\ref{sec:Analytic} that $\rho$ is a lift of the holonomy, so that $\Gamma $ is a uniform lattice in $\mathrm{SL}_2(\CC)$):
\begin{equation}
\label{eqn:PCCC}
\mathcal{PC}(N) \longleftrightarrow \left\{[\gamma]\in [\Gamma]=\Gamma\big/\textrm{conjugation} \mid \gamma\neq 1\textrm{ primitive}  \right\}. 
\end{equation}
For each element $\gamma\in \Gamma$, 
up to conjugation (in $\operatorname{SL}_2(\CC)  $),
\begin{equation}
\label{eqn:ltheta}
\rho(\gamma)\sim\begin{pmatrix}
              e^{\lambda(\gamma)/2} & 0 \\
              0 & e^{-\lambda(\gamma)/2}
             \end{pmatrix}
\end{equation}
where $\lambda(\gamma)= \ell(\gamma)+ i \theta(\gamma)$ is the complex length of the geodesic represented by $\gamma$.
In particular $\ell(\gamma)>0$ denotes the length of $\gamma$ and the parameter $\theta(\gamma)$ is determined modulo $4\pi i \ZZ$
($4\pi$ instead of $2\pi$ because we take care of the spin structure, equivalently, the lift to $\operatorname{SL}(2,\CC)$).

\subsection{Twisted Ruelle functions}
We consider Ruelle functions twisted by $\chi\otimes\rho_n$, where $\rho_n$ is the symmetric power of the lift of the holonomy $\rho$ and 
$\chi\colon\Gamma\to\mathbb{S}^ 1$ is a unitary twist.

\begin{definition} For $s\in \CC$ with $\operatorname{Re}(s)>2$, the \emph{twisted Ruelle zeta function} is
 $$
 {\mathcal R}_{\chi\otimes\rho_n}(s)=\prod_{[\gamma]\in  \mathcal{PC}(N)} \det( \operatorname{Id}-\chi(\gamma)\rho_n(\gamma) e^{-s \ell(\gamma)}).
 $$
\end{definition}

Convergence of $  {\mathcal R}_{\chi\otimes\rho_n}(s) $ for $\operatorname{Re}(s)>2$ follows from the following estimate of Margulis \cite{Margulis}:
\begin{equation}
\label{eqn:Margulis}
\#\{ [\gamma]\in  \mathcal{PC}(N)\mid \ell(\gamma)\leq L\} < C {e^{2L}/2L}.
\end{equation}
In fact \cite{Margulis} proves that for a manifold $N'$ of negative curvature:
\begin{equation}
\lim_{L\to\infty}\frac{\#\{ [\gamma]\in  \mathcal{PC}(N')\mid \ell(\gamma)\leq L\}} {e^{h\,L}/h\,L}=1,
\end{equation}
where $h$ is the topological entropy of the geodesic flow ($h=2$ for a hyperbolic $3$-manifold). 
We provide a weaker estimate in Lemma~\ref{lemma:PK}, that is uniform for the family of Dehn fillings we are considering.

The main theorem of this section is the following, whose proof is just an adaptation of \cite{Muller}, where it is shown for $\chi$ trivial. We delay the proof to Subsection \ref{subsec:WotzkeThm}.
\begin{theorem}[Fried's Theorem]
\label{Thm:Wotzke}
 The function $  {\mathcal R}_{\chi\otimes\rho_n} $  extends meromorphically to $\CC$, it is holomorphic at $s=0$ and
 $$\vert {\mathcal R}_{\chi\otimes \rho_n}(0) \vert  = T(N, E_{\chi\otimes\rho_n})^2.$$
\end{theorem}

\subsection{A functional equation}
\label{subsec:FunctionalEqua}
We consider another kind of Ruelle zeta functions.

\begin{definition} For $k\in\ZZ$, $\chi\colon\pi_1(N)\to\mathbb S^1 $, and  $s\in \CC$ with $\operatorname{Re}(s)>2$,
define
 $$
 R_{\chi,k}(s)=\prod_{[\gamma]\in  \mathcal{PC}(N)} ( 1-\chi(\gamma)e^{\frac k2 i \theta(\gamma)}  e^{-s \ell(\gamma)}).
 $$
 where $\ell(\gamma)$ and $\theta(\gamma)$ are as  in \eqref{eqn:ltheta}.
\end{definition}

The factor $e^{\frac k2 i \theta(\gamma)}$ is the character $\sigma_k$ in \cite{Muller}, that appears in Subsection~\ref{subsec:Representations} below. 
Again convergence for $\operatorname{Re}(s)>2$
follows from~\eqref{eqn:Margulis}.

The proof of the following lemma can be checked term-wise, cf.~\cite[(3.14)]{Muller}:
\begin{lemma}
\label{lem:zeta} For $s\in \CC$ with $\operatorname{Re}(s)>2$,
$${\mathcal R}_{\chi\otimes \rho_{n}}(s) = \prod\limits_{k=0}^n R_{\chi,n-2k}(s-(\frac n 2 -k)).$$
\end{lemma}

The following functional equation is a key ingredient of the main result:
\begin{proposition}
\label{prop:Functional}
The function $R_{\chi,k}$ extends meromorphically to the whole complex plane.
Moreover, it satisfies the functional equation:
$$\vert R_{\chi,k}(s) \vert  = e^{{4 \Vol(N) s}/{\pi}} \vert R_{\chi,-k}(-s)\vert .$$
\end{proposition}
\begin{proof}
The proposition will follow from a functional equation (\ref{equa:FunctionalEquaSelberg}) satisfied by 
the Selberg zeta function, defined as (see \cite[Definition 3.2]{BO}, \cite[(3.2)]{Muller}):
\begin{equation}
\label{equa:SelbergZeta}
Z_{\chi,k}(s) = \prod_{ {[\gamma]\in  \mathcal{PC}(N)} } 
\prod_{l=0}^\infty \det \left(1- e^{ik\theta(\gamma)/2} \chi(\gamma) \Sym^l(\Ad(\gamma)_{\bar{\mathfrak n}})e^{-(s+1)\ell(\gamma)}\right).
\end{equation}
It can be seen to converge on the half-plane $\operatorname{Re}(s)>1$, and it is proved in \cite[Section 3.3.1]{BO} that this function admits a meromorphic continuation on the whole complex plane $\mC$. Moreover, it satisfies a functional equation
\begin{equation}
\label{equa:FunctionalEquaSelberg}
| Z_{\chi,k}(s) | = \exp\left( -4\pi \int_0^s P_k(r) dr \right) |Z_{\chi, -k}(-s) |  ,
\end{equation}
where $P_k(z) = \frac 1{4\pi}\left(\frac {k^2} 4 -z^2\right)$ is the Plancherel polynomial (see \cite[(3.4)]{Muller}).
This functional equation is proved in \cite[Section 4]{Muller} in the case of a trivial twist, and in \cite[Section 3.3.2]{BO} in the general case, although 
\cite{BO} has a different choice of Plancherel polynomials.

Now we have the following relation between Ruelle zeta functions $R_{\chi,k}$ and Selberg zeta functions $Z_{\chi,k}$, whose proof is a straightforward adaptation of \cite[Lemma 3.1]{Muller}:
\begin{equation}
\label{equa:RelSelRuel}
R_{\chi,k}(s) = \frac{Z_{\chi,k}(s+1) Z_{\chi,k}(s-1)}{Z_{\chi,k+2}(s)Z_{\chi,k-2}(s)}
\end{equation}
and Proposition \ref{prop:Functional} follows now from a direct computation inserting (\ref{equa:FunctionalEquaSelberg}) in (\ref{equa:RelSelRuel}) exactly as in \cite[Proof of Proposition 3.2]{Muller}.
\end{proof}

\subsection{Proof of Theorem \ref{Thm:Wotzke}}
\label{subsec:WotzkeThm}
In this section we prove Theorem \ref{Thm:Wotzke}, adapting the proof of \cite{Muller} to the case of a non-trivial twist $\chi \colon \pi_1(N) \to \mS^1$. 
We denote by~$[\Gamma]$ the set of conjugacy classes of elements of $\Gamma \cong \pi_1(N)$, and by
$[\Gamma_+]=~[\Gamma]\setminus~\{e\}$ the set of non-trivial ones.

\subsubsection{Lie Groups}
Recall that $G = \SL_2(\mC) = KAN$, where 
$$K=\SU(2), \quad A = \lbrace \exp(\lambda H) | \lambda \in \mR, \, H =  \bsm 1&0\\0&-1\esm  \rbrace, \quad N = \lbrace \bsm 1&z\\0&1\esm | \ z \in \mC\rbrace.$$
We denote by $M$ the centralizer of $A$ in $K$, that is, 
$$
M = \left\{ \bsm e^{i\theta}&0\\0&e^{-i\theta}\esm | \ \theta \in [0,2\pi] \right\}.
$$
Let $W_A$ denote the Weyl group of $A$, whose only non trivial element acts as $\omega_A \cdot H = -H$. 
There is also $W_G$, the Weyl group of $G$, which is isomorphic to the Klein group, acting on $\mC^2$ (the complexification of the Cartan algebra) by 
changing the sign of its coordinates, corresponding to the respective Lie algebras of $A$ and $M$.
 
The Lie algebra of $G$ is $\mathfrak g = \mathfrak{sl}_2(\CC)  =\{ X \in M_2(\mC) | \Tr X =0 \}$. A natural $\CC$-basis is given by 
$$H =  \bsm 1&0\\0&-1\esm, \ E = \bsm 0&1\\0&0 \esm, \ F = \bsm 0&0\\1&0 \esm.$$
\subsubsection{Representations}
\label{subsec:Representations}
Let $\widehat M$ denote the set of non equivalent irreducible representations of $M$, we have
$$\widehat M = \{\sigma_k | \ k\in \mZ \},$$
where $\sigma_k\bsm e^{i\theta}&0\\0&e^{-i\theta}\esm = e^{ik\theta}$. Naturally $W_A$ acts on $\widehat M$ by $\omega_A \cdot \sigma_k = \sigma_{-k}$.

For any $k$, the representation $\sigma_k + \sigma_{-k}$ is the restriction to $M$ of a unique (formal sum of) representation $\xi_k$ of $K$. 
More precisely, for any nonnegative integer $l$, denote by $\nu_l\colon K \to \SL_{l+1}(\mC)$ the restriction of the $l$-th symmetric power of $G$, then we have 
$$
  \nu_l |_M = \bigoplus_{a=0}^l \sigma_{l-2a},
$$
so that for any $k\ge 2$, one gets $\xi_k = \nu_k - \nu_{k-2}$,  $\xi_1 = \nu_1$ and  $\xi_0 = 2\nu_0$.
We summarize by denoting 
$$\xi_k = \sum_{\nu \in \widehat K} m_\nu(k) \nu$$
with $m_\nu(k) \in \{0, \pm1\}$ except for $m_{\nu_0}(0)=2$.

Finally, there is a family of irreducible representations of $MAN$ that will be important in the sequel, called the unitary principal series.
They are defined as follows:
$$\pi_{k,\lambda}  = \sigma_k \otimes e^{i\lambda} \otimes 1 \colon MAN \to \op{End}{\mathcal H_{k,\lambda}}  ,$$
where $H_{k,\lambda}$ is the Hilbert completion of the space 
$$H_{k,\lambda}^\infty = \{ f \in C^\infty(G,\mC_{\sigma_{k}}) | \ f(gman) = e^{-(i\lambda+1)(\log a)} \sigma_k(m)^{-1}f(g)\}.$$

\subsubsection{Bochner--Laplace operators}
\label{sec:Bochner}
Given $\nu \in \widehat K$, let $V_\nu$ the corresponding representation and let $\widetilde E_\nu = (G\times V_\nu)/K$ denote 
the induced vector bundle on $G/K = \mH^3$. Let $E_\nu = \Gamma \backslash \widetilde E_\nu$ be the corresponding vector bundle on $N$. 
Moreover we denote by $F_\chi = \mH^3 \times_\chi \mC_\chi$ the flat vector bundle induced by $\chi$, and the bundle twisted by $\chi\otimes\nu$ is
isomorphic to $E_{\chi,\nu} = F_\chi \otimes E_\nu $.
Then, for any integer $k$, we define
$$E_{\chi,k} = \bigoplus_{\nu |m_\nu(k) \neq 0} E_{\chi, \nu}.$$
The sign of $m_\nu(k)$ provides a natural grading yielding the decomposition
$$E_{\chi,k} = E_{\chi,k}^+ \oplus E_{\chi,k}^-.$$

Recall that the Casimir operator $\Omega$ lies in the center $Z(\mathfrak g)$ of the universal enveloping algebra of~$\mathfrak g$. Using the Killing form to identify $\mathfrak g$ with its dual $\mathfrak g^*$, the Casimir operator can be written as 
$$HH^* + EE^* +FF^* = \frac {H^2} 2+ EF+FE.$$
It induces a $G$-invariant second order differential operator on $C^\infty(G)$, and we denote by $\widetilde A_\nu$ the induced operator on $C^\infty(G/K, \widetilde E_{\nu})$ induced by $-\Omega$.

There is a canonical connection $\nabla^\nu$ on $\widetilde E_\nu$, so that the connection Laplacian operator is defined as
$$\widetilde \Delta_\nu = (\nabla^\nu)^* \nabla_\nu, $$
it defines a $G$-invariant second order differential operator satisfying
$$\widetilde A_\nu = \widetilde \Delta_\nu - \nu(\Omega_K)$$
where $\Omega_K$ is the Casimir element of $K$ \cite[(4.7)]{Muller}.

Now $\widetilde A_\nu$ defines an operator,  symmetric and bounded from below,
$$A_{\chi, \nu} \colon C^\infty(N, E_{\chi, \nu}) \to C^\infty(N, E_{\chi, \nu}).$$
For any integer $k$, we define $c(k) = \frac{k^2}4-1$ and the operator 
\begin{equation}
\label{equa:BochnerLaplace}
A_{\chi,k} = \bigoplus_{\nu | m_\nu(k) \neq 0} A_{\chi, \nu} + c(k),
\end{equation}
that preserves the grading of the bundle $E_{k,\chi}$.

Let $H_t^\nu$ be the kernel of $e^{-t\widetilde \Delta_\nu}$, $h_t^\nu = \Tr H_t^\nu$ and 
$$h_t^k (g)= \sum_{\nu | \ m_\nu(k) \neq 0} m_\nu(k) h_t^\nu(g).$$
Recall that 
$$
  P_l(z) = \frac 1 {4\pi^2}\Big(\frac{l^2}4-z^2\Big)
$$
is the Plancherel polynomial of the representation $\sigma_l$ (see \cite[(3.4)]{Muller} or \eqref{equa:FunctionalEquaSelberg}).

%
Define  $\Theta_{l,\lambda} = \Tr \pi_{l,\lambda}$ to be the trace of the corresponding irreducible representation of the unitary principal series, 
and  
the Weyl denominator
\begin{equation}
 \label{eqn:Weylden}
D(\gamma) = e^{\ell(\gamma)} |\det(1- \Ad(m_\gamma a_\gamma)_{\overline{\mathfrak n}})|
. \end{equation}
%
%
%
By $\Tr_s$ we denote the supertrace of the corresponding operator with respect to the grading of the corresponding vector bundle (here $E_k)$. Namely: 
$$
\Tr_s T = \Tr T |_{E_k^+} -\Tr T |_{E_k^-}.
$$
Using the Selberg trace formula, the supertrace of the heat kernel $e^{-tA_{\chi,\nu}}$ is computed in \cite[(4.18)]{Muller}:
\begin{theorem}
\label{theo:TraceBochner}
We have
\begin{multline*}
\Tr_s (e^{-tA_{\chi,k}}) = \Vol(N) \sum_{l \in \mZ} \int_\mR P_l(i\lambda)\Theta_{l,\lambda}(h_t^k) d\lambda\\
+\sum_{[\gamma]\in [\Gamma_+]} \frac{ \chi(\gamma)\ell(\gamma_0)}{2\pi D(\gamma)q} \sum_{l \in \mZ} \sigma_{-l}(\gamma) \int_\mR \Theta_{l, \lambda}(h_t^k)e^{-\ell(\gamma)\lambda}d\lambda
\end{multline*}
\end{theorem}
\begin{proof}
From the analogous of (\ref{equa:STF1}) in the present situation one gets:
\begin{equation}
\Tr_s (e^{-tA_{\chi,\nu}}) = \int_{\Gamma \backslash G} \sum_{\gamma \in \Gamma} \chi(\gamma) \Tr_s(e^{-t\widetilde A_\nu})(g,\gamma g) dg .
\end{equation}
Now we regroup the summands by conjugacy class, and use the classical computations of orbital integral, see \cite{Muller, Wallach}.
\end{proof}

By \cite[Section 4]{MuellerPfaff} we have 
$$\Theta_{l,\lambda} (h_t^k) = \delta_{| l |, k} e^{-t\lambda^2}.
$$
Moreover 
\begin{align*}
\int_\mR \frac{l^2/4+\lambda^2}{4\pi^2}e^{-t\lambda^2} d\lambda &= \frac 1{4\pi\sqrt{\pi t}} (\frac{l^2}4+\frac 1{2t}), \\
\frac{1}{2\pi}\int_\mR e^{-t\lambda^2}e^{-\ell(\gamma)\lambda} d\lambda &= \frac{e^{-\frac{\ell(\gamma)^2}{4t}}}{\sqrt{4\pi t}},
\end{align*}
so that we get for any $k \neq 0$:
\begin{multline}
\label{equa:TraceBochner}
\Tr_s(e^{-tA_{\chi,k}}) =\frac{ \Vol(N) }{2\sqrt{\pi t}}(\frac{k^2}4+\frac 1{2t}) 
\\
+\sum_{[\gamma]\in [\Gamma_+]} \frac{\chi(\gamma) \ell(\gamma_0)}{D(\gamma)q} (e^{ik\theta(\gamma)/2} + e^{-ik\theta(\gamma)/2})  \frac{e^{-\frac{\ell(\gamma)^2}{4t}}}{\sqrt{4\pi t}}.
\end{multline}

\subsubsection{The determinant formula}
We follow \cite[Section 6]{Muller} to express in Proposition \ref{prop:DetSelberg} the (symmetrized) Selberg zeta functions as the graded determinant of the operators defined in Section \ref{sec:Bochner}. Then we deduce a similar determinant formula for the Ruelle zeta function in Proposition \ref{prop:RuelleDet}. It will be the main tool to prove Fried's theorem.

We introduce the symmetrized Selberg zeta function: for $k \neq 0$ it is defined by
$$S_{\chi,k}(s) = Z_{\chi,k}(s)  Z_{\chi,k}(s).$$
The notion of graded determinant is the analogous of the notion of supertrace for elliptic self-adjoint operators.
Consider the following xi function defined, for $\operatorname{Re}(z)$ and $\operatorname{Re}(s)$ large enough, by the formula
\begin{equation}
\label{equa:xi_function}
\xi_k(z,s) = \int_0^\infty e^{-ts^2}\Tr_s(e^{-tA_{\chi,k}})t^{z-1} dt
\end{equation}
A reference for the following is \cite[Appendix A]{Wotzke}.
This function admits a meromorphic extension in the variable $z$ to the whole complex plane, which is differentiable in the variable $s$. Moreover, we have:
\begin{lemma}
\label{lem:Xizero}
The evaluation at $z=0$ of the xi function gives:
$$\xi_k(0,s) = -\log \det{}_{\mathrm{gr}}(A_{\chi,k}+s^2).$$
\end{lemma}
\begin{proof}
In \cite[Appendix A]{Wotzke}, it is proved in Satz~A.15 that for any $k$, the $\xi$-function can be written as
$$\xi_k(z,s) = \frac{a_{-1}^k(0,s)}{z} + a_0^k(0,s) + O(z),$$
but the explicit computation  $a_{-1}^k(0,s)=\sum\limits_{\alpha+i=0} (-1)^i \frac{p^i}{i !} c_\alpha$ given 
there vanishes in our case, since $\alpha$ is an half-integer and $i$ is an integer. So $\xi_k$ is holomorphic at $z=0$, and by \cite[Korollar A.16]{Wotzke}  (or a direct computation) we obtain
\begin{equation}
\label{equa:XiDet}
\frac d{dz} \frac{\xi_k(z,s)}{\Gamma(z)} |_{z=0} = \xi_k(0,s)
\end{equation}
and the left hand side of (\ref{equa:XiDet}) is the definition of $-\log \det{}_{\mathrm{gr}}(A_{\chi,k}+s^2)$ (notice that $\frac{\xi_k(z,s)}{\Gamma(z)}$ is a two variable zeta function).
\end{proof}

A straightforward computation gives
\begin{equation}
\label{equa:XiDeriv}
-\frac{1}{2s} \frac d{ds}\xi(z,s) = \int_0^\infty e^{-ts^2}\Tr_s(e^{-tA_{\chi,k}})t^zdt .
\end{equation}
Combining Lemma \ref{lem:Xizero} with (\ref{equa:XiDeriv}), one obtains
\begin{multline*}
\frac 1 {2s} \frac d{ds} \det{}_{\mathrm{gr}}(A_{\chi,k}+s^2) -\frac 1 {2s_0} \frac d{ds} \det{}_{\mathrm{gr}}(A_{\chi,k}+s^2) |_{s=s_0}
\\
=\int_0^\infty(e^{-ts^2}-e^{-ts_0^2})\Tr_s(e^{-tA_{\chi,k}}) dt .
\end{multline*}
Now using \cite[(5.3) and (5.7)]{Muller} it turns into
$$ \frac d{ds} \det{}_{\mathrm{gr}}(A_{\chi,k}+s^2) = \frac d{ds} \log S_{k,\chi}(s) + 4\pi \Vol(N) P_k(s) + bs$$
where the constant $b$ can be deduced from the previous equation.
Integrating this equation yields
$$\log S_{k,\chi}(s) = \log \det{}_{\mathrm{gr}}(A_{\chi,k} +s^2) -4\pi \Vol(N) \int_0^s P_k(r) dr +\frac b 2 s^2+c$$
and a computation in \cite{Muller} shows that $b=c=0$, hence:
\begin{proposition}
\label{prop:DetSelberg}
If $k\neq 0$, 
$$S_{k,\chi}(s) = \det{}_{\mathrm{gr}}(A_{\chi,k}+s^2) \exp\left(-4\pi \Vol(N) \int_0^s P_k(r) dr \right).$$
\end{proposition}

We next aim to write the Ruelle function as a product of Selberg functions.
In order to avoid   subscripts, we denote by $\tau$ the irreducible, $(n+1)-$dimensional, holomorphic representation of $G$:
$$\tau = \Sym^n \colon G \to \SL_{n+1}(\mC).$$
Following \cite[Section~3]{Muller} we decompose the character of the restriction  $\tau |_{MA}$ as a sum (indexed by the Weyl group  $ W_G$) of
terms in 
$\sigma_k \colon M \to \mS^1$ and $e^{\lambda  \alpha} \colon A \to \mR, a \mapsto e^{\lambda \alpha(\log a)}$ 
where $\alpha$ is the unique root of $G$ (dual to $H$), for genuine $k=k_{\tau, \omega}$ and $\lambda= \lambda_{\tau, \omega}$.
The highest weight of $  \Sym^{n_1}\otimes \overline\Sym^{n_2}$ is $(n_1,n_2)\in\ZZ_{\geq 0}\times \ZZ_{\geq 0}$,
in particular the highest weight of $\tau$ is $\Lambda_\tau= (n,0)$.
Following \cite[(3.16)]{Muller}, 
to each pair of integers $(n_1,n_2)\in \ZZ\times\ZZ$ we consider the  character $\varsigma_{(n_1,n_2)}\colon M A\to\CC^*$ defined by
$$
\varsigma_{(n_1,n_2)}= \sigma_{n_1-n_2}\otimes e^{\frac{n_1+n_2}{2}\alpha}.
$$
Let
$\mu_p\colon MA\to \mathrm{GL}(\bigwedge^p \mathfrak{n}_\CC) $ be the $p$-th exterior power of the adjoint representation of $MA$ on  $\mathfrak{n}_\CC$, the complexification of the Lie algebra of $N$.
Let $\tilde\mu_p$ be the contragredient representation of $\mu_p$. Then
\begin{equation}
\label{eqn:WeilG}
\sum_{p=0}^2(-1)^p \operatorname{tr}  \tilde\mu_p \operatorname{tr} \tau\vert_{MA} =\sum_{\omega\in W_G} (-1)^{\ell(\omega)} \varsigma_{\omega(\Lambda_\tau+(1,1))-(1,1)}
 \end{equation}
where the Weyl length $\ell(\omega)$ is also the number of sign changes by $\omega$.
Formula~\eqref{eqn:WeilG} is \cite[Lemma~3.3]{Muller},
though it can be established by direct computation,
using that  $\mu_0=\sigma_0$, $\mu_1=(\sigma_2\otimes e^\alpha)\oplus (\sigma_{-2}\otimes e^{\alpha})$, and $\mu_2=\sigma_0\otimes e^{2\alpha}$ 
and $\Lambda_\tau=(n,0)$. 

Define $\sigma_{\tau, \omega}$ and $\lambda_{\tau,\omega}$ by:
$$
 \varsigma_{\omega(\Lambda_\tau+(1,1))-(1,1)}=\sigma_{\tau, \omega}\otimes e^{(\lambda_{\tau,\omega}-1)\alpha}
$$
Then relating 
$
\sum_{p=0}^2(-1)^p \operatorname{tr}  \tilde\mu_p 
$ 
to a Weyl denominator \eqref{eqn:Weylden} as in  \cite[(3.24)]{Muller}, from \eqref{eqn:WeilG} we have
\begin{equation}
\label{equa:sumWeil}
\Tr \tau(ma) = \sum_{\omega \in W_G} (-1)^{\ell(\omega)} \frac{\sigma_{\tau, \omega}(m)}{\det(1-\Ad(ma)_{\overline {\mathfrak n}})}e^{(\lambda_{\tau,\omega}-1)\alpha(\log a)} .
\end{equation}
We denote the Ruelle zeta function associated to $\tau$ and $\chi$ by
$${\mathcal R}_{\chi, \tau}(s) ={\mathcal R}_{\chi\otimes \tau}(s) ={\mathcal R}_{\chi \otimes \varrho _{n+1}}(s)$$
Using (\ref{equa:sumWeil}) as in \cite[Proposition 3.4 and Proposition 3.5]{Muller} one gets:
\begin{proposition}
\label{prop:RuelleSelberg}
We have 
$${\mathcal R}_{\chi, \tau}(s) = \prod_{\omega \in W_G} Z_{k_{\tau,\omega},\chi}(s-\lambda_{\tau,\omega})^{(-1)^{\ell(\omega)}}$$
and
$${\mathcal R}_{\chi, \tau}(s) {\mathcal R}_{\chi,\overline \tau}(s) = \prod_{\omega \in W_G} S_{k_{\tau,\omega},\chi}(s-\lambda_{\tau,\omega})^{(-1)^{\ell(\omega)}}$$
\end{proposition}

Now for each $\omega$ in $W_G$ we define the operator
\begin{equation}
\label{equa:DeltaOm}
\Delta(\omega) = \bigoplus_{\nu | \ m_\nu(k_{\tau ,\omega}) \neq 0} A_{\chi,\nu} +\tau(\Omega)
\end{equation}
where $\tau(\Omega)$ is the Casimir eigenvalue given by (\cite[(6.16)]{Muller}):
\begin{equation}
\label{equa:CasimirEig}
\tau(\Omega) = \frac{n(n+2)} 2 = \lambda_{\tau, \omega}^2 +c(k_{\tau, \omega}).
\end{equation}
It follows from the definition of $A(k_{\tau, \omega})$  in (\ref{equa:BochnerLaplace}), from (\ref{equa:DeltaOm}) and (\ref{equa:CasimirEig}) that 
\begin{equation}
\label{equa:DeltaBochner}
A_{\chi,k_{\tau, \omega}} +\lambda_{\tau, \omega}^2 = \Delta(\omega) ,
\end{equation}
so that Proposition \ref{prop:DetSelberg} turns into
\begin{equation}
\label{equa:DetOmega}
S_{\chi,k_{\tau, \omega}}(s-\lambda_{\tau, \omega}) = \det{}_{\mathrm{gr}}(s^2-2\lambda_{\tau, \omega} s+\Delta(\omega)) \exp\left(-4\pi \Vol(N) \int_0^{s-\lambda_{\tau, \omega}} P_{k_{\tau, \omega}}(r) dr \right).
\end{equation}
Now a direct computation involving Proposition \ref{prop:RuelleSelberg} and (\ref{equa:DetOmega}) yields, as in \cite[Proposition~6.2]{Muller}:
\begin{proposition}
\label{prop:RuelleDet}
For $\tau= \Sym^n \colon G \to \SL_{n+1}(\mC)$, we have
$${\mathcal R}_{\tau, \chi}(s) {\mathcal R}_{\overline \tau, \chi}(s)  = e^{-\frac{4(n+1)\Vol(N)s}{\pi}}\prod_{\omega \in W_G} \det{}_{\mathrm{gr}}(s^2-2\lambda_{\tau, \omega}s+\Delta(\omega))^{(-1)^{\ell(\omega)}}.$$
\end{proposition}

To prove Fried's theorem we will need the fact that the operators $\Delta(\omega)$ are positive for each $\omega$ in $W_G$.
\begin{lemma}
\label{lem:KernelD}
For any $\omega$ in $W_G$, we have $\Delta(\omega)>0$. In particular the right-hand side of the equality in Proposition \ref{prop:RuelleDet} converges as $s$ goes to $0$ to $\prod_{\omega \in W_G} \det{}_{\mathrm{gr}} (\Delta(\omega))$, where $\det{}_{\mathrm{gr}} (\Delta(\omega))= \frac d {ds}\left(\frac 1 {\Gamma(s)} \int_0^\infty \Tr_s(e^{-t\Delta(\omega)})t^{s-1} dt\right)|_{s=0}$.
\end{lemma}

\begin{proof}
The second statement follows directly from the first and from Lemma~\ref{lem:Xizero}. We prove that $\Delta(\omega)>0$. We will show that it is a consequence of the vanishing of the kernel of the Hodge Laplacians $\Delta_{\chi,n}^p$, see  Theorem \ref{Thm:vanishingL2}. 

First, we can express the Hodge Laplacian as a direct sum of Bochner-Laplace operators on $\mH^3$ (see \cite[(5.7)]{MuellerPfaff}) (recall that $\tau$ is $(n+1)$-dimensional):
\begin{equation*}
\widetilde \Delta_{n+1}^p = \bigoplus_{\substack{\nu \in \widehat K \\ [\nu_{n+1}^p\colon \nu] \neq 0}} \widetilde \Delta_\nu + \left(\tau(\Omega) - \nu(\Omega_K)\right) \textrm{Id}
\end{equation*}
From this equation and the fact that for any $p$, one has $\Delta_{\chi,n+1}^p >0$, we deduce that 
$A_{\chi,\nu} = \Delta_{\chi,\nu} - \nu(\Omega_K) > - \tau(\Omega)$ for any $\nu$ such that $[\nu_{n+1}^p\colon \nu] \neq 0$.
Inserting that in (\ref{equa:DeltaOm}) it follows that $\Delta(\omega) >0$, as claimed.
\end{proof}
\subsubsection{Fried's theorem with non-trivial twist}

The key point is the use of the Selberg trace formula for the heat kernel of the twisted Hodge--Laplace operators introduced in Section \ref{subsec:STF}. 

As in Section \ref{sec:Bochner}, we obtain, with the notation of Section \ref{sec:Analytic}:
\begin{multline}
\label{equa:STFLaplace}
K_{\chi,n}(t) = \Vol(N) \sum_{l \in \mZ} \int_\mR P_l(i\lambda)\Theta_{l,\lambda}(k_n(t)) d\lambda \\
 + \sum_{[\gamma] \in [\Gamma_+]} \frac{\chi(\gamma)\ell(\gamma_0)}{2\pi D(\gamma)q} \sum_{l \in \mZ} \sigma_{-l}(\gamma) \int_\mR \Theta_{l, \lambda}(k_n(t))e^{-\ell(\gamma)\lambda}d\lambda .
\end{multline}
Comparing term-wise the contributions of the STF applied to $e^{-t\Delta_{\chi,n}^p}$ and to $e^{-tA_{\chi,k}}$ and using (\ref{equa:DeltaBochner}), we obtain as in \cite[(7.21)]{Muller}:
\begin{equation}
\label{equa:STF}
K_{\chi,n}(t) = \frac 1 2 \sum_{\omega \in W_G} (-1)^{\ell(\omega)+1} \Tr_s(e^{-t\Delta(\omega)}) .
\end{equation}
Now we take Mellin transforms in both sides of (\ref{equa:STF}) and we obtain
\begin{equation}
\label{equa:Mellin}
\frac 1{\Gamma(s)} \int_0^\infty K_{\chi,n}(t) t^{s-1} dt = \frac 1 2 \sum_{\omega \in W_G} (-1)^{\ell(\omega)+1}\frac 1 {\Gamma(s)} \int_0^\infty \Tr_s(e^{-t\Delta(\omega)})t^{s-1} dt .
\end{equation}
Deriving the left hand side at $s=0$, multiplying by two and taking exponential, one gets a power of the analytic torsion $T(N, E_{\chi \otimes \varrho_n})^4$ while the right hand side yields
$$
\prod_{\omega \in W_G} \det{}_{\mathrm{gr}}(\Delta(\omega))^{(-1)^{\ell(\omega)}}.
$$
Finally taking the limit when $s$ goes to zero in Proposition \ref{prop:RuelleDet} (see Lemma~\ref{lem:KernelD}) and using the symmetry properties of the functions $Z_k(s)$ and ${\mathcal R}_{\chi \otimes \varrho_n}(s)$ at 
$0$ as in \cite[(7.29) and (7.30)]{Muller}, we obtain:
$$
| {\mathcal R}_{\chi\otimes\varrho_n}(0) |^2 = \prod_{\omega \in W_G} \det{}_{\mathrm{gr}}(\Delta(\omega))^{(-1)^{\ell(\omega)}} .
$$
so that Theorem \ref{Thm:Wotzke} follows. \qed

\section{Approximation by Dehn fillings}
\label{section:approx}

In this section we describe the geometric convergence of Dehn fillings $M_{p/q}$ to $M$, focusing on the behavior of geodesics and 
their role in Ruelle functions.

\subsection{Geometric convergence}
For a sequence of compatible Dehn fillings $M_{p/q}$ such that $(p,q)\to\infty$,
not only we have convergence of representations  $[\varrho_n^{p/q}]\to[\rho_n]$, but we have also pointed bi-Lipschitz convergence,
see Thurston's notes \cite{Thu97} or \cite{Gromov_Bourb}:
\begin{theorem}[Thurston]
\label{Thm:bilipschitz}
Given $\epsilon>0$ and $\delta> 0$, there exists $C>0$ such that, if $ p_j^2+q_j^2>C(\epsilon,\delta)$ for $i=1,\ldots,l$, 
then there is a $(1+\epsilon)$-bi-Lipschitz homeomorphism of the $\delta$-thick parts
$M^{[\delta,+\infty)} \to M_{p/q}^{[\delta,+\infty)}$.
 \end{theorem}

The $\delta$-thick part of $N$ is defined as
$$
 N^{[\delta,+\infty)}=\{ x\in  N\mid \operatorname{inj}(x)\geq \delta\},
$$
where $\operatorname{inj}(x)$ denotes the injectivity radius of $x$.

Let $\gamma_1^ {p_1/q_1},\ldots ,\gamma_l^ {p_l/q_l}$ denote the souls of the filling solid tori of $M_{p/q}$, whose length converges to zero as $(p,q)\to\infty$. 

\begin{proposition}
\label{prop:geodesicsthick}
Except for $\gamma_1^ {p_1/q_1},\ldots ,\gamma_l^ {p_l/q_l}$, all primitive closed geodesics of $M_{p/q}$ must intersect the $\delta$-thick part, 
provided that  $0<\delta<\delta_0$  for a $\delta_0>0$ depending only on $M$. 
\end{proposition}

\begin{proof}
 By Margulis lemma,
using the thin-thick decomposition (see \cite{Thu97} again) and taking $\delta_0>0$ less than half the length of the shortest geodesic of $M$,
 $M\setminus M^ {[\delta,+\infty)}$ is the union of cusp neighborhoods. Therefore, by choosing $\delta_0$ even less, 
by Theorem~\ref{Thm:bilipschitz}
  $M_{p/q}\setminus M_{p/q}^{[\delta,+\infty)}$ is the union 
of Margulis tubes around the geodesics  $\gamma_1^{p_1/q_1},\ldots ,\gamma_l^{p_l/q_l}$. Then the proposition holds true because 
Margulis tubes contain no closed geodesics  other than their souls.
\end{proof}

As the diameter of $M^{[\delta,+\infty)}$ goes to infinity when $\delta\to 0$, by Proposition~\ref{prop:geodesicsthick} and geometric convergence we have:

\begin{proposition}
\label{prop:geodesicsconv}
 For given $L>0 $  there exists a constant $C(L)>0$ such that if   $ p_j^2+q_j^2>~C(L)$,  for $i=1\ldots,l$, then the inclusion induces a bijection:
$$
 \{[\gamma]\in \mathcal{PC}(M)\mid \ell(\gamma)\leq L \}\longleftrightarrow  \{[\gamma]\in \mathcal{PC}(M_{p/q})\mid \ell(\gamma)\leq L,\ \gamma\neq (\gamma_i^ {p_i/q_i})^{\pm 1} \}.
$$
As the length of the $\gamma_i$ converges to zero, the inclusion induces  a bijection:
$$
 \{[\gamma]\in \mathcal{PC}(M)\mid  \ell(\gamma)\leq L \}\longleftrightarrow  \{[\gamma]\in \mathcal{PC}(M_{p/q})\mid  \frac1L\leq \ell(\gamma)\leq L \}.
$$
Furthermore, the length and the holonomy of each geodesic in $\mathcal{PC}(M)$ is the limit of length and holonomy of the corresponding
geodesic in $M_{p/q}$ as $(p,q)\to\infty$.
\end{proposition}

See \cite[Section 6.3 and 6.4]{MP14} for a detailed proof, for instance. Another consequence of bi-Lipschitz convergence is a uniform estimate on the growth of geodesics.
We next quote Lemma~6.3 from \cite{MP14}, based on \cite{CoornaertKnieper}:

\begin{lemma}
 \label{lemma:PK}
Let $X$ be a complete hyperbolic 3-manifold. For a compact domain $K\subset X$, 
$$
 \#\{
 [\gamma]\in \mathcal{PC}(X)\mid \gamma\cap K\neq\emptyset,\ \ell(\gamma)\leq L
 \}
 \leq C e^{2L},
$$
with $C=\pi \, e^{8\operatorname{diam} (K)}/\operatorname{vol}(K)$.
\end{lemma}

This is not the best estimate, for instance \eqref{eqn:Margulis} due to Margulis \cite{Margulis} is better, see also 
\cite{CoornaertKnieper}, but Lemma~\ref{lemma:PK} provides a uniform bound for the family of Dehn fillings. From Proposition~\ref{prop:geodesicsthick}
and Theorem~\ref{Thm:bilipschitz}, by taking $K=M^ {[\delta,+\infty)}$ or  $K=M_{p/q}^ {[\delta,+\infty)}$, Lemma~\ref{lemma:PK} yields:

\begin{lemma}
 \label{lemma:uniformC}
 There is a uniform $C$ such that 
 $$
 \#\{
 [\gamma]\in \mathcal{PC}(X)\mid \ell(\gamma)\leq t
 \}
 \leq C e^{2t}, 
 $$
 for $X=M$ and $X=M_{p/q}$.
\end{lemma}

\subsection{Estimates for Ruelle functions}
\label{sub:Estimates}

We want to apply the results of the previous subsection to find uniform estimates on Ruelle  functions for the Dehn fillings.
We start with two elementary inequalities:
\begin{align}
 \textrm{For } z\in\CC ,\ |z|<1, \qquad &\big|\log | 1-z| \big|\leq \big|\log ( 1-|z| ) \big| \label{eqn:z<1}. \\
 \textrm{For }  z\in\CC ,\ |z|<1/2, \qquad &\big|\log | 1-z| \big|\leq  4 |z| . \label{eqn:z<1/2}
\end{align}
To prove \eqref{eqn:z<1} apply  logarithms to
$$
1-\vert z\vert\leq \vert 1-z\vert \leq 1+\vert z\vert \leq \frac1{1-\vert z\vert }
$$
and take into account that   $\log ( 1-|z| )<0$. Inequality \eqref{eqn:z<1/2} is then straightforward.

The next lemma reformulates the key calculus required for analysis of Ruelle functions, without using the formalism of measures of \cite{MP14}.

\begin{lemma}
\label{Lemma:length>L}
 For $\epsilon>0 $   there exists $C'(\epsilon)$ such that, if $s> 2+\epsilon$ and $L\geq 1$, then
  $$
 \sum_{\substack{ [\gamma]\in \mathcal{PC}(X)\\ \ell(\gamma) > L}}
\big\vert \log \vert 1- \chi(\gamma) e^{-s\, \ell(\gamma)}\vert   \big\vert
 \leq C'(\epsilon) \, e^{L(2+\epsilon -s)}
 $$
for $X= M_{p/q}$ or $X=M$, where $C'(\epsilon)$ is uniform on  $X$ and the unitary twist~$\chi$.
\end{lemma}

\begin{proof} 
We omit the subscript $ [\gamma]\in \mathcal{PC}(X)$ from the sums, which is always understood 
in the summations along the proof, and is combined with restrictions on the 
length of the geodesics. First, by \eqref{eqn:z<1/2}
 $$  
    \sum_{ \ell(\gamma) > L} \big |\log  \vert 1- \chi(\gamma) e^{-s\, \ell(\gamma)}\vert\big| \leq 
 4 \sum_{ \ell(\gamma) > L}  e^{-s\, \ell(\gamma)}.
 $$
 We divide the set $\mathcal{PC}(X)$ according to lengths. Set
 $$
 l_j= (1+\tfrac{j}{2}\epsilon) L.
 $$
Then by  using Lemma~\ref{lemma:uniformC}:
 \begin{equation}
 \label{eqn:telescope}
 \sum_{ \ell(\gamma) > L}  e^{-s\, \ell(\gamma)}\leq \sum_{j=0}^\infty\sum_{l_j< \ell(\gamma)\leq l_{j+1} }
 e^{-s \, \ell(\gamma)} 
 \leq \sum_{j=0}^\infty  C\, e^{2 l_{j+1} } e^{-s\, l_j} 
 \end{equation}
Since 
$$
2 l_{j+1} -s\, l_j=(2+\epsilon-s) L+ j (-\tfrac{s}{2} +1)\epsilon\,  L  ,
$$
the bound in 
\eqref{eqn:telescope} can be explicitly computed:
$$
 \sum_{j=0}^\infty  C\, e^{2 l_{j+1} } e^{-s\, l_j} = 
C\frac{e^{L(2+\epsilon-s)}}{ 1- e^{(-s/2+1)\epsilon \, L}}
 \leq \frac{C}{1-e^{-\epsilon^2/2}} 
$$
and we are done.
\end{proof}

The following bound is used in the proof of the theorem on the asymptotic behavior.

\begin{lemma}
\label{lem:Series}
For a closed hyperbolic three-manifold $N$ there exists a constant $C(N)$ depending only on $N$ such that 
\begin{equation}
\label{equa:Sum}
  \sum_{k=5}^\infty \big\vert \log\vert R_{\chi,-k}(k/2)\vert  \big\vert\leq C(N).
\end{equation}
 \end{lemma}

\begin{proof}
For each $k\geq 5$  split $\log |R_{\chi,-k}(k/2)|$ into two summations:
\begin{multline*}
\big|\log\vert R_{\chi,-k}(k/2)\vert \big| \leq \sum_{\substack{ [\gamma]\in \mathcal{PC}(N)\\ \ell(\gamma) \leq 1}}   \big| \log \vert 1-\chi(\gamma) e^{-k \lambda(\gamma)/2 }\vert \big|\\+   
\sum_{\substack{ [\gamma]\in \mathcal{PC}(N)\\ \ell(\gamma) > 1}}   
\big|\log \vert 1-\chi(\gamma) e^{-k \lambda(\gamma)/2 }\vert\big| \,.
\end{multline*}
We bound the contribution of the first summation.
There exists $k_0$ (depending on the length of the shortest geodesic of $N$) such that for
each $k> k_0$ we have for all $[\gamma] \in \mathcal{PC}(N)$:
\[
\vert \chi(\gamma) e^{-k \lambda(\gamma)/2 }\vert = \vert e^{-k\ell(\gamma)/2} \vert <\frac12\,.
\]
By~\eqref{eqn:z<1/2} we obtain for all $[\gamma] \in \mathcal{PC}(N)$:
$$
\big\vert\log \vert 1-\chi(\gamma) e^{-k \lambda(\gamma)/2 }\vert\big\vert\leq 
4 e^ {-k \ell(\gamma)/2}.
$$ 
As the number of geodesics of length $\leq 1$ is finite,
the contribution of the summation indexed by $\ell(\gamma)\leq 1$ in the left-hand side of (\ref{equa:Sum}) is bounded (by finitely many geometric series, starting from $k_0$).

For the summation of geodesics $[\gamma]$ with $
{\ell(\gamma)> 1}$, we use Lemma~\ref{Lemma:length>L} (that we stated for Dehn fillings but applies to any closed hyperbolic manifold if we do not require 
uniformity on the manifold). As $k\geq 5$, this yields again a bound by a geometric series.
\end{proof}

\begin{remark}
 In Lemma~\ref{lem:Series} we do not have uniformity on the Dehn fillings because of short geodesics ($k_0$ depends on the length of the shortest geodesic in $N$). We will get rid of short geodesics  by Dehn filling formulas in the next section (see Lemma~\ref{lemma:rationalpq}).
 Notice that we do have uniformity on the twist~$\chi$.
\end{remark}

\section{Asymptotic behavior of torsions}
\label{sec:asymtotic}
I this section we prove Theorem \ref{theorem:main} and Theorem \ref{theorem:Mahler} from the introduction.

\subsection{M\"uller's theorem for closed Dehn filling}

We give first the proof of M\"uller's theorem for the Dehn fillings $M_{p/q}$. We follow \cite{Muller}, just with the minor change of the rational unitary twist
$\chi$:

\begin{theorem}[M\"uller] For $\chi$ rational, $M_{p/q}$ a compatible Dehn filling and $m\geq 3$:
\label{Thm:Mueller}
\begin{align*}
\log \left\vert \frac{\tor(M_{p/q},\chi\otimes\rho_{2m})}{\tor(M_{p/q}, \chi\otimes\rho_4)} \right\vert = \sum\limits_{k=2}^{m-1} \log\vert R_{\chi,-2k-1}(k+\tfrac12) \vert-\frac 1 \pi \Vol(M_{p/q})(m^2-4),\\
\log \left\vert \frac{\tor(M_{p/q},\chi\otimes\rho_{2m+1})}{\tor(M_{p/q}, \chi\otimes\rho_5)} \right\vert = \sum\limits_{k=3}^{m} \log\vert R_{\chi,-2k}(k) \vert-\frac 1 \pi \Vol(M_{p/q})(m-2)(m+3).
\end{align*}
\end{theorem}


\begin{proof}
We prove the odd-dimensional case, the even dimensional case is similar.
Observe first that, by Lemma \ref{lem:zeta}, 
\begin{align}
\label{equa:RuelleDecompose}
{\mathcal R}_{\chi\otimes \rho_{2m+1}}(s)& = \prod\limits_{k=0}^{2m} R_{\chi,2m-2k}(s-(m-k)) \nonumber \\
& = R_{\chi,0}(s) \prod\limits_{k=1}^m R_{\chi,2k}(s-k)R_{\chi,-2k}(s+k)\nonumber \\
&={\mathcal R}_{\chi \otimes \rho_5}(s) \prod\limits_{k=3}^m R_{\chi,2k}(s-k)R_{\chi,-2k}(s+k)
\end{align}
Then, taking $s=0$ in (\ref{equa:RuelleDecompose}) and by Theorem~\ref{Thm:Wotzke}:
\begin{equation}
\label{equa:TorsionSquare}
T(M_{p/q}, E_{\chi\otimes\rho_{2m+1}})^2 = T(M_{p/q},E_{\chi\otimes\rho_5})^2 \prod\limits_{k=3}^m \vert R_{\chi,2k}(-k)\vert \vert R_{\chi,-2k}(k)\vert .
\end{equation}
Next recall from Proposition \ref{prop:Functional} that $$\vert R_{\chi,2k}(-k)\vert  = \vert R_{\chi,-2k}(k) \vert e^{-{4 k\Vol(M_{p/q})}/{\pi}},$$ then (\ref{equa:TorsionSquare}) becomes
\begin{equation*}
\log  \frac{T(M_{p/q},E_{\chi\otimes\rho_{2m+1}})}{T(M_{p/q}, E_{\chi\otimes\rho_5})} = -\frac{2}{\pi} \Vol(N)\sum\limits_{k=3}^m k +  \sum\limits_{k=3}^m \log \vert R_{\chi,-2k}(k) \vert 
\end{equation*}
and the statement follows from  Cheeger--M\"uller Theorem, Thm.~\ref{theo:EquivCM}.
\end{proof}

This theorem holds for any closed oriented hyperbolic 3-manifold, not only for Dehn fillings. 
Combined with  Lemma~\ref{lem:Series} it yields a twisted version of M\"uller's theorem:

\begin{corollary}[\cite{Muller}]
Let $N$ be a closed hyperbolic, oriented 3-manifold. Let $\chi \colon \pi_1(N) \to \mS^1$ be a homomorphism. Then:
$$\lim_{n \to \infty} \frac{\log \vert \tor(N, \chi \otimes \rho_n)\vert}{n^2} = -\frac{\Vol(N)}{4\pi}.$$
\end{corollary}
 
\subsection{Proof of the main theorem}

Assume that the twist $\chi$ is rational, until Lemma~\ref{lemma:corollaryholds}.
For an admissible Dehn filling $M_{p/q}$, let
$$
A=\{(\gamma_{p_1/q_1})^ {\pm 1},\ldots,(\gamma_{p_l/q_l})^ {\pm 1}\}$$ 
denote the set of oriented souls of the filling tori, namely the $l$ short geodesics added 
for the Dehn filling, with both orientations (hence $A$ has cardinality $2 l$).
Define, for $k\geq 5$,
$$
B^ {p/q}_ {\chi,k}=\sum_{ [\gamma]\in \mathcal{PC}(M_{p/q})-A} \log\vert 1-\chi(\gamma) e^ {-k \lambda_{p/q}(\gamma)/2}\vert .
$$
The convergence of this series follows again from \eqref{eqn:Margulis}, because  $k/2 \geq 5/2>2$.
We discuss below in Lemma~\ref{lemma:Bpq} further properties of this series.

Recall that $\varrho_n^{p/q}=\rho_n^ {p/q}\circ i_*$, where $\rho_n^ {p/q}$ is the symmetric power of the lift of the holonomy of 
$M_{p/q}$ and $i_*\colon\pi_1( M)\to \pi_1( M_{p/q})$ is induced by inclusion.

\begin{lemma}
\label{lemma:rationalpq}
Given a rational twist $\chi$ of $M$, 
 for any integer $m\geq 3$:
 \begin{equation*}
   \log\left|\frac{\tor(M,\chi\otimes\varrho^ {p/q}_{2m})}{\tor(M,\chi\otimes\varrho^ {p/q}_{4})}\right|=
 -\frac{m^ 2-4}{2}\Big(\sum_{i=1}^l \ell(\gamma_{p_i/q_i})+\frac{2}{\pi}\Vol(M_{p/q})\Big)+\sum_{k=2}^ {m-1} B^ {p/q}_ {\chi,2k+1} 
 \end{equation*}
and
 \begin{multline*}
  \log\left|\frac{\tor(M,\chi\otimes\varrho^ {p/q}_{2m+1}; b_1^{p/q},b_2^{p/q})}{\tor(M,\chi\otimes\varrho^ {p/q}_{5}; b_1^{p/q},b_2^{p/q})}\right|=
 -\tfrac{(m-2)(m-3)}{2}\Big(\sum_{i=1}^l \ell(\gamma_{p_i/q_i})+\frac{2}{\pi}\Vol(M_{p/q})\Big) \qquad \\
 +\sum_{k=3}^ {m} B^ {p/q}_ {\chi,2k} 
 \end{multline*}
\end{lemma}

\begin{proof}
 We discuss the even case, $2m$, and assume for simplicity that there is only one cusp, $l=1$. Set $\lambda=\lambda( \gamma_{p_1/q_1})$ and $\zeta=\chi(m_1)$. By Proposition~\ref{prop:Dehhfillingformula}:
 \begin{align*}
\log\left|
\frac{\tor(M_{p/q},\chi\otimes\rho_{2m}^ {p/q}) }
{ \tor(M,\chi\otimes\varrho_{2m}^ {p/q}) }
\right| & =
\sum_{k=0}^ {2 m-1}\log\vert e^ {(2m-1-2k)\lambda/2}\zeta-1\vert \\
&= \sum_{k=0}^ { m-1}\log\vert ( e^ {(k+\frac{1}{2})\lambda}\zeta-1  ) ( e^ {-(k+\frac{1}{2})\lambda}\zeta-1  )  \vert.
 \end{align*}
Thus
\begin{equation*}
 \log\left|
\frac{\tor(M_{p/q},\chi\otimes\rho_{2m}^ {p/q}) }
{\tor(M_{p/q},\chi\otimes\rho_{4}^ {p/q}) }
\frac
{ \tor(M,\chi\otimes\varrho_{4}^ {p/q}) }
{ \tor(M,\chi\otimes\varrho_{2m}^ {p/q}) }
\right| 
 =
 \sum_{k=2}^ { m-1}\log\vert ( e^ {(k+\frac{1}{2})\lambda}\zeta-1  ) ( e^ {-(k+\frac{1}{2})\lambda}\zeta-1  )  \vert.
\end{equation*}
With Theorem~\ref{Thm:Mueller} it yields
\begin{multline}
\label{eqn:quotienttor}
  \log\left|
\frac
{ \tor(M,\chi\otimes\varrho_{2m}^ {p/q}) }
{ \tor(M,\chi\otimes\varrho_{4}^ {p/q}) }
\right|
=\sum_{k=2}^ {m-1}\log\vert R^ {p/q}_ {\chi,-2k-1}(k+\frac{1}{2})\vert-\frac{1}{\pi}\Vol(M_{p/q})(m^ 2-4)\\
-  \sum_{k=2}^ { m-1}\log\vert ( e^ {(k+\frac{1}{2})\lambda}\zeta-1  ) ( e^ {-(k+\frac{1}{2})\lambda}\zeta-1  )  \vert,
\end{multline}
where $R^ {p/q}_ {\chi,-2k-1}$ denotes the twisted Ruelle zeta function of $M_{p/q}$. By definition of $B^ {p/q}_ k$:
\begin{equation}
\label{eqn:defB}
\log\vert R^ {p/q}_ {\chi,-2k-1}(k+\tfrac{1}{2})\vert
=\log\vert B^ {p/q}_{\chi,2k+1}\vert
+\log\vert (1-\zeta e^ {-(k+\frac12)\lambda}) ( 1-\overline \zeta e^ {-(k+\frac12)\lambda}) \vert.
\end{equation}
To combine   \eqref{eqn:quotienttor} and \eqref{eqn:defB}, we use:
\begin{multline}
 \label{eqn:short}
\left\vert
\frac{(1-\zeta e^ {-(k+\frac12)\lambda}) ( 1 - \overline\zeta e^ {-(k+\frac12)\lambda}) }
{(1-\zeta e^ {-(k+\frac12)\lambda}) (1-\zeta e^ {(k+\frac12)\lambda}) }
\right\vert
=
\left\vert
\frac{ 1 - \overline\zeta e^ {-(k+\frac12)\lambda} }
{1-\zeta e^ {(k+\frac12)\lambda} }
\right\vert
\\
=\left\vert
e^ {-\lambda(\frac12+k)}
\right\vert
\left\vert
\frac{  1 - \overline\zeta e^ {-(k+\frac12)\lambda} }
{  \overline\zeta e^ {-(k+\frac12)\lambda} -1}
\right\vert=
e^ {-\ell(\lambda) (k+\frac{1}{2}) } .
\end{multline}
The lemma follows from  \eqref{eqn:quotienttor}, \eqref{eqn:defB} and \eqref{eqn:short}.
\end{proof}

Define, for $k\geq 5$, a Ruelle function on $M$:
$$
 R_{\chi,-k} (k/2) = \prod_{\gamma\in\mathcal{PC}(M) }  ( 1-\chi(\gamma) e^ {- k \lambda(\gamma)/2} ) .
$$

\begin{lemma}
 \label{lemma:Bpq}
 For $k\geq 3$:
 \begin{enumerate}[(a)]
  \item  The series
 $$
 \sum_{\gamma\in\mathcal{PC}(M) }  \log \vert 1-\chi(\gamma) e^ {- k \lambda(\gamma)/2}\vert  
 \quad\textrm{ and }
 \sum_{\gamma\in\mathcal{PC}(M) }  \log \vert 1- e^ {- k \ell(\gamma)/2}\vert  
 $$ 
 converge uniformly.  
 \item There exists a constant $C>0$, uniform in $\chi$, such that $$
   \sum_{k=3}^ \infty \big\vert \log\vert R_{\chi, -2k} (k)\vert\big\vert \leq C,\quad\textrm{ and  } 
   \quad  \sum_{k=3}^ \infty \big\vert\log \vert R_{\chi,-2k-1} (k+\tfrac{1}{2})\vert\big\vert\leq C.$$
 \item The series $B^{p/q}_{\chi,k}$ also converges uniformly, uniformly on $(p,q)$ and the twist $\chi$. In addition, 
 $$
 \lim_{(p,q)\to\infty} B^{p/q}_{\chi,k} = R_{\chi,-k} (k/2)
 $$
 uniformly on the twist $\chi$.
 \end{enumerate}
\end{lemma}

In the lemma, uniformity on $(p,q)$ or on $\chi$ means that the series can be bounded term-wise in absolute value by a convergent series, independently on $(p,q)$ or/and on $\chi$.

\begin{proof}
Assertion (a) follows from Margulis bound on geodesic length growth \eqref{eqn:Margulis}, using inequalities~\eqref{eqn:z<1} and \eqref{eqn:z<1/2}.  
Assertion (b) has the very   same proof as Lemma~\ref{lem:Series}. For (c), we get uniformity on $(p,q)$ from Lemma~\ref{lemma:uniformC} and the fact that the sum does 
  not include any of the short geodesics in $A=\{(\gamma_1^{p/q})^{\pm 1}, \ldots, (\gamma_n^{p/q})^{\pm 1} \}$; hence
  there is a uniform lower bound away from zero on the length of the geodesics that appear in the sum of $B^{p/q}_{\chi,k}$, and from this, with  Proposition~\ref{prop:geodesicsconv}
  and  Lemma~\ref{lemma:uniformC},  we get uniformity.  
  Finally, 
 the limit follows also from Proposition~\ref{prop:geodesicsconv}.
 \end{proof}

 \begin{proposition}
 \label{Prop:MuellerRational}
  For a rational twist and $m\geq 3$,
  \begin{equation*}
   \log\left\vert\frac{\tor(M,\chi\otimes\rho_{2m+1}; b_1,b_2)}{\tor(M,\chi\otimes\rho_{5}; b_1,b_2)}
   \right\vert = 
   \sum_{k=3}^ m\log\vert R_ {\chi,-2k}(k)\vert -\frac{1}{\pi}\Vol(M)(m-2)(m+3) 
  \end{equation*}
  and
  \begin{equation*}
      \log\left\vert\frac{\tor(M,\chi\otimes\rho_{2m})}{\tor(M,\chi\otimes\rho_{4})}
   \right\vert =  
   \sum_{k=2}^ m\log\vert R_ {\chi,-2k-1}(k+\tfrac12)\vert -\frac{1}{\pi}\Vol(M)(m-2)(m+2)
     \end{equation*}
 \end{proposition}

\begin{proof}
 We take limits on Lemma~\ref{lemma:rationalpq} when $(p,q)\to\infty$.
 On the left hand side of the formula in Lemma~\ref{lemma:rationalpq}, 
 we apply Corollary~\ref{coro:quotients}.
 On the right hand side, we apply that $\Vol(M_{p/q})\to\Vol(M)$, that $\ell(\gamma_{p_i/q_i})\to 0$ \cite{Thu97, Gromov_Bourb}, and Lemma~\ref{lemma:Bpq}.
\end{proof}

Using Propositions~\ref{Prop:evaluation} and~\ref{Prop:MuellerRational}, we get:

\begin{corollary}
\label{coro:rational}
 Assume that $\zeta=(\zeta_1,\ldots,\zeta_r)\in (\mathbb{S}^ 1)^ r$ satisfies that $\zeta_j\in e^ {\pi i\mathbb Q}$, for $j=1,\ldots,r$, then:
 \begin{enumerate}[(a)]
  \item For $2m$ even:
  $$\log \left\vert \frac{\Delta_M^{\alpha,2m}(\zeta)}{ \Delta_M^{\alpha,4}(\zeta) }\right\vert= \frac{1}{\pi}\Vol(M)(m-2)(m+2) - \sum_{k=2}^ m\log\vert R_ {\chi,-2k-1}(k+\tfrac12)\vert 
  $$
  \item For $2m+1$ odd:
  $$\log \left\vert \frac{\Delta_M^{\alpha,2m+1}(\zeta)}{ \Delta_M^{\alpha,5}(\zeta) }\right\vert= 
  \frac{1}{\pi}\Vol(M)(m-2)(m+3) - \sum_{k=3}^ m\log\vert R_ {\chi,-2k}(k)\vert.
  $$
 \end{enumerate}
\end{corollary}

The proof of Theorem~\ref{theorem:main} for $\chi$ rational follows from Corollary \ref{coro:rational} and Lemma~\ref{lemma:Bpq} (b).
Next we remove the hypothesis on the rationality of $\chi$.

 \begin{lemma}
 \label{lemma:corollaryholds}
 Corollary~\ref{coro:rational} holds for any $\zeta=(\zeta_1,\ldots,\zeta_r)\in (\mathbb{S}^ 1)^ r$, without any assumption on rationality of $\zeta_1,\ldots,\zeta_r$.
\end{lemma}

\begin{proof} The proof is a density argument, using continuity of the terms that appear in Corollary~\ref{coro:rational},
that we need to justify. 

By Theorem~\ref{Thm:AlexNonZero}, we know that $\Delta_M^{\alpha,n}(\zeta)$ does not vanish, hence
$\log\vert \Delta_M^{\alpha,n}(\zeta)\vert$  is continuous on 
$\zeta=(\zeta_1,\ldots,\zeta_r)\in \mathbb{S}^ 1$, for $n\geq 2$.

For the continuity of $R_{\chi, -k}(k/2)$, $k\geq 5$, as in the proof of Lemma~\ref{lem:Series} we split again the series
$
\log \vert R_{\chi, -k}(k/2) \vert
$
in two: a finite sum indexed by geodesics of length $<L$ and a 
series indexed by geodesics of length $>L$.  The finite sum is continuous
on $\zeta$, so we need to chose $L$ so that the series indexed by geodesics of length $>L$ is arbitrarily small, 
uniformly on $\zeta$.
More precisely, by \eqref{eqn:z<1}, for each $[\gamma]\in\mathcal{PC}(M)$
$$
\vert 1-  \chi(\gamma) e^ {- k \lambda(\gamma)/2}\vert \leq \vert 1- e^ {- k \ell(\gamma)/2}\vert .
$$
As the series
$
 \sum_{[\gamma]\in\mathcal{PC}(M) } \big| \log \vert 1- e^ {- k \ell(\gamma)/2}\vert \big|  
$
converges (Lemma~\ref{lemma:Bpq}), for every $\varepsilon>0$ there exists $L=L(\varepsilon)>0$
such that 
$$
\sum_{\substack{ [\gamma]\in \mathcal{PC}(M)\\ \ell(\gamma) > L}}   
\big\vert
\log \vert 1-\chi(\gamma) e^{-k \lambda(\gamma)/2 }\vert
\big\vert<\varepsilon,
$$
uniformly on $\chi$. As $\mathcal{PC}(M)$ has finitely many elements of length $\leq L$, continuity of $R_{\chi, -k}(k/2)$
 on $\chi$ is clear.
%
%
%
%
%
 \end{proof}

\begin{proof}[Proof of Theorems~\ref{theorem:main} and~\ref{theorem:Mahler}]
By Lemma~\ref{lemma:corollaryholds}, equations of Corollary \ref{coro:rational} hold true for any $\zeta$. Moreover, the bound on the series in the right-hand side given by Lemma~\ref{lemma:Bpq} (b) is uniform in $\chi$, so that we obtain Theorem \ref{theorem:main} by dividing by $(2m)^2$ in case (a), or by $(2m+1)^2$ in case (b), and by letting $m$ tend to infinity.
This proves Theorem~\ref{theorem:main}, and
Theorem \ref{theorem:Mahler} follows in the same way, again using uniformity on $\zeta$ of the bound from Lemma \ref{lemma:Bpq}.
\end{proof}

\appendix

\input{appendix_torsion}

\input{appendix_L2}

\input{appendix_anosov}

\bibliography{biblio}
\bibliographystyle{plain}

\end{document}

%% file: appendix_torsion.tex

\section{Review on combinatorial torsion}
 \label{Appendix:Orbifolds}

The goal of this appendix is to review basic properties of combinatorial torsion.

We restrict to compact orientable three-dimensional manifolds $N$, possibly with boundary. To simplify notation, we write 
$\Gamma=\pi_1(N).
$
We also fix  a field $\mathbb F$ of characteristic 0, and a representation $\rho\colon \Gamma\to\mathrm{GL}_n(\mathbb F)$.

\subsection{Twisted chain complexes}
Fix a CW-complex structure $K$ on $N$.
The complex of chains on the universal covering $\widetilde K $ is the free $\ZZ$-module on the cells of $\widetilde K$, equipped with the usual boundary operator, and it
is denoted by $C_*(\widetilde K,\ZZ)$. 
It has an action of $\Gamma=\pi_1(N)$ that turns it into a left $\mZ[\Gamma]$-module.
The group  $\Gamma$ acts on $\mathbb{F}^n$ via $\rho$ on the left,  and 
for the tensor product $\Gamma$ acts on $\mathbb{F}^n$ on the right using inverses: any $\gamma\in \Gamma$ 
maps $v\in \mathbb{F}^n$ to $\rho(\gamma^{-1}) (v)$. We write $_\rho\mathbb{F}^n$ and $\mathbb{F}^n_\rho$ to emphazise the left and right $\mZ[\Gamma]$-module structures, respectively.
The  \emph{twisted} chain and cochain complexes are defined as:
\begin{align}
 C_*(K,\rho)&= \mathbb{F}^n_\rho\otimes_{\Gamma} C_*(\widetilde K,\ZZ),\\
 C^*(K,\rho)&= \Hom_{\Gamma} ( C_*(\widetilde K,\ZZ), _\rho\mathbb{F}^n).
\end{align}
Those are complexes and co-complexes of finite-dimensional vector spaces, and the corresponding homology and cohomology groups are denoted by
$
H_*(K,\rho)$ and $ H^*(K,\rho).
$

\subsection{Geometric bases}
For a cell $\widetilde e\in\widetilde K$, $\ZZ[\Gamma]\widetilde e$ denotes the free $\mZ[\Gamma]$-module of rank one on its $\Gamma$-orbit (i.e.~the free 
module on all lifts of a given cell $e$ in $K$).

\begin{lemma}
We have natural isomorphisms of $\mathbb F$-vector spaces:
$$
\begin{array}{rcl}
 \Hom_\Gamma( \ZZ[\Gamma]\widetilde e, \mathbb F^n) & \mapsto &  \mathbb F^n \\
 \theta & \mapsto & \theta(\widetilde e)
\end{array}
\qquad
\begin{array}{rcl}
 \mathbb F^n \otimes_\Gamma \ZZ[\Gamma]\widetilde e & \mapsto &  \mathbb F^n \\
 v\otimes \widetilde e & \mapsto & v
\end{array}
$$
\end{lemma}

The proof is straightforward.

Chose $\{v_1,\ldots, v_n\}$ to be a  basis for 
$ \mathbb F^n$.
Let $\{e^i_1,\ldots e^i_{j_i}\}$ be the set of $i$-dimensional cells of $K$. For each cell $e^i_ j$
chose a lift $\widetilde e^i_j$ to $\widetilde K$.
Then
$\{v_k\otimes  \widetilde e^i_j\}_{i,j,k}$ is an $\mathbb F$-basis for $C_i(K,\rho)$.
Similarly
$\{(\widetilde e^i_j)^*\otimes v_l   \}_{i,j,l}$ is an $\mathbb F$-basis for $C^i(K,\rho)$
where $\big((\widetilde e^i_j)^*\otimes v_l\big) (\gamma \widetilde e^i_k) = \rho(\gamma)v_l \delta_{jk}$.

\begin{definition}
\label{definition:geometricbasis}
 We call this  basis a \emph{geometric basis}  for $C_*(K,\rho)$, respectively for $C^*(K,\rho)$.
\end{definition}

\subsection{Combinatorial torsion}

Recall the definition of torsion of a complex of finite dimensional $\mathbb F$-vector spaces $C_*$ with bases $\{c^i\}_i$  for the chain complexes and
bases  $\{h^i\}_i$ for the homology groups,
following for instance \cite{Mil66}. 
For that purpose we consider
the space of boundaries $B_i= \mathrm{im}( \partial\colon C_{i+1}\to C_{i})$, the space of cycles 
$Z_i=\ker( \partial\colon C_i\to C_{i-1})$ and the homology $H_i=Z_i/B_i$.
We chose $b^i$ an $\mathbb{F}$-basis for $B_i$.
Using the exact sequences
$$
0\to B_i\to Z_i\to H_i\to 0, \quad 0\to Z_i\to C_i\to B_{i-1}\to 0
$$
we lift  $b^i$ it to a subset $\widetilde b^i$
of $C_{i+1}$, and $h^i$ to a subset $\widetilde h^i$ of $C_i$, so that $\widetilde b^{i-1}\cup \widetilde h^i \cup b^i$ is an $\mathbb{F}$-basis for $C_i$. We denote
$[\widetilde b^{i-1}\cup \widetilde h^i \cup b^i\colon c^i]$ the determinant of the matrix which takes $c^i$ to 
 $\widetilde b^{i-1}\cup \widetilde h^i \cup b^i$ (in the colomns of the matrix are the coordinates of
$ \widetilde b^{i-1}\cup \widetilde h^i \cup b^i$ with respect to $c^i$). Then we define
$$
\mathrm{tor} ( C_*, \{c^i\}_i,  \{h^i\}_i)=\prod_{i=0}^3 [\widetilde b^{i-1}\cup \widetilde h^i \cup b^i\colon c^i]^{(-1)^i}\in \mathbb{F}^ *
$$
If we have defined a geometric basis $\{c^i\}_i$ as in Definition~\ref{definition:geometricbasis}, then  the torsion   is:
$$
\mathrm{tor}(N,\rho, \{h^i\}_i) = \mathrm{tor}( C_*(\widetilde K,\rho), \{c^i\}_i ,  \{h^i\}_i  ) \in \mathbb F^*/\pm \det\big(\rho(\Gamma)\big) .
$$
It is straightforward to check that it is well defined (see  \cite{Mil66} and \cite{Porti97}). Topological invariance follows from uniqueness of triangulations on three-manifolds.

\subsection{Duality homology-cohomology}
\label{subDuality}
We aim to define the torsion from the cohomological point of view. 
Let $V$ be a finite dimensional $\FF$-vector space, and
let $\rho\colon\Gamma\to \GL(V)$ be a representation.
The \emph{contravariant representation} or \emph{dual} representation
$\rho^*\colon\Gamma\to\mathrm{GL}(V^*)$ is
defined by $\rho^*(\gamma) (f) =f \circ \rho(\gamma^{-1})$.
\begin{lemma}
\label{lem:bilinear_dual}
The representations $\rho$ and $\rho^*$ are equivalent
if and only if there exists a non-degenerate bilinear form
$B\colon V\otimes V\to\FF$ which is  $\Gamma$-invariant.
\end{lemma}
If we choose a basis in $V$ and its dual basis in $V^*$, we obtain  matrix representations
$\rho,\,\rho^*\colon\Gamma\to\GL_n(\FF)$, and they are related by
$\rho^*(\gamma)= \rho(\gamma^{-1})^t$. Notice that $(\rho^*)^*=\rho$.

\begin{example}
\label{ex:sldual}
For any representation $\rho\colon\Gamma\to\operatorname{SL}_2(\FF)$, the module 
$V=\FF^2$ has a skew-symmetric 
non-degenerate bilinear form defined by the determinant. Namely, the vectors $(x_1,x_2)$ and $(y_1,y_2)\in\FF^2$
are mapped to
\[\det
  \begin{pmatrix}
   x_1 & y_1 \\ x_2 & y_2
  \end{pmatrix}
.
\]
In view of Lemma~\ref{lem:bilinear_dual}, $ \rho^*$ and $\rho$ are equivalent.
More concretely, for any matrix $A\in \SL_2(\FF)$ we have
\[
\big(\begin{smallmatrix} 0 & 1\\-1&0\end{smallmatrix}\big) A 
\big(\begin{smallmatrix} 0 & -1\\1&0\end{smallmatrix}\big) = (A^{-1})^t\,.
\]
\end{example}

\bigskip

The pairing $V^*\otimes V \to \FF$ induces a perfect pairing of complexes
$$
\langle\,,\,\rangle\colon C_i(K,\rho) \otimes C^i(K,\rho^*) \to\mathbb F ,
$$
defined by: 
$$
\langle  v\otimes \tilde{e}, \theta  \rangle =\theta(\tilde{e})(v),
$$
where $\tilde{e}$ is a cell of $\widetilde K$, $v\in V$ and 
$\theta\in \hom_\Gamma(C_*(\widetilde K), V^* )$.
It is easy to check that it is well defined, non-degenerate and that it is compatible with the boundaries and coboundaries:
$$
\langle\partial \cdot \, , \, \cdot \rangle=\pm\langle \cdot\, ,\, \delta \cdot \rangle
$$
where the sign depends only on the dimension.
Hence in its turn it induces a non-degenerate Kronecker pairing between homology and cohomology
$$
H_i(K,\rho)\times H^i(K,\rho^*)\to\FF\,  .
$$
Now we can relate the torsion in homology with the torsion in cohomology. 
We denote by $B^i$, $Z^ i$ and $H^ i$ the coboundary, cocycle and cohomology spaces, respectively. In addition, we take $ \bar b^i$ 
basis for $B^ i$ that we lift to $\widetilde {\bar b}^ i$ in $C^ {i-1}$. We define the torsion of a cocomplex with bases in cohomology $h^ i$ as:
$$
\mathrm{tor} ( C^ *, \{\bar c^i\}_i,  \{\bar h^i\}_i)=\prod_{i=0}^3 [\widetilde b^{i+1}\cup \widetilde {\bar h}^i \cup \bar b^i\colon \bar c^i]^{(-1)^{i+1}}\in \mathbb{F}^ *
$$
To relate torsion in homology and cohomology, notice that
the geometric basis $c^i$ of $C_i(K,\rho)$ and  $\bar c^i$ of $C^i(K,\rho^*)$ can be chosen to be dual. Then the matrices of the boundary
operators with respect to those basis are transpose to the matrices of the respective coboundary operators.
From this, we have:

\begin{proposition}\label{prop:torsion}
 If the basis $h^ i$ for $H_i(K,\rho)$  and the basis $\bar h^ i$ for $H^ i(K,\rho^ *)$ are dual for each $i$, then 
 $$
\mathrm{tor}( C_i(K,\rho), \{c^i\}, \{h_i\})=\mathrm{tor}( C^i(K,\rho^*), \{\bar c^i\}, \{\bar h^i\}).
$$
\end{proposition}

\begin{remark}
\label{Remark:Poincare}
We shall also use Poincar\'e duality with twisted coefficients, see for instance   \cite{JohnsonMillson,Porti97}. For $N$ a compact orientable manifold:
$$
 H^i(N;\rho)^*\cong H^{\dim N-i}(N,\partial N;\rho^*)
$$
\end{remark}

\subsection{The representations we are interested in the paper}

Here we list the representations we use in the paper. We describe in which space they are defined. Since the torsion lies by definition in
$\mathbb F/\pm\det\rho(\Gamma)$, we need to understand $\det\rho(\Gamma)$. 
We start with a representation
$\rho\colon\Gamma\to\SL_2(\mathbb{C})$, and we put $\rho_n := \mathrm{Sym}^{n-1}\circ\rho$.

\begin{enumerate}
 \item 
 For the representation $\rho_n\colon \Gamma\to \mathrm{SL}_n(\CC)$,  the torsion is well defined
 up to sign, as $\det\rho(\Gamma)=\{1\}$.
 Recall that $\CC^n$ has a non-degenerate $\mathrm{Sym}^{n-1}$-invariant bilinear form
which is symmetric for $n$ odd and antisymmetric for $n$ even.
By irreducibility, the form is unique up to scalar.
For $n=2$ this form is the determinant
(see Example~\ref{ex:sldual}). 
For general $n$, it is the symmetrization of this bilinear form on $\CC^2$,
an explicit formula is given in Lemma~3.1.4 in \cite{Springer}. Thus
$$
\rho_n^*\cong \rho_n\,.
$$

 \item 
 For the representation $\alpha\otimes \rho_n\colon  \Gamma\to \mathrm{GL}_n(\CC(t_1,\ldots, t_r))$, the torsion is well defined up to sign and multiplication by monomials
$t^m = t_1^{m_1}\cdots t_r^{m_r}$.
There is no $\Gamma$-invariant bilinear form on $\CC(t_1,\ldots, t_r)^ n$, and hence
$(\alpha\otimes \rho_n)^*$ and $\alpha\otimes \rho_n$ are not equivalent (for  non-trivial $\alpha$), nevertheless
$$
 (\alpha\otimes \rho_n)^*= \alpha^{-1}\otimes \rho_n^* \cong \alpha^{-1}\otimes \rho_n.
$$

\item 
In the case of the representation $\chi\otimes\rho_n\colon    \Gamma\to \mathrm{GL}_n(\CC)$ for a character $\chi\colon\Gamma\to \mathbb S^1\subset \CC$, 
 only the modulus of the torsion is well defined.
There is no $\Gamma$-invariant bilinear form on $\CC^n$, and 
$$(\chi\otimes\rho_n)^*=\overline\chi\otimes\rho_n^* \cong\overline \chi\otimes\rho_n. $$
\end{enumerate}
By the classical duality theorems of Franz \cite{Franz37} and Milnor \cite{MilnorAlexander}  we have
$$
\tor(N, \alpha\otimes\rho_n)=\pm t^m \tor(N, \alpha^ {-1}\otimes\rho_n)
$$
for some  multiplicative factor $\pm t^m$, and
$$
\vert \tor(N, \chi\otimes\rho_n)\vert =\vert\tor(N, \bar \chi\otimes\rho_n)\vert .
$$

\begin{remark}
\label{remark:isotropic}
Let us recall some basic facts about the irreducible representation $\mathrm{Sym}^{n-1}\colon \SL_2(\CC)\to\SL_n(\CC)$.
For details we refer to Springer's book \cite[Section~3.1]{Springer}. 
\begin{itemize}
\item The representation $\mathrm{Sym}^{n-1}$ factors through $\operatorname{PSL}_2(\CC)$ for $n$ odd.
\item The space $\CC^n$ has a non-degenerate $\mathrm{Sym}^{n-1}$-invariant bilinear form, 
that is symmetric for $n$ odd and antisymmetric for $n$ even. By irreducibility, this form is unique up to scalar.
An explicit formula is given in Lemma~3.1.4 in \cite{Springer} (for $n=2$ see Example~\ref{ex:sldual}). In higher dimensions the $\mathrm{Sym}^{n-1}$-invariant bilinear form is the 
symmetrization of the determinant.

\item The image of a non-trivial unipotent element in $\SL_2(\CC)$ is a \emph{regular} unipotent element in $\SL_n(\CC)$, 
i.e.\ it is conjugate to an upper-triangular matrix which has only ones on the diagonal and a 
\emph{single block} in the Jordan-H\"older form.

It follows that the image of a parabolic element $g\in\SL_2(\CC)$, with trace $\epsilon_g 2$, for some $\epsilon_g=\pm 1$ has
a unique eigenspace, of dimension one and with eigenvalue $ \epsilon_g^{n-1}$. Moreover, this  1-dimensional eigenspace is an \emph{isotropic} subspace of~$\CC^n$.

\item 
We also use Clebsch-Gordan formula:
\begin{equation}
\label{eqn:Clebsch-Gordan}
\operatorname{Ad}\circ \operatorname{Sym}^{n-1}
\cong
\mathrm{Sym}^{2(n-1)}\oplus \mathrm{Sym}^{2(n-2)}\oplus\cdots \oplus \mathrm{Sym}^{4}\oplus \mathrm{Sym}^{2}
\end{equation}
(see for instance~\cite[Exercise 3.2.4]{Springer})

\end{itemize}
\end{remark}

%% file: appendix_L2.tex

\section{Vanishing of $L^2$-cohomology}
\label{Appendix:L2}
 
The goal of this appendix is to show that the classical 
vanishing theorems in cohomology 
\`a la Matsushima-Murakami \cite{MatMur}
apply to our situation with a twist $\chi$.
\subsection{Review on $L^2$-forms on hyperbolic manifolds}

In this appendix $M$ is an oriented hyperbolic three-manifold (possibly of infinite volume) and
$
\chi\colon\pi_1(M) \to \mathbb S^1 
$ 
a unitary character,  possibly trivial. 
In the rest of the paper  we assume that $M$ has finite volume, but not in this appendix.
Let 
$$
\rho\colon \pi_1(M)\to\mathrm{SL}_2(\CC)
$$
be a lift of the holonomy, and
$$
\operatorname{Sym}^{n-1}\colon \mathrm{SL}_2(\CC)\to \mathrm{SL}_n(\CC)
$$
be the $n$-dimensional holomorphic irreducible representation. The composition is denoted by 
$$
\rho_n= \operatorname{Sym}^{n-1}\circ \rho \colon \pi_1(M)\to\mathrm{SL}_n(\CC).
$$
We consider the flat vector bundle 
$\CC^n\to E_{\chi \otimes \rho_n}\to M$
with total space:
$$
E_{\chi \otimes \rho_n}= \CC^n\times_{\chi \otimes \rho_n} \widetilde M .
$$

We describe the hermitian metric on the bundle (i.e.~on each fibre).  
View the universal covering $\widetilde M$
as the quotient $ \textrm{SL}(2,\CC)/ \textrm{SU}(2)$, 
start with a hermitian product on $\CC^n$ invariant by the action of 
the compact group $\textrm{SU}(2)$, 
and translate it along  $ \textrm{SL}(2,\CC)/\textrm{SU}(2) \cong \mathbb H^3$, via $\operatorname{Sym}^{n-1}$.
This hermitian product is compatible with the action of $\rho_n$ by construction, but also with the action of $\chi $,
because hermitian products are invariant by multiplication by unit complex numbers. Thus it induces a non-flat hermitian
metric  on the bundle $E_{\chi \otimes \rho_n}$.

We consider 
$\Omega^p(M,E_{\chi \otimes \rho_n})$ the space of $p$-forms valued in $E_{\chi \otimes \rho_n}$, 
namely smooth sections of the bundle $ E_{\chi \otimes \rho_n} \otimes \Omega^p(M)   $.
The Riemannian metric on $TM$ and the hermitian metric on the fibres yield a Hodge star operator $*$, a codifferential $\delta$, 
and a Laplacian $\Delta^p_{\chi,n} $ on  
$\Omega^p(M,E_{\chi \otimes \rho_n})$.
They also provide a hermitian product on $p$-forms:
$$
(\phi,\psi)= \int_M \phi\wedge*\psi   \qquad \forall\phi,\psi\in \Omega^p_c(M,E_{\chi \otimes \rho_n}),
$$
where $\wedge$ denotes the exterior product on forms in $\Omega^*(M)$ combined with the hermitian product on $E_{\chi \otimes \rho_n}$,
and $\Omega^p_c(M,E_{\chi \otimes \rho_n})$ the space of compactly supported forms in $\Omega^p(M,E_{\chi \otimes \rho_n})$

Pointwise  we use the Riemannian metric on $M$ and the  hermitian product on the bundle to define a hermitian
product on forms $\langle\phi,\psi\rangle_x$ at any $x\in M$, so that
$$
   (\phi\wedge*\psi )_x= \langle\phi,\psi\rangle_x \, d\mathrm{vol}\qquad  \textrm{ for every } x\in M.
$$

De Rham cohomology of the cocomplex $(\Omega^*(M,E_{\chi \otimes \rho_n}), d)$ is 
denoted by $H^*(M, E_{\chi \otimes \rho_n} )$; it is isomorphic to the simplicial cohomology $H^*(M; \chi \otimes \rho_n)$.
%
The aim of this appendix is to prove that every closed form with finite norm is exact:

\begin{theorem}
\label{Thm:vanishingL2} 
For any form $\omega\in \Omega^1(M,E_{\chi \otimes \rho_n})$ satisfying $(\omega,\omega)<\infty$, if $d\omega=0$ then there exists
$\eta\in \Omega^0(M,E_{\chi \otimes \rho_n})$ such that $ \omega = d\eta$.
\end{theorem}

Theorem~\ref{Thm:vanishingL2} is in fact a theorem on vanishing of $L^2$-cohomology and it is a version of a theorem of  Garland  \cite{Garland}.
This theorem is proved in Subsection~\ref{subsection:proofvanishing}.


\subsection{Proof of the theorem}
\label{subsection:proofvanishing}

The proof is based on the following theorem of Andreotti-Vesentini \cite{AndreottiVesentini} and Garland  \cite{Garland}:
uniform ellipticity implies that closed forms of finite norm are exact.

\begin{theorem}[Thm~3.22 in~\cite{Garland}]
\label{Thm:Garland}
If there exists a constant $c>0$ such that for every form $\omega\in \Omega^1(M,E_{\chi \otimes \rho_n})$ with compact support
\begin{equation}
\label{eqn:uniformellipticity}
 (d\,\omega,d\,\omega)+ (\delta\,\omega, \delta\,\omega)\geq c (\omega,\omega),
\end{equation}
then Theorem~\ref{Thm:vanishingL2} holds.
\end{theorem}

Inequality \eqref{eqn:uniformellipticity} is called \emph{uniform ellipticity} because, for forms $\omega$ with compact support,
it is equivalent to 
$$
(\Delta^1_{\chi,n}\omega,\omega) \geq c (\omega,\omega).
$$
In order to prove   uniform ellipticity,  we use the formalism of
Matsushima and Murakami
\cite{MatMur}, as in  \cite{Garland,RagAJM,HKJdG,MP10,Rag}. 
Since $\mathbb{H}^3\cong \textrm{SL}(2,\CC)/\textrm{SU}(2)$, from the decomposition 
$\mathfrak{sl}(2,\CC)=\mathfrak{su}(2)\oplus i \mathfrak{su}(2)$, 
orthogonal for the real Killing form,
there is a natural identification $\phi_x$ of the tangent space at 
each point $T_x \mathbb{H}^3$ with 
 $  i \mathfrak{su}(2)$. 

Let $E=\CC^n\times \mathbb{H}^3$ be the trivial bundle, equipped with the natural flat connection, with covariant derivative $\nabla_v$
and the standard differential $d\colon \Omega^p(\mathbb{H}^3, E)\to \Omega^{p+1}(\mathbb{H}^3, E)$. 
Following \cite{MatMur} we define a new covariant derivative
$$
\widetilde\nabla_v=\nabla_v-\operatorname{sym}^{n-1}(\phi_x(v)),  \qquad\forall v\in T_x\mathbb{H}^3   ,
$$
where $\operatorname{sym}^{n-1}\colon\mathfrak{sl}_2(\CC)\to\mathfrak{sl}_n(\CC)$ is the representation of the Lie algebra associated to $\operatorname{Sym}^{n-1}$.
The corresponding connection $D\colon \Omega^0(\mathbb{H}^3, E)\to \Omega^1(\mathbb{H}^3, E)$ is given by
$$
D s (v)= d\, s (v)-\operatorname{sym}^{n-1}(\phi_x(v))(s) 
$$
for every section $s\in \Omega^0(\mathbb{H}^3, E)$ and every tangent vector $v\in T_x\mathbb{H}^3$.
By construction, $D$ is a connection:
$
D(f\, s)= f Ds+ s d\, f
$
for any function $f$ and any section $s$, and it can be checked that it is metric:
$$
d\langle s_1,s_2\rangle_ x= D s_1\wedge s_2+s_1\wedge D s_2 ,
$$
where $x\mapsto \langle s_1,s_2\rangle_ x$ is a function on $\mathbb{H}^3$.
The connection $D$ is introduced in \cite{MatMur} as induced from the natural connection associated to the principal bundle on 
$\mathrm{SL}_n(\CC) \to\mathrm{SL}_2(\CC)$, corresponding to the representation $\mathrm{Sym}^{n-1}$.

Now, for any frame $\{e_1, e_2, e_3\}$ of $\mathbb H^3$, let  $\{\omega^1, \omega^2, \omega^3\}$ denote its dual coframe.
As $d=\sum_{j=1}^3 \omega^j\wedge\nabla_{e_j}$ \cite[(6.19)]{Wu} we have: 
 
\begin{proposition}\cite{MatMur}
\label{prop:dDT}
On $\Omega^*(\mathbb{H}^3,E)$ we have
 $$d=D+ T\quad\textrm{ and }\quad\delta= D^* + T^*,$$ 
 where
\begin{align}
 D& = \sum_j \omega^j\wedge\widetilde\nabla_{e_j}, &   T=& \sum_j  \omega^j\wedge\operatorname{sym}^{n-1}(\phi_x( e_j) ), \\
 D^* &=   -\sum_j i(e_j) \widetilde\nabla_{e_j}, &   T^*=&  \sum_j i(e_j)\operatorname{sym}^{n-1} ( \phi_x( e_j) ).
\end{align}
\end{proposition}

Up to now these operators are defined on the   bundle $E$ on $\mathbb{H}^3$, and we want to descend them to the bundle 
on $M$ twisted
by $\chi\otimes\rho_n$.
Notice that the operators and Proposition~\ref{prop:dDT}  are found in the literature without the twist, so we need to justify why it works in our situation.
We may view 
the definitions and Proposition~\ref{prop:dDT}
on the trivial bundle $E=\CC^n\times \mathbb H^3\to \mathbb H^3$ as being equivariant for the action of $\pi_1(M)$
via the representation 
$\rho_n$. On the other hand, these formulas are $\CC$-linear, so they are equivariant for the action via $\chi\otimes\rho_n$:

\begin{remark}
Proposition~\ref{prop:dDT} holds true on  $\Omega^*(M,E_{\chi\otimes\rho})$.
\end{remark}

\begin{proposition}\cite{MatMur} There is 
 a Weitzenb\"ock formula:
$$
\Delta= d\delta+\delta d= D\, D^*+D^* D+ T\, T^*+T^* T= \Delta_D+H
$$
where $\Delta_D= D\, D^*+D^* D$ and $H= T\, T^*+T^* T$.

Moreover, for any form $\omega\in \Omega^1(M, E_{\chi\otimes\rho_n})$ with compact support:
\begin{equation}
\label{eqn:Stokes}
 (d\tilde\omega,d\tilde\omega)+(\delta\tilde\omega,\delta\tilde\omega)=
 (D\tilde\omega,D\tilde\omega)+(D^*\tilde\omega,D^*\tilde\omega)+ (H\tilde\omega,\tilde\omega).
\end{equation}
\end{proposition}

A proof of both propositions can be found in \cite{MatMur} and also in \cite{HKJdG, MP10,Rag}. The Weitzenb\"ock formula requires the identity
$$D\,T^*+T\,D^*+D^*\,T+T^*\,D=0$$ 
on forms with compact support. Identity \eqref{eqn:Stokes} is proved from Stokes theorem, cf.~\cite[Equation (5)]{HKJdG}.
 On the other hand:

\begin{proposition}\cite{MatMur}
 For any $ \omega\in \Omega^1(M,E_{\chi\otimes\rho_n})$, pointwise, 
$$
\langle H\omega,\omega\rangle_x\geq c_n \langle \omega,\omega\rangle_x , \qquad \textrm{ at every point }x\in M,
$$
for a uniform constant $c_n$ that
depends only on $\mathrm{Sym}^{n-1}$.
 
\end{proposition}

See also \cite{RagAJM,MP10} for a proof. Thus 
uniform ellipticity 
holds 
for any form 
in $\omega\in \Omega^1(M, E_{\chi\otimes\rho_n})$ with compact support; so  Theorem~\ref{Thm:Garland} applies.

\medskip

Next we assume that the orientable hyperbolic manifold $M$ has finite topology, that  it has a compactification $\overline{M}$ 
that consists in adding surfaces, as it has finitely many ends that  are 
topologically products.
Every element in the kernel of $H^1(M; \chi\otimes\rho_n)\to H^1(\partial\overline M; \chi\otimes\rho_n)$
is represented by a differential form with compact support in $M$, in particular it has finite norm. Thus, by
 Theorem~\ref{Thm:vanishingL2}:

\begin{corollary}
\label{Coro:injection}
We have an injection induced by inclusion:
 $$H^1(M; \chi\otimes\rho_n)\hookrightarrow H^1(\partial\overline M; \chi\otimes\rho_n).$$ 
\end{corollary}

%% file: appendix_anosov.tex

\section{Dynamics of pseudo-Anosov diffeomorphisms on the variety of characters}
\label{Appendix:pseudoAnosov}

Let $\Sigma$ be a compact orientable surface, possibly with boundary, connected and with negative Euler characteristic, and let 
$$\phi\colon \Sigma\to\Sigma$$
be a pseudo-Anosov diffeomorphism. Note that $\phi$ does not necessarily act by the identity on $\partial \Sigma$, but it may permute the boundary components.

Consider the mapping torus of $\phi$:
$$M(\phi)=\Sigma\times[0,1]/(x,1)\sim (\phi(x),0).$$
Its fundamental group is a semi-direct product 
\begin{equation}\label{eq:pres}
\pi_1 (M(\phi)) \cong \pi_1(\Sigma)\rtimes\ZZ\cong \langle \pi_1(\Sigma),\tau\mid \tau \gamma \tau^{-1} = \phi_\#(\gamma)\text{ for all } \gamma\in\pi(\Sigma)\rangle\,.
\end{equation}
Here $\phi_\#\colon\pi_1(\Sigma)\to\pi_1(\Sigma)$ denotes the isomorphism induced by $\phi$. Notice that a different choice of $\tau$ would yield the composition of
$\phi_\#$ with an inner automorphism.

According to Thurston’s hyperbolization theorem,  $M(\phi)$ admits a finite volume and complete hyperbolic metric \cite{Otal}. 
The holonomy of this hyperbolic structure lifts to a representation 
$\pi_1 (M(\phi))\to\SL_2(\CC)$ and its restriction to $\pi_1(\Sigma)$ yields a representation
$\rho_2\colon\pi_1 (\Sigma)\to\SL_2(\CC)$. The composition of $\rho_2$ with $\mathrm{Sym}^{n-1}$ will be denoted by $\rho_n := \mathrm{Sym}^{n-1}\circ\rho_2$. 
Notice that, by \eqref{eq:pres}, $\rho_n$ and 
$$\phi^*(\rho_n) := \rho_n\circ\phi_\#$$ are conjugate by $\rho_n(\tau)$.
This implies that the equivalence class $[\rho_n]$  of $\rho_n$ is a fixed point of the action of $\phi^*$ on the character variety
\[
\mathcal{R}\big(\Sigma,\SL_n(\CC)\big) := \Hom\big(\pi_1(\Sigma),\SL_n(\CC)\big)\sslash \SL_n(\CC)\,.
\]
For definitions and more details see  \cite[Section~4.3]{KapovichBook} or~\cite{HPvol}.

In the case of a closed surface ($\partial\Sigma=\emptyset$) and $n=2$,  M.~Kapovich proved in \cite{Kapovich98}
that  $[\rho]$ is a hyperbolic fixed point of 
$\phi^*\colon \mathcal{R}(\Sigma,\mathrm{SL}_2(\CC))\to\mathcal{R}(\Sigma,\mathrm{SL}_2(\CC))$, namely
the tangent map
\[
(d\phi^*)_{[\rho_2]}\colon T_{[\rho_2]} \mathcal{R}(\Sigma,\mathrm{SL}_2(\CC))\to
T_{[\rho_2]}\mathcal{R}(\Sigma,\mathrm{SL}_2(\CC))
\]
has no eigenvalues of modulus one.

In the case of a surface with boundary this assertion is no longer true, as
the trace functions of the peripheral elements are invariant under a power of $\phi^*$. This causes
$(d\phi^*)_{[\rho_2]}$ to have 1 as an eigenvalue
(see   \cite[Section~4.5]{Porti97} for $n=2$ and a punctured torus).

In order to generalize Kapovich's result, for surfaces with boundary we consider the relative  character variety. Let 
$$\partial \Sigma = \partial_1\sqcup\cdots\sqcup \partial_s $$  
be the decomposition in connected components. The  relative  character variety is
\[
\mathcal R\big(\Sigma,\partial \Sigma,\mathrm{SL}_n(\CC)\big) :=
\{ [\rho]\in \mathcal R(\Sigma,\mathrm{SL}_n(\CC)) \mid \text{$\rho(\partial_i)$ and $\rho_n(\partial_i)$ are similar}\}\,, 
\]
with the convention that 
$\mathcal R(\Sigma,\emptyset,\mathrm{SL}_n(\CC)) = \mathcal R(\Sigma,\mathrm{SL}_n(\CC))$ for closed surfaces.

The main result of this appendix is the following: 
\begin{theorem}
\label{Theorem:AnosovDynamics} Let $\Sigma$ be a compact orientable surface, possibly with boundary, connected and with negative Euler characteristic, and let $\phi\colon \Sigma\to\Sigma$
be a pseudo-Anosov diffeomorphism.

Then the character $[\rho_n]$ is a hyperbolic fixed point of $\phi^*$, i.e.~the tangent map 
\[
(d\phi^*)_{[\rho_n]}\colon T_{[\rho_n]} \mathcal{R}(\Sigma,\partial \Sigma,\mathrm{SL}_n(\CC))\to
T_{[\rho_n]}\mathcal{R}(\Sigma,\partial \Sigma,\mathrm{SL}_n(\CC))
\]
has no eigenvalues of modulus one. 
\end{theorem}

The proof is based on the cohomological interpretation of the tangent spaces to varieties of characters.
Recall that by a result of A.~Weil \cite{Weil} there is a natural isomorphism
$$
T_{[\rho_n]} \mathcal{R}(\Sigma,\mathrm{SL}_n(\CC))\cong H^1(\Sigma, \operatorname{Ad}\circ\rho_n)
$$
and the tangent space of the relative character variety can be interpreted as a kernel
(see for instance  \cite[Proposition 18]{HPvol}):
$$
T_{[\rho_n]} \mathcal{R}(\Sigma,\partial \Sigma,\mathrm{SL}_n(\CC))\cong \ker( H^1(\Sigma, \operatorname{Ad}\circ\rho_n)\to  H^1(\partial \Sigma, \operatorname{Ad}\circ\rho_n) ).
$$
Moreover, the tangent map $d\phi^*_{[\rho_n]}$ corresponds to the induced map $\phi^*$ in cohomology.

 There is a natural surjection 
$$\alpha\colon\pi_1( M(\phi)) \twoheadrightarrow \ZZ,\quad 
\text{ $\alpha(\tau) = 1$, and $\alpha(\gamma) =0$ for $\gamma\in\pi_1(\Sigma)$,} $$
induced by the fibration 
$
\Sigma\to M(\phi) \to S^1 
$.

\begin{remark}
In Section~\ref{subsec:twistedAlex} we have defined the twisted Alexander polynomial for manifolds with cusps, but the same definition applies to closed manifolds,
without requiring
any assumption on $\alpha$ (just non-triviality). Thus we can define a one variable twisted Alexander polynomial  $ \Delta^{\alpha,n}_{M(\phi)}(t)$ even for $M(\phi)$ closed.
The main results of the paper hold true and are simpler to prove in the closed case, without needing the approximation results of Sections~3 and~6, for instance.
 \end{remark}

\begin{proposition} Up to multiplication by a factor $\pm t^m$:
 \label{prop:deltchar}
 $$\det( (d\phi^*)_{[\rho_n]}-t\operatorname{Id})
 =\prod\limits_{k=1}^{n-1}  \Delta^{\alpha,2 k+1}_{M(\phi)}(t)
 .$$
\end{proposition}

\begin{proof}[Proof of Theorem~\ref{Theorem:AnosovDynamics}]
Apply Proposition~\ref{prop:deltchar} and Theorem~\ref{Thm:AlexNonZero}. 
\end{proof}

\begin{remark} 
\label{rem:sympchar}
There is a natural $\CC$-valued symplectic form on  $\mathcal{R}(\Sigma,\partial\Sigma, \mathrm{SL}_n(\CC))$ 
\cite{LawtonTransactions,  HPvol}. By naturality, this symplectic form is $\phi^*$-invariant,
therefore:
$$
\det( (d\phi^*)^{-1}_{[\rho_n]}-t\operatorname{Id}) = \det( (d\phi^*)_{[\rho_n]}-t\operatorname{Id}).
$$
\end{remark}

Proposition~\ref{prop:deltchar} is based on the following lemma:

\begin{lemma}
\label{Lemma:torMappingTorus}
\begin{enumerate}[(a)]
 \item When $\Sigma$ is closed,
$$
\tor(M(\phi), \alphaa\otimes\operatorname{Ad}\circ\rho_n)= \det( (d\phi^*)_{[\rho_n]}-t\operatorname{Id})^{-1}.
$$
\item
When $\partial\Sigma \neq\emptyset$, if $\sigma_\phi$ denotes the permutation matrix on the components of  $\partial\Sigma$, then
$$
\tor(M(\phi), \alphaa\otimes\operatorname{Ad}\circ\rho_n)= \det( (d\phi^*)_{[\rho_n]}-t\operatorname{Id})^{-1} \det( \sigma_\phi-t \operatorname{Id})^{1-n} .
$$
\end{enumerate}
\end{lemma}

\begin{remark}
\label{rem:perm_matrix}
As $\partial \Sigma$ has $s$ components, $\sigma_\phi$ is a permutation matrix of size $s\times s$, that decomposes into $l$ cycles, where $l$
is the number of components of $\partial M(\phi)=T^2_1\sqcup\cdots\sqcup T^2_l$. Furthermore, if $c_i$ is the order of the cycle corresponding to  $T^2_i$, the $i$-th
component of $\partial M(\phi)$, then  $\alpha(\pi_1(T^2_i)) = c_i\ZZ$, $s=c_1+\cdots +c_l$, and 
$$
\det( \sigma_\phi-t \operatorname{Id})= (-1)^{s-l}\prod_{i=1}^l (1-t^{c_i}).
$$
\end{remark}

\begin{proof}[Proof of Lemma~\ref{Lemma:torMappingTorus}]
Let $K$ be CW-complex with underlying space $\vert K\vert=M(\phi)$. Consider its lift $\overline K$ to $\Sigma\times \RR$ 
and its lift $\widetilde K$ to the universal covering. We work with the following chain complexes: 
\begin{align*}
C^*(K,\alphaa\otimes \mathrm{Ad}\circ\rho_n) & = \hom_{\pi_1(M(\phi))} (C_*(\widetilde K, \ZZ), \CC(t)\otimes \mathfrak{sl}_n(\CC)), \\
C^*(\overline K,\alphaa\otimes \mathrm{Ad}\circ\rho_n) & = \hom_{\pi_1(\Sigma)} (C_*(\widetilde K, \ZZ), \CC(t)\otimes \mathfrak{sl}_n(\CC))  .
\end{align*}
Choosing  $\tau$, a representative in $\pi_1(M(\phi))$ of a generator of $\ZZ$, it  acts on $C^*(\overline K,\alphaa\otimes \mathrm{Ad}\circ\rho)$ by
$$
\theta\mapsto \alpha(\tau) \mathrm{Ad}_{\rho_n(\tau)} \circ \theta\circ\tau^{-1},\qquad \forall \theta\in C^*(\overline K,\alphaa\otimes \mathrm{Ad}\circ\rho_n).
$$
The action does not depend on the choice of the representative $\tau\in \pi_1(M(\phi))$ of the generator of $\ZZ$.
We have then a short exact sequence of complexes:
$$
0\to C^*(K,\alphaa\otimes \mathrm{Ad}\circ\rho_n) \to 
C^*(\overline K,\alphaa\otimes \mathrm{Ad}\circ\rho_n) \overset{(\tau-1)} \longrightarrow
C^*(\overline K,\alphaa\otimes \mathrm{Ad}\circ\rho_n) \to 0.
$$
It induces Wang's exact sequence in cohomology. Hence, as $H^i( \overline K,\alphaa\otimes \mathrm{Ad}\circ\rho_n)
\cong  H^i( \Sigma,\alphaa\otimes \mathrm{Ad}\circ\rho_n) $ 
vanishes for every $i\neq 1$, by Milnor's theorem on torsion of exact sequences  \cite[Theorem~3.2]{Mil66}:
\begin{equation}
 \label{eqn:torchartau}
\tor(M(\phi), \alphaa\otimes\operatorname{Ad}\circ\rho_n)= \det ( \tau^{*,1}-\operatorname{Id})^{-1},
 \end{equation}
where $\tau^{*,1}$ is the morphism  on $H^1(\Sigma , \alphaa\otimes \mathrm{Ad}\circ\rho_n)$ induced by $\tau$.
As $\alpha$ is trivial on $\pi_1(\Sigma)$:
$$ 
H^1(\Sigma, \alphaa\otimes \mathrm{Ad}\circ\rho_n)\cong \CC(t)\otimes H^1(\Sigma, \mathrm{Ad}\circ\rho_n)
$$
and  $\tau^{*,1}$ corresponds via this isomorphism to $t\otimes \phi^{*,1}$,
where $\phi^{*,1}$ is the action that $\phi$ induces on $H^1(\Sigma, \mathrm{Ad}\circ\rho_n)$.
Hence 
\begin{equation}
  \label{eqn:chartauphi}
\det ( \tau^{*,1}-\operatorname{Id}) = \det ( t \,\phi^{*,1}-\operatorname{Id})=
t^{n} \det (  \phi^{*,1}- t^{-1}\operatorname{Id}).  
\end{equation}
From \eqref{eqn:torchartau} and \eqref{eqn:chartauphi} we get
$$
\tor(M(\phi), \alphaa\otimes\operatorname{Ad}\circ\rho_n)=  \det ( \phi^{*,1}-t^{-1} \operatorname{Id})^{-1},
$$
up to a factor $\pm t ^m$.
Next we look at the action of $ (\phi^{*,1})^{\pm 1}$ to the  exact sequence coming from the cohomology of the pair $(\Sigma,\partial\Sigma)$
\begin{equation}
 \label{eqn:exact}
0\to \ker(i^*)\to  H^1(\Sigma, \operatorname{Ad}\circ\rho_n)\overset {i^*}\to  H^1(\partial \Sigma, \operatorname{Ad}\circ\rho_n) )\to 0. 
\end{equation}
We claim that:
\begin{enumerate}
 \item  The action of $ (\phi^{*,1})^{-1}$  on $\ker(i^*)$ corresponds to the action of $d\phi^*$ by the 
 isomorphism $\ker(i^*)\cong  T_{[\rho_n]} \mathcal{R}(\Sigma,\partial \Sigma,\mathrm{SL}_n(\CC))$.
 \item  The action of $ (\phi^{*,1})$  on  $H^1(\partial \Sigma, \operatorname{Ad}\circ\rho_n) )$ is equivalent to $\sigma_{\phi}\times\overset{(n-1)}\cdots \times\sigma_\phi$.
\end{enumerate}
The proof of these claims and the product formula of torsions applied to the exact sequence \eqref{eqn:exact} will complete the proof of Lemma~\ref{Lemma:torMappingTorus}.

\smallskip

\emph{Proof of Claim (1)}. We consider the action on the variety of representations
$$
\rho\mapsto \operatorname{Ad}_{\rho_n(\tau^{-1})}\circ\rho\circ\phi_{\#}
$$
so that  $\rho_n$ is fixed and induces the previous action on the space of conjugacy classes of representations. 
Next we recall Weil's isomorphism between  the tangent space and group cohomology:
\begin{equation}
 \label{eqn:Weiliso}
\begin{array}{rcl}
T_{[\rho_n]} \mathcal{R}(\Sigma,\mathrm{SL}_n(\CC)) &\cong &H^1(\pi_1(\Sigma), \operatorname{Ad}\circ\rho_n)\\
\frac{d\phantom{t}}{dt} \rho_t \vert_{t=0}& \mapsto &  \left[ \gamma\mapsto   \frac{d\phantom{t}}{dt}\rho_t(\gamma)\rho_0(\gamma^{-1}) \vert_{t=0} \right]
\end{array} 
\end{equation}
where $\rho_t$ is a path  of representations, smooth on $t\in (-\varepsilon,\varepsilon)$ and with $\rho_0=\rho_n$. 
Here $H^1(\pi_1(\Sigma), \operatorname{Ad}\circ\rho_n)=Z^1/B^1$ where  $Z^1$ is the space of cocycles or crossed morphisms,
i.e.~maps $\theta\colon\pi_1(\Sigma)\to \mathfrak{sl}_n(\CC)$ that satisfy 
$\theta(\gamma_1\gamma_2)=\theta(\gamma_1)+\Ad_{\rho_n(\gamma_1)}(\theta(\gamma_2))$ $\forall\gamma_1,\gamma_2\in\pi_1(\Sigma)$,
and $B^1$ is the space of inner cocycles, i.e.~maps $\gamma\mapsto a-\Ad_{\rho_n(\gamma)}(a)$ for some $a\in \mathfrak{sl}_n(\CC)$.
Using Weil's isomorphism \eqref{eqn:Weiliso}, the action of $ d\phi^* $ on $H^1(\pi_1(\Sigma), \operatorname{Ad}\circ\rho_n)$ using classes of cocycles is 
$$
[\theta]\mapsto [\operatorname{Ad}_{\rho_n(\tau^{-1})}\circ\theta\circ\phi_{\#}] \qquad \forall \theta\in Z^1. 
$$
The claim follows using the isomorphism between group cohomology of $\pi_1(\Sigma)$ and cohomology of the surface $\Sigma$, which is aspherical,
and the naturality of the actions induced by $\phi$ on each cohomology group. This establishes the first claim.

\smallskip
\emph{Proof of Claim (2)}. We use the isomorphism:
\begin{equation}
\label{eqn:h1sum} 
H^1(\partial\Sigma, \operatorname{Ad}\circ\rho_n) =\bigoplus_{i=1}^r H^1(\partial_i, \operatorname{Ad}\circ\rho_n )
\cong \bigoplus_{i=1}^r H_0(\partial_i, \operatorname{Ad}\circ\rho_n ).
\end{equation}
Namely, we decompose $\partial\Sigma$ along connected components and use Poincar\'e duality on each of the circles. Next, we use the canonical isomorphism 
\begin{equation}
 \label{eqn:coinvariants}
 H_0(\partial_i, \operatorname{Ad}\circ\rho_n )  \cong \mathfrak{sl}_n(\CC)_{\operatorname{Ad}(\rho_n(\partial_i))}\cong \CC^{n-1} 
\end{equation}
where 
$\mathfrak{sl}_n(\CC)_{\operatorname{Ad}(\rho_n(\partial_i))} = \mathfrak{sl}_n(\CC) \otimes_{\pi_1(\partial_i)}\ZZ $ 
denotes the space of coinvariants. It is isomorphic to the quotient of 
$\mathfrak{sl}_n(\CC)$ by the image of $(\operatorname{Ad}(\rho_n(\partial_i))-\operatorname{Id})$
\cite{Bro}, and its dimension $n-1$ has been computed for instance in \cite{MP10}.

Next we aim to understand the action of $\tau$ on these spaces.  We view the chain complex for computing the homology of each $\partial _i$ as a subcomplex
of
$$
C_*(\overline K, \Ad\rho_n)=\mathfrak{sl}_n(\CC)\otimes_{\pi_1(\Sigma)} C_*(\widetilde K,\ZZ),
$$ 
the complex that computes the homology of $\Sigma\times \RR$,
where $K$, $\overline K$ and $\widetilde K$ are as in the beginning of the proof of the lemma.
Here every $\gamma\in \pi_1(\Sigma)$ acts on $C_*(\widetilde K,\ZZ)$ by deck transformations and 
on $\mathfrak{sl}_n(\CC)$ by $\Ad(\rho_n(\gamma^{-1}))$. The action of $\tau$  maps the  chain $m\otimes c$ 
(for  $m\in\mathfrak{sl}_n(\CC)$ and $ c\in C_*(\widetilde K,\ZZ)$) to
$$
m\otimes c\mapsto m\cdot\tau^{-1}\otimes \tau\, c= \Ad(\tau)(m) \otimes \tau\, c
$$
(see \cite{Porti97}).

Each component $\partial_ i$ lifts to a union of lines in the universal covering $\widetilde \Sigma$ whose stabilizer by the action of $\pi_1(\Sigma)$ is precisely a representative
of $\pi_1(\partial_i)$ in the conjugacy class. Then choosing $\tilde e_0$ a 0-cell of $\widetilde K$ that projects to $\partial_i$, the canonical isomorphism 
\eqref{eqn:coinvariants}
is induced by the projection 
$$
\mathfrak{sl}_n (\CC)\otimes \tilde e_0  \to  \mathfrak{sl}_n (\CC)\otimes _{\pi_1(\partial_i)}\ZZ = \mathfrak{sl}_n(\CC)_{\operatorname{Ad}(\rho_n(\partial_i))},
$$
where the choice of $\pi_1(\partial_i)$ corresponds to the stabilizer of the line in  $\widetilde K$ that contains $\tilde e_0$.
A different choice of $\tilde e_0$ would be $\gamma\, \tilde e_0$ for $\gamma\in\pi_1(\Sigma)$, then the subgroup $\pi_1(\partial_i)$ should be replaced by 
$\gamma\pi_1(\partial_i)\gamma^{-1}$. This leads to the natural isomorphism between coinvariant subspaces
$$
\begin{array}{rcl}
 \mathfrak{sl}_n(\CC)_{\operatorname{Ad}(\rho_n(\partial_i))} & \cong &   \mathfrak{sl}_n(\CC)_{\operatorname{Ad}(\rho_n(\gamma \partial_i\gamma^{-1}))}\\
 m & \mapsto & \Ad_{\rho_n(\gamma)} (m)
\end{array}
$$
Furthermore, the action of $\phi^{*,1}$ on $H^1(\Sigma,\Ad\circ\rho_n )$ corresponds via \eqref{eqn:h1sum} and~\eqref{eqn:coinvariants} to
$$
\begin{array}{rcl}
 \mathfrak{sl}_n(\CC)_{\operatorname{Ad}(\rho_n(\partial_i))} & \to &   \mathfrak{sl}_n(\CC)_{\operatorname{Ad}(\rho_n(\phi_{\#}(\partial_i)))}\\
 m & \mapsto & \Ad_{\rho_n(\tau)} (m)
\end{array}
$$

If $\phi$ 
defines a cycle of order $c_i>0$ on the component $\partial_i$ (in particular $\phi^{c_i}(\partial_i)=\partial_i$), then, viewing $\partial_i$ as an element of the fundamental group,
 there exists   $\gamma\in\pi_1(\Sigma)$ that conjugates  $\phi^{c_i}_{\#}(\partial_i)=\tau^{c_i}\partial_i \tau^{-c_i}$ and $\partial_i$.
Namely $\gamma\tau^{c_i}$ commutes with $\partial_i$ in $\pi_1(M(\phi))$, 
in fact both $\gamma\tau^{c_i}$ and  $\partial_i$ belong to the same peripheral subgroup $\pi_1(T^2_i)$ of $\pi_1(M(\phi))$. 
Therefore 
$$
\mathfrak{sl}_n(\CC)_{\operatorname{Ad}(\rho_n(\partial_i))}=\mathfrak{sl}_n(\CC)_{\operatorname{Ad}(\rho_n(\langle\partial_i, \gamma\tau^{c_i}\rangle))}=
\mathfrak{sl}_n(\CC)_{\operatorname{Ad}(\rho_n(   \pi_1(T^2_i) ))}
$$
and
$\Ad(\rho_n( \gamma\tau^{c_i} ))$ acts as the identity on  
$\mathfrak{sl}_n(\CC)_{\operatorname{Ad}(\rho_n(\partial_i))}$, by Lemma~\ref{lemma:trivialaction} below. 
In other words, $\phi^{c_i}$ acts trivially on $H^1(\partial_i, \Ad\circ\rho_n)\cong\CC^{n-1}$, and the 
second claim  follows.
\end{proof}

\begin{lemma}
\label{lemma:trivialaction}
For any peripheral torus $T^2_i$ and 
 for any nontrivial $\gamma\in\pi_1(T^2_i)$, $\Ad(\rho_n(\gamma))$ acts trivially on  the space of coinvariants 
 $\mathfrak{sl}_n(\CC)_{\operatorname{Ad}(\rho_n(\pi_1(T^2_i)))}$.
\end{lemma}

\begin{proof}
%
By construction, the action of  $\Ad(\rho_n(\gamma))$ on the space of invariants
$$
\mathfrak{sl}_n(\CC)^{\operatorname{Ad}(\rho_n(\pi_1(T^2_i)))}=\{  m\in\mathfrak{sl}_n(\CC) \mid \operatorname{Ad}_{\rho_n(g )}(m)=m,\   \forall g \in\pi_1(T^2_i)\}
$$ 
is trivial.  As the $\CC$-valued Killing form on  $\mathfrak{sl}_n(\CC)$ is $\Ad$-invariant, 
the space of coinvariants is the quotient of $\mathfrak{sl}_n(\CC)$ by the orthogonal of the invariants. Namely
$$
\left(  \mathfrak{sl}_n(\CC)^{\operatorname{Ad}(\rho_n(\pi_1(T^2_i)))}  \right) ^{\perp}=
\{ m-\operatorname{Ad}_{\rho_n(g)}(m)\mid  m\in\ \mathfrak{sl}_n(\CC), \ \forall g\in \pi_1(T^2_i) \}
$$
and
$$
\mathfrak{sl}_n(\CC)_{\operatorname{Ad}(\rho_n(\pi_1(T^2_i)))}  = \mathfrak{sl}_n(\CC)/ 
\left(  \mathfrak{sl}_n(\CC)^{\operatorname{Ad}(\rho_n(\pi_1(T^2_i)))}  \right) ^{\perp} .
$$
Therefore, the Killing form   induces  a 
 pairing 
$$
\mathfrak{sl}_n(\CC)^{\operatorname{Ad}(\rho_n(\pi_1(T^2_i)))}\times  \mathfrak{sl}_n(\CC)_{\operatorname{Ad}(\rho_n(\pi_1(T^2_i)))} \to \CC.
$$ 
As the Killing form is non-degenrate and $\Ad$-invariant, this pairing is perfect and $\Ad(\rho_n(\gamma))$-invariant.
 In particular, as
the action of  $\Ad(\rho_n(\gamma))$ is trivial on the space of invariants, it is also trivial on the space of coinvariants.
\end{proof}

\begin{proof}[Proof of Proposition~\ref{prop:deltchar}]
When $\partial \Sigma = \emptyset$ this is a direct consequence of the Clebsch-Gordan formula \eqref{eqn:Clebsch-Gordan} and of Lemma~\ref{Lemma:torMappingTorus}.

When $\partial \Sigma\neq\emptyset$, first notice that $\alpha$ always satisfies Assumption~\ref{Assumption:alpha}. 
Furthermore, we do not need to care about Assumption~\ref{Assumption:Lift} because 
$\Ad\circ \rho_n$ decomposes by \eqref{eqn:Clebsch-Gordan} in a sum of odd dimensional representations and no spin structure is involved in the computation of the torsion in this case 
(however we could chose a lift of the holonomy satisfying this assumption).

Finally, we need to discuss the term $\det(\sigma_\phi-t\operatorname{Id})^ {n-1}$.
Following Remark~\ref{rem:perm_matrix}, we decompose 
$\sigma_\phi$ into $l$ disjoint cycles of order $c_1,\ldots, c_l$, respectively, with $c_1+\cdots+ c_l=s$.
Therefore 
$$
\det(\sigma_{\phi}- t\operatorname{Id})=  \pm \prod_{i=1}^l (t^{c_i}-1)\,.
$$ 
On the other hand, 
each cycle corresponds to a peripheral torus $T^2_i$ of $ M(\phi)$ and
 $\alpha(\pi_1(T^2_i))= c_i\ZZ$, for $i=1,\ldots, l$. 
Thus the factor that appears in the definition of
$\Delta^{\alpha,2 k+1}_{M(\phi)}(t)$ (Definition~\ref{Def:twisted}) is also 
 $ \prod_{i=1}^l (t^{c_i}-1)$.
  Finally the exponent $n-1$ is the number of factors in the Clebsch-Gordan formula.
 \end{proof}